\documentclass[11pt,english]{article}
\usepackage[T1]{fontenc}
\usepackage[latin9]{inputenc}
\usepackage{geometry}
\usepackage{amsthm}
\usepackage{amsmath}
\usepackage{amssymb}
\usepackage{esint}

\makeatletter
\numberwithin{equation}{section}
\theoremstyle{plain}
\newtheorem{thm}{\protect\theoremname}[section]
  \theoremstyle{definition}
  \newtheorem{defn}[thm]{\protect\definitionname}
  \theoremstyle{remark}
  \newtheorem{rem}[thm]{\protect\remarkname}
  \theoremstyle{plain}
  \newtheorem{prop}[thm]{\protect\propositionname}
  \theoremstyle{plain}
  \newtheorem{lem}[thm]{\protect\lemmaname}
  \theoremstyle{definition}
  \newtheorem{example}[thm]{\protect\examplename}

\makeatother

\usepackage{babel}
  \providecommand{\definitionname}{Definition}
  \providecommand{\examplename}{Example}
  \providecommand{\lemmaname}{Lemma}
  \providecommand{\propositionname}{Proposition}
  \providecommand{\remarkname}{Remark}
\providecommand{\theoremname}{Theorem}

\begin{document}

\title{$G$-Brownian Motion as Rough Paths and Differential Equations Driven
by $G$-Brownian Motion}

\author{X. Geng%
\thanks{Mathematical Institute, University of Oxford, Oxford OX1 3LB, England
and the Oxford-Man Institute, University of Oxford, Oxford OX2 6ED,
England. Email: xi.geng@maths.ox.ac.uk. %
}, Z. Qian%
\thanks{Exeter College, University of Oxford, Oxford OX1 3DP, England. Email:
qianz@maths.ox.ac.uk.%
} and D. Yang%
\thanks{Mathematical Institute, University of Oxford, Oxford OX1 3LB, England
and the Oxford-Man Institute, University of Oxford, Oxford OX2 6ED,
England. Email: danyu.yang@maths.ox.ac.uk. %
}}
\maketitle
\begin{abstract}
The present paper is devoted to the study of sample paths of $G$-Brownian
motion and stochastic differential equations (SDEs) driven by $G$-Brownian
motion from the view of rough path theory. As the starting point,
we show that quasi-surely, sample paths of $G$-Brownian motion can
be enhanced to the second level in a canonical way so that they become
geometric rough paths of roughness $2<p<3.$ This result enables us
to introduce the notion of rough differential equations (RDEs) driven
by $G$-Brownian motion in the pathwise sense under the general framework
of rough paths. Next we establish the fundamental relation between
SDEs and RDEs driven by $G$-Brownian motion. As an application, we
introduce the notion of SDEs on a differentiable manifold driven by
$G$-Brownian motion and construct solutions from the RDE point of
view by using pathwise localization technique. This is the starting
point of introducing $G$-Brownian motion on a Riemannian manifold,
based on the idea of Eells-Elworthy-Malliavin. The last part of this
paper is devoted to such construction for a wide and interesting class
of $G$-functions whose invariant group is the orthogonal group. We
also develop the Euler-Maruyama approximation for SDEs driven by $G$-Brownian
motion of independent interest.
\end{abstract}

\section{Introduction}

The classical Feynman-Kac formula (see \cite{kac1949distributions},
\cite{karatzas1991brownian}) provides us with a way to represent
the solution of a linear parabolic PDE in terms of the conditional
expectation of certain functional of a diffusion process (solution
of an SDE). However, it works only for the linear case, which is mainly
due to the linearity nature of diffusion processes. To understand
nonlinear parabolic PDEs from the probabilistic point of view, Peng
and Pardoux (see \cite{pardoux1990adapted}, \cite{pardoux1992backward},
\cite{pardoux1994backward}) initiated the study of backward stochastic
differential equations (BSDEs) and showed that the solution of a certain
type of quasilinear parabolic PDEs can be expressed in terms of the
solution of BSDE . This result suggests that BSDE reveals a certain
type of nonlinear dynamics, and was made explicit by Peng \cite{peng1997backward}.
More precisely, Peng introduced a notion of nonlinear expectation
called the $g$-expectation in terms of the solution of BSDE which
is filtration consistent. However, it was developed under the framework
of classical It$\hat{\mbox{o}}$ calculus and did not capture the
fully nonlinear situation.

Motivated from the study of fully nonlinear dynamics, Peng \cite{peng2007g}
introduced the notion of $G$-expectation in an intrinsic way which
does not rely on any particular probability space. It reveals the
probability distribution uncertainty in a fundamental way which is
crucial in many situations such as modeling risk uncertainty in mathematical
finance. The underlying mechanism corresponding to such kind of uncertainty
is a fully nonlinear parabolic PDE. In \cite{peng2007g}, \cite{peng2008multi}
he also introduced the concept of $G$-Brownian motion which is generated
by the so-called nonlinear $G$-heat equation and related stochastic
calculus such as $G$-It$\hat{\mbox{o}}$ integral, $G$-It$\hat{\mbox{o}}$
formula, SDEs driven by $G$-Brownian motion, etc. One of the major
significance of such theory is the corresponding nonlinear Feynman-Kac
formula proved by Peng \cite{peng2010nonlinear}, which gives us a
way to represent the solution of a fully nonlinear parabolic PDE via
the solution of a forward-backward SDE under the framework of $G$-expectation. 

On the other hand, motivated from the study of integration against
irregular paths and differential equations driven by rough signals,
Lyons \cite{lyons1998differential} proposed a theory of rough paths
which reveals the fundamental way of understanding the roughness of
a continuous path. He pointed out that to understand the evolution
of a system whose input signal (driven path) is rough, a finite sequence
of ``iterated integrals'' (higher levels) of the driving path which
satisfy a certain type of algebraic relation (Chen identity) should
be specified in advance. Such point of view is fundamental, if we
look at the Taylor expansion for the solution of an ODE whose driving
path is of bounded variation (see (\ref{Euler scheme}) and a more
detailed introduction in the next section). In other words, it is
essential to regard a path as an object valued in some tensor algebra
which records the information of higher levels if we wish to understand
the ``differential'' of the path. Moreover, Lyons \cite{lyons1998differential}
proved the so-called universal limit theorem (see Theorem \ref{universal limit theorem}
in the next section), which allows us to introduce the notion of differential
equations driven by rough paths (simply called RDEs) in a rigorous
way. The theory of rough paths has significant applications in classical
stochastic analysis, as we can prove that the sample paths of many
stochastic processes we've encountered are essentially rough paths
with certain roughness. According to Lyons' universal limit theorem,
we are able to establish RDEs driven by the sample paths of those
stochastic processes in a pathwise manner. It provides us with a new
way to understand SDEs, especially when the driving process is not
the classical Brownian motion in which case a well-developed It$\hat{\mbox{o}}$
SDE theory is still not available.

The case of classical Brownian motion is quite special, since we have
a complete SDE theory in the $L^{2}$-sense, as well as the notion
of Stratonovich type integrals and differential equations. The fundamental
relation between the two types of stochastic differentials (one-dimensional
case) can be expressed by 
\[
X\circ dY=XdY+\frac{1}{2}dX\cdot dY.
\]
It is proved in the rough path theory (see \cite{friz2010multidimensional},
\cite{lyons2002system}, and also \cite{ikeda1989stochastic}, \cite{wong1965relation}
from the view of Wong-Zakai type approximation) that the Stratonovich
type integrals and differential equations are equivalent to the pathwise
integrals and RDEs in the sense of rough paths. In other words, the
following to types of differential equations driven by Brownian motion
\begin{eqnarray*}
dX_{t} & = & \sum_{\alpha=1}^{d}V_{\alpha}(X_{t})dW_{t}^{\alpha}+b(X_{t})dt,\ \ \ \mbox{(It\ensuremath{\hat{\mbox{o}}}\ type SDE)}\\
dY_{t} & = & \sum_{\alpha=1}^{d}V_{\alpha}(Y_{t})dW_{t}^{\alpha}+(b(Y_{t})-\sum_{\alpha=1}^{d}\frac{1}{2}DV_{\alpha}(Y_{t})\cdot V_{\alpha}(Y_{t}))dt,\ \ \ \mbox{(RDE)}
\end{eqnarray*}
which are both well-defined under some regularity assumptions on the
generating vector fields, are equivalent in the sense that if their
solutions $X_{t}$ and $Y_{t}$ satisfy $X_{0}=Y_{0}$, then $X=Y$
almost surely.

Under the framework of $G$-expectation, SDEs driven by $G$-Brownian
motion introduced by Peng , can be regarded as nonlinear diffusion
processes in Euclidean spaces. The idea of constructing $G$-It$\hat{\mbox{o}}$
integrals and SDEs driven by $G$-Brownian motion is similar to the
classical It$\hat{\mbox{o}}$ calculus, which is also an $L^{2}$-theory
but under the $G$-expectation instead of probability measures. What
is missing is the notion of Stratonovich type integrals, mainly due
to the reason that the theory of $G$-martingales is still not well
understood. In particular, we don't have the corresponding nonlinear
Doob-Meyer type decomposition theorem and the notion of quadratic
variation processes for $G$-martingales. However, by the key observation
in the classical case that the Stratonovich type integrals and the
pathwise integrals are essentially equivalent in the sense of rough
paths, we can study the sample paths of $G$-Brownian motion and SDEs
driven by $G$-Brownian motion from the view of rough path theory,
once we prove that the sample paths of $G$-Brownian motion can be
regarded as objects in some rough path space with certain roughness.
This is in fact what the present paper is mainly focused on. The basic
language to describe path structure under the $G$-expectation is
quasi-sure analysis and capacity theory, which was developed by Denis,
Hu and Peng \cite{denis2011function}. They generalized the Kolmogorov
continuity theorem and studied sample path properties of $G$-Brownian
motion. In particular, they also studied the relation between $G$-expectation
and upper expectation associated to a family of probability measures
which defines a Choquet capacity and the relation between the corresponding
two types of $L^{p}$-spaces. The pathwise properties and homeomorphic
flows for SDEs driven by $G$-Brownian motion in the quasi-sure setting
was studied by Gao \cite{gao2009pathwise}.

There are two main goals of the present paper. This first one is to
study the rough path nature of sample paths of $G$-Brownian motion
so that we can define RDEs driven by $G$-Brownian motion (the Stratonovich
counterpart in the classical case) in the pathwise sense, and establish
the fundamental relation between two types of differential equations
driven by $G$-Brownian motion. The second one is to understand nonlinear
diffusion processes in a (Riemannian) geometric setting, from the
view of paths and distributions (the generating nonlinear PDE).

The present paper is organized in the following way. Section 2 is
a basic review of the theory of $G$-expectation and rough paths,
which provides us with the general framework and basic tools for our
study. In Section 3 we study the Euler-Maruyama approximation scheme
for SDEs driven by $G$-Brownian motion. In Section 4 we show that
for quasi-surely, the sample paths of $G$-Brownian motion can be
enhanced to the second level in a canonical way so that they become
geometric rough paths of roughness $2<p<3$ by using techniques in
rough path theory. In Section 5 we establish the fundamental relation
between SDEs and RDEs driven by $G$-Brownian motion by using rough
Taylor expansions. In section 6 we introduce the notion of SDEs on
a differentiable manifold driven by $G$-Brownian motion from the
RDE point of view. In the last section, we study the infinitesimal
diffusive nature and the generating PDEs of nonlinear diffusion processes
in a (Riemannian) geometric setting, which leads to the construction
of $G$-Brownian motion on a Riemannian manifold. We restrict ourselves
to compact manifolds only, although the general case can be treated
in a similar way with more technical complexity.

Throughout the rest of this paper, we will use standard geometric
notation for differential equations. Moreover, we will use the Einstein
convention of summation, that is, when an index $\alpha$ appears
as both subscript and superscript in the same expression, summation
over $\alpha$ is taken automatically.

\section{Preliminaries on $G$-expectation and Rough Path Theory}

\subsection{$G$-expectation and Related Stochastic Calculus}

We first introduce some fundamentals on $G$-expectation and related
stochastic calculus. For a systematic introduction, see \cite{peng2007g},
\cite{peng2008multi}, \cite{peng2010nonlinear}.

Let $\Omega$ be a nonempty set, and $\mathcal{H}$ be a vector space
of functionals on $\Omega$ such that $\mathcal{H}$ contains all
constant functionals and for any $ $$X_{1},\cdots,X_{n}\in\mathcal{H}$
and any $\varphi\in C_{l,Lip}(\mathbb{R}^{n})$, 
\[
\varphi(X_{1},\cdots,X_{n})\in\mathcal{H},
\]
where $C_{l,Lip}(\mathbb{R}^{n})$ denotes the space of functions
$\varphi$ on $\mathbb{R}^{n}$ satisfying
\[
|\varphi(x)-\varphi(y)|\leqslant C(1+|x|^{m}+|y|^{m})(|x-y|),\ \forall x,y\in\mathbb{R}^{n},
\]
for some constant $C>0$ and $m\in\mathbb{N}$ depending on $\varphi$.
$\mathcal{H}$ can be regarded as the space of random variables.
\begin{defn}
A sublinear expectation $\mathbb{E}$ on $(\Omega,\mathcal{H})$ is
a functional $\mathbb{E}:\ \mathcal{H}\rightarrow\mathbb{R}$ such
that 

(1) if $X\leqslant Y,$ then $\mathbb{E}[X]\leqslant\mathbb{E}[Y]$; 

(2) for any constant $c$, $\mathbb{E}[c]=c$;

(3) for any $X,Y\in\mathcal{H},$ $\mathbb{E}[X+Y]\leqslant\mathbb{E}[X]+\mathbb{E}[Y];$

(4) for any $\lambda\geqslant0$ and $X\in\mathcal{H},$ $\mathbb{E}[\lambda X]=\lambda\mathbb{E}[X].$
\end{defn}
The triple $(\Omega,\mathcal{H},\mathbb{E})$ is called a sublinear
expectation space.

The relation between sublinear expectations and linear expectations,
which was proved by Peng \cite{peng2010nonlinear}, is contained in
the following representation theorem.
\begin{thm}
\label{representation of E}Let $(\Omega,\mathcal{H},\mathbb{E})$
be a sublinear expectation space. Then there exists a family of linear
expectations (linear functionals) $\{\mathbb{E}_{\theta}:\theta\in\Theta\}$
on $\mathcal{H},$ such that 
\[
\mathbb{E}[X]=\sup_{\theta\in\Theta}\mathbb{E}_{\theta}[X],\ \forall X\in\mathcal{H}.
\]

\end{thm}

Under the frame work of sublinear expectation space, we also have
the notion of independence and distribution (law).
\begin{defn}
(1) A random vector $Y\in\mathcal{H}^{n}$ is said to be independent
from another random vector $X\in\mathcal{H}^{m}$ under the sublinear
expectation $\mathbb{E},$ if for any $\varphi\in C_{l,Lip}(\mathbb{R}^{m}\times\mathbb{R}^{n})$,
\[
\mathbb{E}[\varphi(X,Y)]=\mathbb{E}[\mathbb{E}[\varphi(x,Y)]_{x=X}].
\]

(2) Given a random vector $X\in\mathcal{H}^{n}$, the distribution
(or the law) of $X$ is defined as the sublinear expectation 
\[
\mathbb{F}_{X}[\varphi]:=\mathbb{E}[\varphi(X)],\ \varphi\in C_{l,Lip}(\mathbb{R}^{n}),
\]
on $(\mathbb{R}^{n},C_{l,Lip}(\mathbb{R}^{n}))$. By saying that two
random vectors $X,Y$ (possibly defined on different sublinear expectation
spaces) are identically distributed, we mean that their distributions
are the same. 
\end{defn}

Now we introduce the notion of $G$-distribution, which is the generalization
of degenerate distributions and normal distributions. It captures
the uncertainty of probability distributions and plays a fundamental
role in the theory of sublinear expectation. 

Let $S(d)$ be the space of $d\times d$ symmetric matrices, and let
$G:\ \mathbb{R}^{d}\times S(d)\rightarrow\mathbb{R}$ be a continuous
and sublinear function monotonic in $S(d)$ in the sense that:

(1) $G(p+\bar{p},A+\bar{A})\leqslant G(p,A)+G(\bar{p},\bar{A}),\ \forall p,\bar{p}\in\mathbb{R}^{d},\ A,\bar{A}\in S(d);$

(2) $G(\lambda p,\lambda A)=\lambda G(p,A),\ \forall\lambda\geqslant0;$

(3) $G(p,A)\leqslant G(p,\bar{A}),\ \forall A\leqslant\bar{A}$.
\begin{defn}
Let $X,\eta\in\mathcal{H}^{d}$ be two random vectors. $(X,\eta)$
is called $G$-distributed if for any $\varphi\in C_{l,Lip}(\mathbb{R}^{d}\times\mathbb{R}^{d}),$
the function 
\[
u(t,x,y):=\mathbb{E}[\varphi(x+\sqrt{t}X,y+t\eta)],\ (t,x,y)\in[0,\infty)\times\mathbb{R}^{d}\times\mathbb{R}^{d},
\]
is a viscosity solution of the following parabolic PDE (called a $G$-heat
equation):
\begin{equation}
\partial_{t}u-G(D_{y}u,D_{x}^{2}u)=0,\label{G-Heat eqn}
\end{equation}
with Cauchy condition $u|_{t=0}=\varphi$.
\end{defn}

\begin{rem}
From the general theory of viscosity solutions (see \cite{crandall1992user},
\cite{peng2010nonlinear}), the $G$-heat equation (\ref{G-Heat eqn})
has a unique viscosity solution. By solving the $G$-heat equation
(\ref{G-Heat eqn}) (in some special cases, it is explicitly solvable),
we can compute the sublinear expectation of some functionals of a
$G$-distributed random vector. The case of convex functionals, for
instance, the power function $|x|^{k}$, is quite interesting.
\end{rem}

It can be proved that for such a function $G,$ there exists a bounded,
closed and convex subset $\Gamma\subset\mathbb{R}^{d}\times\mathbb{R}^{d\times d},$
such that $G$ has the following representation:
\[
G(p,A)=\sup_{(q,Q)\in\Gamma}\{\frac{1}{2}\mbox{tr\ensuremath{(AQQ^{T})}}+\langle p,q\rangle\},\ \forall(p,A)\in\mathbb{R}^{d}\times S(d).
\]
The set $\Gamma$ captures the uncertainty of probability distribution
(mean uncertainty and variance uncertainty) of a $G$-distributed
random vector. 

In particular, if $G$ only depends on $p\in\mathbb{R}^{d},$ then
there exists some bounded, closed and convex subset $\Lambda\subset\mathbb{R}^{d},$
such that 
\[
G(p)=\sup_{q\in\Lambda}\langle p,q\rangle.
\]
In this case a $G$-distributed random vector $\eta$ is called maximal
distributed and is denoted by $\eta\sim N(\Lambda,\{0\}).$ Similarly,
if $G$ only depends on $A\in S(d),$ then there exists some bounded,
closed and convex subset $\Sigma\subset S_{+}(d)$ (the space of symmetric
and nonnegative definite matrices) such that 
\begin{equation}
G(A)=\frac{1}{2}\sup_{B\in\Sigma}\mbox{tr\ensuremath{(AB)}},\ \forall A\in S(d).\label{G representation}
\end{equation}
A $G$-distributed random vector $X$ for such $G$ is called $G$-normal
distributed and is denoted by $X\sim N(\{0\},\Sigma).$

Now we introduce the concept of $G$-Brownian motion and related stochastic
calculus.

From now on, let $G:\ S(d)\rightarrow\mathbb{R}$ be a function given
by (\ref{G representation}).
\begin{defn}
\label{G-Brownian motion}A $d$-dimensional process $B_{t}$ is called
a $G$-Brownian motion if 

(1) $B_{0}(\omega)=0,\ \forall\omega\in\Omega;$

(2) for each $s,t\geqslant0$, $B_{t+s}-B_{t}\sim N(\{0\},s\Sigma)$
and is independent from $(B_{t_{1}},\cdots,B_{t_{n}})$ for any $n\geqslant1$
and $0\leqslant t_{1}<\cdots<t_{n}\leqslant t.$
\end{defn}

Similar to the classical situation, a $G$-Brownian motion can be
constructed explicitly on the canonical path space by using independent
$G$-normal random vectors. We refer the readers to \cite{peng2010nonlinear}
for a detailed construction. 

In summary, let $\Omega=C_{0}([0,\infty);\mathbb{R}^{d})$ be the
space of $ $$\mathbb{R}^{d}$-valued continuous paths starting at
the origin, and let $B_{t}(\omega):=\omega_{t}$ be the coordinate
process. For any $T\geqslant0,$ define
\[
L_{ip}(\Omega_{T}):=\{\varphi(B_{t_{1}},\cdots,B_{t_{n}}):\ n\geqslant1,t_{1},\cdots,t_{n}\in[0,T],\varphi\in C_{l,Lip}(\mathbb{R}^{d\times n})\},
\]
and 
\[
L_{ip}(\Omega):=\bigcup_{n=1}^{\infty}L_{ip}(\Omega_{n}).
\]
Then on $(\Omega,L_{ip}(\Omega))$ we can define the canonical sublinear
expectation $\mathbb{E}$ such that the coordinate process $B_{t}$
becomes a $G$-Brownian motion, which is usually called the $G$-expectation
and denoted by $\mathbb{E}^{G}.$ $(\Omega,L_{ip}(\Omega),\mathbb{E}^{G})$
is also called the canonical $G$-expectation space. Throughout the
rest of this paper, we will restrict ourselves on the canonical $G$-expectation
space and its completion (to be defined later on).

On $(\Omega,L_{ip}(\Omega),\mathbb{E}^{G})$ we can introduce the
notion of conditional $G$-expectation. More precisely, for 
\[
X=\varphi(B_{t_{1}},B_{t_{2}}-B_{t_{1}},\cdots,B_{t_{n}}-B_{t_{n-1}})\in L_{ip}(\Omega),
\]
where $0\leqslant t_{1}<t_{2}<\cdots<t_{n},$ the $G$-conditional
expectation of $X$ under $\Omega_{t_{j}}$ is defined by 
\[
\mathbb{E}^{G}[X|\Omega_{t_{j}}]:=\psi(B_{t_{1}},B_{t_{2}}-B_{t_{1}},\cdots,B_{t_{j}}-B_{t_{j-1}}),
\]
where 
\[
\psi(x_{1},\cdots,x_{j}):=\mathbb{E}^{G}[\varphi(x_{1},\cdots,x_{j},B_{t_{j+1}}-B_{t_{j}},\cdots,B_{t_{n}}-B_{t_{n-1}})],\ x_{1},\cdots,x_{j}\in\mathbb{R}^{d}.
\]
The conditional $G$-expectation $\mathbb{E}^{G}[\cdot|\Omega_{t}]$
has the following properties: for any $X,Y\in L_{ip}(\Omega),$

(1) if $X\leqslant Y,$ then $\mathbb{E}^{G}[X|\Omega_{t}]\leqslant\mathbb{E}^{G}[Y|\Omega_{t}];$

(2) $\mathbb{E}^{G}[X+Y|\Omega_{t}]\leqslant\mathbb{E}^{G}[X|\Omega_{t}]+\mathbb{E}^{G}[Y|\Omega_{t}];$

(3) for any $\eta\in L_{ip}(\Omega_{t})$, 
\begin{eqnarray*}
\mathbb{E}^{G}[\eta|\Omega_{t}] & = & \eta,\\
\mathbb{E}^{G}[\eta X|\Omega_{t}] & = & \eta^{+}\mathbb{E}^{G}[X|\Omega_{t}]+\eta^{-}\mathbb{E}[-X|\Omega_{t}];
\end{eqnarray*}

(4) $\mathbb{E}^{G}[\mathbb{E}^{G}[X|\Omega_{t}]|\Omega_{s}]=\mathbb{E}^{G}[X|\Omega_{t\wedge s}].$
In particular, $\mathbb{E}^{G}[\mathbb{E}^{G}[X|\Omega_{t}]]=\mathbb{E}^{G}[X].$

For any $p\geqslant1,$ let $L_{G}^{p}$ (respectively, $L_{G}^{p}(\Omega_{t}))$)
be the completion of $L_{ip}(\Omega)$ (respectively, $L_{ip}(\Omega_{t})$)
under the semi-norm $\|X\|_{p}:=(\mathbb{E}^{G}[|X|^{p}])^{\frac{1}{p}}.$
Then $ $$\mathbb{E}^{G}$ can be continuously extended to a sublinear
expectation on $L_{G}^{p}(\Omega)$ (respectively, $L_{G}^{p}(\Omega_{t})$),
still denoted by $\mathbb{E}^{G}.$

For $t<T\leqslant\infty,$ the conditional $ $$G$-expectation $\mathbb{E}^{G}[\cdot|\Omega_{t}]:\ L_{ip}(\Omega_{T})\rightarrow L_{ip}(\Omega_{t})$
is a continuous mapping under $\|\cdot\|_{1}$ and can be continuously
extended to a mapping
\[
\mathbb{E}^{G}[\cdot|\Omega_{t}]:\ L_{G}^{1}(\Omega_{T})\rightarrow L_{G}^{1}(\Omega_{t}),
\]
which can still be interpreted as the conditional $G$-expectation.
It is easy to show that the properties (1) to (4) for the conditional
$G$-expectation still hold true on $L_{G}^{1}(\Omega_{T})$ as long
as it is well-defined.

Now we introduce the related stochastic calculus for $G$-Brownian
motion and (It$\hat{\mbox{o}}$ type) stochastic differential equations
(SDEs) driven by $G$-Brownian motion.

First of all, similar to the idea in the classical case, we still
have the notion of It$\hat{\mbox{o}}$ integral with respect to a
$1$-dimensional $G$-Brownian motion. More precisely, consider $d=1,$
we can first define It$\hat{\mbox{o}}$ integral of simple processes
and then pass limit under the $G$-expectation $\mathbb{E}^{G}$ in
some suitable functional spaces. Let $M_{G}^{p,0}(0,T)$ be the space
of simple processes $\eta_{t}(\omega)$ on $[0,T]$ of the form
\[
\eta_{t}(\omega)=\sum_{k=1}^{N}\xi_{k-1}(\omega)\mathbf{1}_{[t_{k-1},t_{k})}(t),
\]
where $\pi_{T}^{N}:=\{t_{0},t_{1},\cdots,t_{N}\}$ is a partition
of $[0,T]$ and $\xi_{k}\in L_{G}^{p}(\Omega_{t_{k}})$, and introduce
the semi-norm 
\[
\|\eta\|_{M_{G}^{p}(0,T)}:=(\mathbb{E}^{G}[\int_{0}^{T}|\eta_{t}|^{p}dt])^{\frac{1}{p}}
\]
on $M_{G}^{p,0}(0,T).$ Let $M_{G}^{p}(0,T)$ be the completion of
$M_{G}^{p,0}(0,T)$ under $\|\cdot\|_{M_{G}^{p}(0,T)}.$ It is straight
forward to define It$\hat{\mbox{o}}$ integral $\int_{0}^{T}\eta_{t}dB_{t}$
of simple processes. Moreover, such an integral operator is linear
and continuous under $\|\cdot\|_{M_{G}^{p}(0,T)}$ and hence can be
extended to a bounded linear operator 
\[
I:\ M_{G}^{2}(0,T)\rightarrow L_{G}^{2}(0,T).
\]
The operator $I$ is defined as the It$\hat{\mbox{o}}$ integral operator
with respect to a $G$-Brownian motion. For $0\leqslant s<t\leqslant T,$
define
\[
\int_{s}^{t}\eta_{u}dB_{u}:=\int_{0}^{T}\mathbf{1}_{[s,t]}(u)\eta_{u}dB_{u}.
\]
We list some important properties of $G$-It$\hat{\mbox{o}}$ integral
in the following.
\begin{prop}
Let $\eta,\theta\in M_{G}^{2}(0,T)$ and let $0\leqslant s\leqslant r\leqslant t\leqslant T.$
Then 

(1) 
\[
\int_{s}^{t}\eta_{u}dB_{u}=\int_{s}^{r}\eta_{u}dB_{u}+\int_{r}^{t}\eta_{u}dB_{u};
\]

(2) if $\alpha$ is bounded in $L_{G}^{1}(\Omega_{s}),$ then 
\[
\int_{s}^{t}(\alpha\eta_{u}+\theta_{u})dB_{u}=\alpha\int_{s}^{t}\eta_{u}dB_{u}+\int_{s}^{t}\theta_{u}dB_{u};
\]

(3) for any $X\in L_{G}^{1}(\Omega),$
\[
\mathbb{E}^{G}[X+\int_{r}^{T}\eta_{u}dB_{u}|\Omega_{s}]=\mathbb{E}^{G}[X|\Omega_{s}];
\]

(4) 
\[
\underline{\sigma}^{2}\mathbb{E}^{G}[\int_{0}^{T}\eta_{t}^{2}dt]\leqslant\mathbb{E}^{G}[(\int_{0}^{T}\eta_{t}dB_{t})^{2}]\leqslant\overline{\sigma}^{2}\mathbb{E}^{G}[\int_{0}^{T}\eta_{t}^{2}dt],
\]
 where $\overline{\sigma}^{2}:=\mathbb{E}^{G}[B_{1}^{2}]$ and $\underline{\sigma}^{2}:=-\mathbb{E}^{G}[-B_{1}^{2}].$
\end{prop}

Secondly, we have the notion of quadratic variation process of $G$-Brownian
motion. In the case of $1$-dimensional $G$-Brownian motion, the
quadratic variation process $\langle B\rangle_{t}$ is defined as
\[
\langle B\rangle_{t}:=B_{t}^{2}-2\int_{0}^{t}B_{s}dB_{s},
\]
which can be regarded as the $L_{G}^{2}$-limit of the sum $\sum_{j=1}^{k_{N}}(B_{t_{j}^{N}}-B_{t_{j-1}^{N}})^{2}$
as $\mu(\pi_{t}^{N})\rightarrow0$, where $\pi_{t}^{N}:=\{t_{j}^{N}\}_{j=0}^{k_{N}}$
is a sequence of partitions of $[0,t]$ and 
\[
\mu(\pi_{t}^{N}):=\max\{t_{j}^{N}-t_{j-1}^{N}:\ j=1,2,\cdots,k_{N}\}.
\]
It follows that $\langle B\rangle_{t}$ is an increasing process with
$\langle B\rangle_{0}=0$. 

Similar to the definition of $G$-It$\hat{\mbox{o}}$ integral, we
can define the integration with respect to $\langle B\rangle_{t}$
where $B_{t}$ is a $1$-dimensional $G$-Brownian motion. We refer
the readers to \cite{peng2010nonlinear} for a detailed construction
but we remark that the integral operator with respect to $ $$\langle B\rangle_{t}$
is a continuous linear mapping 
\[
Q_{0,T}:\ M_{G}^{1}(0,T)\rightarrow L_{G}^{1}(\Omega_{T}).
\]

The following identity can be regarded as the $G$-It$\hat{\mbox{o}}$
isometry.
\begin{prop}
Let $\eta\in M_{G}^{2}(0,T),$ then 
\[
\mathbb{E}^{G}[(\int_{0}^{T}\eta_{t}dB_{t})^{2}]=\mathbb{E}^{G}[\int_{0}^{T}\eta_{t}^{2}d\langle B\rangle_{t}].
\]

\end{prop}

Now consider the multi-dimensional case. Let $B_{t}$ is a $d$-dimensional
$G$-Brownian motion, and for any $v\in\mathbb{R}^{d},$ denote 
\[
B_{t}^{v}:=\langle v,B_{t}\rangle,
\]
where $\langle\cdot,\cdot\rangle$ is the Euclidean inner product.
Then for $a,\overline{a}\in\mathbb{R}^{d},$ the cross variation process
$\langle B^{a},B^{\overline{a}}\rangle_{t}$ is defined as 
\[
\langle B^{a},B^{\overline{a}}\rangle_{t}=\frac{1}{4}(\langle B^{a+\overline{a}},B^{a+\overline{a}}\rangle_{t}-\langle B^{a-\overline{a}},B^{a-\overline{a}}\rangle_{t}).
\]
Similar to the case of quadratic variation process, we have 
\begin{eqnarray*}
\langle B^{a},B^{\overline{a}}\rangle_{t} & = & (L_{G}^{2}-)\lim_{\mu(\pi_{t}^{N})\rightarrow0}\sum_{j=1}^{k_{N}}(B_{t_{j}^{N}}^{a}-B_{t_{j-1}^{N}}^{a})(B_{t_{j}^{N}}^{\overline{a}}-B_{t_{j-1}^{N}}^{\overline{a}})\\
 & = & B_{t}^{a}B_{t}^{\overline{a}}-\int_{0}^{t}B_{s}^{a}dB_{s}^{\overline{a}}-\int_{0}^{t}B_{s}^{\overline{a}}dB_{s}^{a}.
\end{eqnarray*}

Note that unlike the classical case, the cross variation process is
not deterministic. The following results characterizes the distribution
of $\langle B\rangle_{t}:=(\langle B^{\alpha},B^{\beta}\rangle_{t})_{\alpha,\beta=1}^{d},$
where $B_{t}$ is a $d$-dimensional $G$-Brownian motion and $B_{t}^{\alpha}$
is the $\alpha$-th component of $B_{t}.$
\begin{prop}
Recall that the function $G$ has the representation (\ref{G representation}).
Then $\langle B\rangle_{t}\sim N(t\Sigma,\{0\})$. 
\end{prop}

As in the classical case, we also have the important $G$-It$\hat{\mbox{o}}$
formula under $G$-expectation, which takes a similar form to the
classical one. The main difference is that $dB_{t}^{\alpha}\cdot dB_{t}^{\beta}$
should be $d\langle B^{\alpha},B^{\beta}\rangle_{t}$ instead of $\delta_{\alpha\beta}dt.$
We are not going to state the full result of $G$-It$\hat{\mbox{o}}$
formula here. See \cite{peng2010nonlinear} for a detailed discussion.

Now we introduce the notion of SDEs driven by $G$-Brownian motion. 

For $p\geqslant1,$ let $\overline{M}_{G}^{p}(0,T;\mathbb{R}^{n})$
be the completion of $M_{G}^{p,0}(0,T;\mathbb{R}^{n})$ under the
norm 
\[
\|\eta\|_{\overline{M}_{G}^{p}(0,T;\mathbb{R}^{n})}:=(\int_{0}^{T}\mathbb{E}^{G}[|\eta_{t}|^{p}]dt)^{\frac{1}{p}}.
\]
It is easy to see that $\overline{M}_{G}^{p}(0,T;\mathbb{R}^{n})\subset M_{G}^{p}(0,T;\mathbb{R}^{n}).$

Consider the following $N$-dimensional SDE driven by $G$-Brownian
motion over $[0,T]$:
\begin{equation}
dX_{t}=b(t,X_{t})dt+\sum_{\alpha,\beta=1}^{d}h_{\alpha\beta}(t,X_{t})d\langle B^{\alpha},B^{\beta}\rangle_{t}+\sum_{\alpha=1}^{d}V_{\alpha}(t,X_{t})dB_{t}^{\alpha}\label{SDE in preliminaries}
\end{equation}
with initial condition $\xi\in\mathbb{R}^{N}.$ Here we assume that
the coefficients $b^{i},h_{\alpha\beta}^{i},V_{\alpha}^{i}$ are Lipschitz
functions in the space variable, uniformly in time. A solution of
(\ref{SDE in preliminaries}) is a process in $\overline{M}_{G}^{2}(0,T;\mathbb{R}^{N})$
satisfying the equation (\ref{SDE in preliminaries}) in its integral
form.

The existence and uniqueness of (\ref{SDE in preliminaries}) was
studied by Peng \cite{peng2010nonlinear}.
\begin{thm}
There exists a unique solution $X\in\overline{M}_{G}^{2}(0,T;\mathbb{R}^{N})$
to the SDE (\ref{SDE in preliminaries}).
\end{thm}

Finally, we introduce the notion of quasi-sure analysis for $G$-expectation.
It plays an important role in studying pathwise properties of stochastic
processes under the framework of $G$-expectation.

First of all, on the canonical sublinear expectation space $(\Omega,L_{ip}(\Omega),\mathbb{E}^{G}),$
we can prove a refinement of Theorem \ref{representation of E}: there
exists a weakly compact family $\mathcal{P}$ of probability measures
on $(\Omega,\mathcal{B}(\Omega))$, such that for any $X\in L_{ip}(\Omega)$
and $P\in\mathcal{P}$, $\mathbb{E}_{P}[X]$ is well-defined and 
\[
\mathbb{E}^{G}[X]=\max_{P\in\mathcal{P}}\mathbb{E}_{P}[X],\ \ \ \forall X\in L_{ip}(\Omega),
\]
where ``max'' means that the supremum is attainable (for each $X$).
Moreover, there is an explicit characterization of the family $\mathcal{P}$.
Let $ $$G$ be represented in the following way:
\[
G(A)=\frac{1}{2}\sup_{Q\in\Gamma}\mbox{tr}(AQQ^{T}),
\]
for some bounded, closed and convex subset $\Gamma\subset\mathbb{R}^{d\times d},$
and let $\mathcal{A}_{\Gamma}$ be the collection of all $\Gamma$-valued
$\{\mathcal{F}_{t}^{W}:t\geqslant0\}$-adapted processes on $[0,\infty),$
where $\{\mathcal{F}_{t}^{W}:t\geqslant0\}$ is the natural filtration
of the coordinate process on $\Omega.$ Let $\mathcal{P}_{0}$ be
the collection of probability laws of the following classical It$\hat{\mbox{o}}$
integral processes with respect to the standard Wiener measure: 
\[
B_{t}^{\gamma}:=\int_{0}^{t}\gamma_{s}dW_{s},\ t\geqslant0,\ \gamma\in A_{\Gamma}.
\]
Then $\mathcal{P}=\overline{\mathcal{P}_{0}}.$ For the proof of this
result, please refer to \cite{denis2011function}. 

For this particular family $\mathcal{P},$ define the set function
$c$ by
\[
c(A):=\sup_{P\in\mathcal{P}}P(A),\ A\in\mathcal{B}(\Omega).
\]
Then we have the following result.
\begin{thm}
The set function $c$ is a Choquet capacity (for an introduction of
capacity theory, see \cite{choquet1953theory}, \cite{dellacherie1972capacites}).
In other words,

(1) for any $A\in\mathcal{B}(\Omega),$ $0\leqslant c(A)\leqslant1;$

(2) if $A\subset B,$ then $c(A)\leqslant c(B);$

(3) if $A_{n}$ is a sequence in $\mathcal{B}(\Omega),$ then $c(\cup_{n}A_{n})\leqslant\sum_{n}c(A_{n});$

(4) if $A_{n}$ is increasing in $\mathcal{B}(\Omega),$ then $c(\cup A_{n})=\lim_{n\rightarrow\infty}c(A_{n}).$
\end{thm}

For any $\mathcal{B}(\Omega)$-measurable random variable $X$ such
that $\mathbb{E}_{P}[X]$ is well-defined for all $P\in\mathcal{P},$
define the upper expectation 
\[
\hat{\mathbb{E}}[X]:=\sup_{P\in\mathcal{P}}\mathbb{E}_{P}[X].
\]
Then we can prove that for any $0\leqslant T\leqslant\infty$ and
$X\in L_{G}^{1}(\Omega_{T}),$ 
\[
\mathbb{E}^{G}[X]=\hat{\mathbb{E}}[X].
\]
For a detailed discussion and other related properties, please refer
to \cite{denis2011function}.

The following Markov inequality and Borel-Cantelli lemma under the
capacity $c$ are important for us.
\begin{thm}
(1) For any $X\in L_{G}^{p}(\Omega)$ and $\lambda>0,$ we have
\[
c(|X|>\lambda)\leqslant\frac{\mathbb{E}^{G}[|X|^{p}]}{\lambda^{p}}.
\]

(2) Let $A_{n}$ be a sequence in $\mathcal{B}(\Omega)$ such that
\[
\sum_{n=1}^{\infty}c(A_{n})<\infty.
\]
Then 
\[
c(\limsup A_{n})=0.
\]

\end{thm}

\begin{defn}
A property depending on $\omega\in\Omega$ is said to hold quasi-surely,
if it holds outside a $\mathcal{B}(\Omega)$-measurable subset of
zero capacity.
\end{defn}

\subsection{Rough Path Theory and Rough Differential Equations}

Now we introduce some fundamentals in the theory of rough paths and
rough differential equations. For a systematic introduction, please
refer to \cite{friz2010multidimensional}, \cite{lyons2007differential},
\cite{lyons2002system}.

For $n\geqslant1,$ define
\[
T^{(\infty)}(\mathbb{R}^{d}):=\oplus_{k=0}^{\infty}(\mathbb{R}^{d})^{\otimes k}
\]
to be the infinite tensor algebra and 
\[
T^{(n)}(\mathbb{R}^{d}):=\oplus_{k=0}^{n}(\mathbb{R}^{d})^{\otimes k}
\]
to be the truncated tensor algebra of order $n,$ equipped with the
Euclidean norm. Let $\Delta$ be the triangle region $\{(s,t):0\leqslant s<t\leqslant1\}.$
A functional $\mathbf{X}:\ \Delta\rightarrow T^{(n)}(\mathbb{R}^{d})$
of order $n$ is called multiplicative if for any $s<u<t,$
\[
\mathbf{X}_{s,t}=\mathbf{X}_{s,u}\otimes\mathbf{X}_{u,t}.
\]
Such a multiplicative structure is called the Chen identity. It describes
the (nonlinear) additive structure of integrals over different intervals.

A control function $\omega$ is a nonnegative continuous function
on $\Delta$ such that for any $s<u<t,$ 
\[
\omega(s,u)+\omega(u,t)\leqslant\omega(s,t),
\]
and for any $t\in[0,1],$ $\omega(t,t)=0.$ An example of control
function $\omega(s,t)$ is the $1$-variation norm over $[s,t]$ of
a path with bounded variation. 

Let $p\geqslant1$ be a fixed constant. A continuous and multiplicative
functional 
\[
\mathbf{X}_{s,t}=(1,X_{s,t}^{1},\cdots,X_{s,t}^{n})
\]
 of order $n$ has finite $p$-variation if for some control function
$\omega$,
\begin{equation}
|X_{s,t}^{i}|\leqslant\omega(s,t)^{\frac{i}{p}},\ \forall i=1,2,\cdots,n,\ (s,t)\in\Delta.\label{finite p-var}
\end{equation}
$\mathbf{X}$ has finite $p$-variation if and only if for any $i=1,2,\cdots,n,$
\[
\sup_{D}\sum_{l}|X_{t_{l-1},t_{l}}^{i}|^{\frac{p}{i}}<\infty,
\]
where $\sup_{D}$ runs over all finite partitions of $[0,1].$ We
can also introduce the notion of finite $p$-variation for multiplicative
functionals in $T^{(\infty)}(\mathbb{R}^{d})$ by allowing $1\leqslant i<\infty$
in (\ref{finite p-var}). A continuous and multiplicative functional
$\mathbf{X}$ of order $[p]$ with finite $p$-variation is called
a rough path with roughness $p.$ The space of rough paths with roughness
$p$ is denoted by $\Omega_{p}(\mathbb{R}^{d}).$

The following Lyons lifting theorem (see \cite{lyons1998differential})
shows that the higher levels of a rough path $\mathbf{X}$ with roughness
$p$ are uniquely determined by $\mathbf{X}$ itself.
\begin{thm}
\label{Lyons lifting} Let $\mathbf{X}$ be a rough path with roughness
$p.$ Then $\mathbf{X}$ can be uniquely extended to a continuous
and multiplicative functional in $T^{(\infty)}(\mathbb{R}^{d})$ with
finite $p$-variation.
\end{thm}

One of the motivation of introducing the concept of rough paths is
to develop the theory of differential equations driven by rough signals. 

If an $\mathbb{R}^{d}$-valued path $X$ has bounded variation, we
know that the Picard iteration for the following differential equation
converges:
\begin{equation}
dY_{t}=V(Y_{t})dX_{t},\label{Lipschitz ODE}
\end{equation}
where $V=(V_{1},\cdots,V_{d})$ is a family of Lipschitz vector fields.
Another way to consider (\ref{Lipschitz ODE}) is to use the Euler
scheme, which can be regarded as the Taylor expansion of functional
of paths. Namely, we can write informally that 
\begin{equation}
Y_{t}-Y_{s}\thicksim\sum_{n=1}^{\infty}\sum_{\alpha_{1},\cdots,\alpha_{n}=1}^{^{d}}V_{\alpha_{1}}\cdots V_{\alpha_{n}}I(Y_{s})\int_{s<u_{1}<\cdots<u_{n}<t}dX_{u_{1}}^{\alpha_{1}}\cdots dX_{u_{n}}^{\alpha_{n}}.\label{Euler scheme}
\end{equation}
From (\ref{Euler scheme}) we can see that the sequence 
\[
\mathbf{X}_{s,t}:=(1,X_{t}-X_{s},\int_{s<u<v<t}dX_{u}\otimes dX_{v},\cdots,\int_{s<u_{1}<\cdots<u_{n}<t}dX_{u_{1}}\otimes\cdots\otimes dX_{u_{n}},\cdots)\in T^{(\infty)}(\mathbb{R}^{d})
\]
contains exactly all the information to determine the solution $Y.$
On the other hand, it can be proved that $\mathbf{X}$ is multiplicative
and of finite $1$-variation. Since $X$ has bounded variation, it
follows from Theorem \ref{Lyons lifting} that $\mathbf{X}$ is the
unique enhancement of $X$. This is the fundamental reason why we
don't need to see the higher levels when solving equation (\ref{Lipschitz ODE})-all
information about $\mathbf{X}$, which uniquely determines the solution
of (\ref{Lipschitz ODE}), is incorporated in the first level.

If the driven signal is rougher, the situation becomes different.
The same thing is that the information to determine the solution lies
in the multiplicative structure in $T^{(\infty)}(\mathbb{R}^{d})$,
while the difference is that, unlike the case of paths with bounded
variation, the classical path itself may not be able to determine
the higher levels which are crucial to characterize the solution of
a differential equation. In other words, we need to specify higher
levels of the classical path in order to make sense of differential
equations. According to Theorem \ref{Lyons lifting}, we know that
the higher levels (levels above $[p]$) of a rough path $\mathbf{X}$
with roughness $p$ are uniquely determined by $\mathbf{X}$ itself.
Therefore, to establish differential equations driven by signals rougher
than paths of bounded variation, we need to interpret the driven signal
as a rough path with certain roughness $p$, that is, the driving
signal should be an element in the space $\Omega_{p}(\mathbb{R}^{d}).$ 

When the driving signal $\mathbf{X}$ is in some smaller space of
$\Omega_{p}(\mathbb{R}^{d})$ in which $\mathbf{X}$ can be approximated
by paths of bounded variation in some sense, we are able to use a
natural approximation procedure to introduce the notion of differential
equations. But first we need to introduce a certain kind of topology.

Define the $p$-variation distance $d_{p}(\cdot,\cdot)$ on $\Omega_{p}(\mathbb{R}^{d})$
by 
\[
d_{p}(\mathbf{X},\mathbf{Y}):=\max_{1\leqslant i\leqslant[p]}\sup_{D}(\sum_{l}|X_{t_{l-1},t_{l}}^{i}-Y_{t_{l-1},t_{l}}^{i}|^{\frac{p}{i}})^{\frac{i}{p}}.
\]
Then $(\Omega_{p}(\mathbb{R}^{d}),d_{p})$ is a complete metric space. 

A continuous path $X\in C([0,1];\mathbb{R}^{d})$ is called smooth
if it has bounded variation. Let 
\[
\Omega_{p}^{\infty}(\mathbb{R}^{d}):=\{\mbox{\ensuremath{\mathbf{X}}}:\ \mbox{\ensuremath{\mathbf{X}} is the unique enhancement of \ensuremath{X}in \ensuremath{T^{([p])}(\mathbb{R}^{d}),}where \ensuremath{X}is smooth}\}
\]
be the subspace of enhanced smooth paths of order $[p]$. The closure
of $\Omega_{p}^{\infty}(\mathbb{R}^{d})$ under the $p$-variation
distance $d_{p},$ denoted by $G\Omega_{p}(\mathbb{R}^{d}),$ is called
the space of geometric rough paths with roughness $p.$

The following theorem, proved by Lyons \cite{lyons1998differential},
which is usually known as the universal limit theorem, enables us
to introduce the notion of differential equations driven by geometric
rough paths.
\begin{thm}
\label{universal limit theorem}Let $V_{1},\cdots,V_{d}\in C_{b}^{[p]+1}(\mathbb{R}^{d})$
be given vector fields on $\mathbb{R}^{N}.$ For a given $y\in\mathbb{R}^{d},$
define the mapping
\[
F(y,\cdot):\ \Omega_{p}^{\infty}(\mathbb{R}^{d})\rightarrow G\Omega_{p}^{\infty}(\mathbb{R}^{N})
\]
in the following way. For any $\mathbf{X}\in\Omega_{p}^{\infty}(\mathbb{R}^{d}),$
let $X$ be the smooth path associated with $\mathbf{X}$ starting
at the origin (i.e., projection of $\mathbf{X}$ onto the first level),
and $Y$ be the unique smooth path which is the solution of the following
ODE:
\[
dY_{t}=V_{\alpha}(Y_{t})dX_{t}^{\alpha}
\]
with $Y_{0}=y.$ $F(y,\mathbf{X})$ is defined to be the enhancement
of $Y$ in $\Omega_{p}^{\infty}(\mathbb{R}^{d}).$ Then the mapping
$F(y,\cdot)$ is continuous with respect to the corresponding $p$-variation
distance $d_{p}.$
\end{thm}

According to Theorem \ref{universal limit theorem}, there exists
a unique continuous extension of $F(y,\cdot)$ on $G\Omega_{p}(\mathbb{R}^{d})$.
The extended mapping
\[
F(y,\cdot):\ G\Omega_{p}(\mathbb{R}^{d})\rightarrow G\Omega_{p}(\mathbb{R}^{N}),
\]
is called the It$\hat{\mbox{o}}$-Lyons mapping. Such a mapping defines
the (unique) solution in the space $G\Omega_{p}(\mathbb{R}^{d})$
to the following differential equation:
\begin{equation}
dY_{t}=V(Y_{t})dX_{t},\label{rough differential equation}
\end{equation}
with initial value $y$. Equation (\ref{rough differential equation})
is called a rough differential equation driven by $\mathbf{X}$ (or
simply called an RDE), and the solution $\mathbf{Y}$ is called the
full solution of (\ref{rough differential equation}). If we are only
interested in classical paths, then 
\[
Y_{t}:=y+\pi_{1}(\mathbf{Y}),\ \ \ t\in[0,1],
\]
is called the solution of (\ref{rough differential equation}) with
initial value $y.$

\section{The Euler-Maruyama Approximation for SDEs Driven by $G$-Brownian
Motion}

In this section, we are going to establish the Euler-Maruyama approximation
for SDEs driven by $G$-Brownian motion. 

This result can be used to establish the Wong-Zakai type approximation
which reveals the relation between  SDEs (in the sense of $L_{G}^{2}(\Omega;\mathbb{R}^{N})$
by S. Peng) and RDEs (in the sense of rough paths by Lyons) driven
by $G$-Brownian motion. In Section 5, the study of such relation
will be our main focus. However, based on the result in the next section
which reveals the rough path nature of $G$-Brownian motion, we are
going to use the rough Taylor expansion in the theory of RDEs instead
of developing the Wong-Zakai type approximation to show that the solution
of an  SDE solves some associated RDE with a correction term in terms
of the cross variation process of multidimensional $G$-Brownian motion.
Such approach reveals the natural of $G$-Brownian motion and differential
equations in the sense of rough paths in a more fundamental way. 

We also believe that there will be other interesting applications
of the Euler-Maruyama approximation, such as in numerical analysis
under $G$-expectation, and in mathematical finance under uncertainty.

Consider the following $N$-dimensional SDE driven by the canonical
$d$-dimensional $G$-Brownian motion over $[0,1]$ on the sublinear
expectation space $(\Omega,L_{G}^{2}(\Omega),\mathbb{E}^{G})$ which
is the $L_{G}^{2}$-completion of the canonical path space $(\Omega,L_{ip}(\Omega),\mathbb{E}^{G})$:
\begin{eqnarray}
dX_{t} & = & b(X_{t})dt+h_{\alpha\beta}(X_{t})d\langle B^{\alpha},B^{\beta}\rangle_{t}+V_{\alpha}(X_{t})dB_{t}^{\alpha},\label{Ito SDE}
\end{eqnarray}
with initial condition $X_{0}=\xi\in\mathbb{R}^{N},$ where the coefficients
$b^{i},h_{\alpha\beta}^{i},V_{\alpha}^{i}$ are bounded and uniformly
Lipschitz. The existence and uniqueness of solution is studied by
Peng \cite{peng2010nonlinear}.

The Euler-Maruyama approximation of the solution $X_{t}$ of (\ref{Ito SDE})
is defined as follows. 

For $n\geqslant1,$ consider the dyadic partition of the time interval
$[0,1],$ i.e., 
\[
t_{k}^{n}=\frac{k}{2^{n}},\ k=0,1,\cdots,2^{n}.
\]
Define $X_{t}^{n}$ to be the approximation of $X_{t}$ in the following
evolutive way:
\begin{eqnarray*}
X_{0}^{n} & = & \xi,
\end{eqnarray*}
and for $t\in[t_{k-1}^{n},t_{k}^{n}],$
\[
(X_{t}^{n})^{i}=(X_{k-1}^{n})^{i}+V_{\alpha}^{i}(X_{k-1}^{n})\Delta_{k}^{n}B^{\alpha}+b^{i}(X_{k-1}^{n})\Delta t^{n}+h_{\alpha\beta}^{i}(X_{k-1}^{n})\Delta_{k}^{n}\langle B^{\alpha},B^{\beta}\rangle,
\]
where 
\[
X_{k-1}^{n}:=X_{t_{k-1}^{n}}^{n},\ \Delta_{k}^{n}B^{\alpha}:=B_{t_{k}^{n}}^{\alpha}-B_{t_{k-1}^{n}}^{\alpha},\ \Delta t^{n}:=\frac{1}{2^{n}},\ \Delta_{k}^{n}\langle B^{\alpha},B^{\beta}\rangle:=\langle B^{\alpha},B^{\beta}\rangle_{t_{k}^{n}}-\langle B^{\alpha},B^{\beta}\rangle_{t_{k-1}^{n}}.
\]

In this section, we are going to prove that $X_{t}^{n}$ converges
to the solution $X_{t}$ of (\ref{Ito SDE}) in $L_{G}^{2}(\Omega;\mathbb{R}^{N})$
with convergence rate $0.5,$ which coincides with the classical case
when $B_{t}$ reduces to a classical Brownian motion.

First of all, the following lemmas is useful for us.
\begin{lem}
\label{lem 1}Let $\eta_{t}$ be a bounded process in $M_{G}^{2}(0,1).$
Then for any $v\in\mathbb{R}^{d},$ $0\leqslant s<t\leqslant1,$
\[
\mathbb{E}^{G}[(\int_{s}^{t}\eta_{u}d\langle B^{v}\rangle_{u})^{2}]\leqslant\overline{\sigma}_{v}^{2}(t-s)\mathbb{E}^{G}[\int_{s}^{t}\eta_{u}^{2}d\langle B^{v}\rangle_{u}],
\]
where $\overline{\sigma}_{v}^{2}:=2G(v\cdot v^{T})$ and $B^{v}:=\langle v,B\rangle,$
in which $\langle\cdot,\cdot\rangle$ denotes the Euclidean inner
product of $\mathbb{R}^{d}.$\end{lem}
\begin{proof}
By approximation, it suffices to consider 
\[
\eta_{u}=\sum_{j=1}^{k}\zeta_{j-1}1_{[u_{j-1},u_{j})},
\]
where $s=u_{0}<u_{1}<\cdots<u_{k}=t$ and $\zeta_{j}\in L_{ip}(\Omega_{u_{j}})$
are bounded. In this case, by definition
\[
\int_{s}^{t}\eta_{u}d\langle B^{v}\rangle_{u}=\sum_{j=1}^{k}\zeta_{j-1}(\langle B^{v}\rangle_{u_{j}}-\langle B^{v}\rangle_{u_{j-1}}),
\]
and 
\[
\int_{s}^{t}\eta_{u}^{2}d\langle B^{v}\rangle_{u}=\sum_{j=1}^{k}\zeta_{j-1}^{2}(\langle B^{v}\rangle_{u_{j}}-\langle B^{v}\rangle_{u_{j-1}}),
\]
which are both defined in the pathwise sense for step functions. Since
$\langle B^{v}\rangle$ is increasing, the Cauchy-Schwarz inequality
yields that
\[
(\int_{s}^{t}\eta_{u}d\langle B^{v}\rangle_{u})^{2}\leqslant(\langle B^{v}\rangle_{t}-\langle B^{v}\rangle_{s})\cdot\int_{s}^{t}\eta_{u}^{2}d\langle B^{v}\rangle_{u}.
\]
Since $\zeta_{j}$ are bounded, if we use $M$ to denote an upper
bound of $\eta_{u}^{2},$ it follows that for any $c\geqslant\overline{\sigma}_{v}^{2},$
\[
(\int_{s}^{t}\eta_{u}d\langle B^{v}\rangle_{u})^{2}\leqslant M(\langle B^{v}\rangle_{t}-\langle B^{v}\rangle_{s}-c(t-s))^{+}(\langle B^{v}\rangle_{t}-\langle B^{v}\rangle_{s})+c(t-s)\int_{s}^{t}\eta_{u}^{2}d\langle B^{v}\rangle_{u}.
\]

Let $\varphi(x)=(x-c(t-s))^{+}x.$ Since $\langle B^{v}\rangle_{t}-\langle B^{v}\rangle_{s}$
is $N([\underline{\sigma}_{v}^{2},\overline{\sigma}_{v}^{2}]\times\{0\})$-distributed,
it follows that 
\begin{eqnarray*}
\mathbb{E}^{G}[\varphi(\langle B^{v}\rangle_{t}-\langle B^{v}\rangle_{s})] & = & \sup_{\underline{\sigma}_{v}^{2}\leqslant x\leqslant\overline{\sigma}_{v}^{2}}\varphi(x(t-s))\\
 & = & (t-s)^{2}\sup_{\underline{\sigma}_{v}^{2}\leqslant x\leqslant\overline{\sigma}_{v}^{2}}(x-c)^{+}x\\
 & = & 0.
\end{eqnarray*}
Therefore, by the sub-linearity of $G,$ we have 
\[
\mathbb{E}^{G}[(\int_{s}^{t}\eta_{u}d\langle B^{v}\rangle_{u})^{2}]\leqslant c(t-s)\mathbb{E}^{G}[\int_{s}^{t}\eta_{u}^{2}d\langle B^{v}\rangle_{u}],\ \ \ c\geqslant\overline{\sigma}_{v}^{2}.
\]

Now the proof is complete.
\end{proof}

Now we are in position to state and prove our main result of this
section.
\begin{thm}
\label{Euler-Maruyama}We have the following error estimate for the
Euler-Maruyama approximation: 
\[
\sup_{t\in[0,1]}\mathbb{E}^{G}[|X_{t}^{n}-X_{t}|^{2}]\leqslant C\Delta t^{n},
\]
where $C$ is some positive constant only depending on $d,N,G$ and
the coefficients of (\ref{Ito SDE}). In particular,
\[
\lim_{n\rightarrow\infty}\sup_{t\in[0,1]}\mathbb{E}^{G}[|X_{t}^{n}-X_{t}|^{2}]=0.
\]

\end{thm}

\begin{proof}
For $t\in[t_{k-1}^{n},t_{k}^{n}],$ by construction we have
\[
X_{t}^{i}-(X_{t}^{n})^{i}=I_{1}^{i}+J_{1}^{i}+K_{1}^{i}+I_{2}^{i}+J_{2}^{i}+K_{2}^{i},
\]
where
\begin{eqnarray*}
I_{1}^{i} & = & \sum_{l=1}^{k-1}\int_{t_{l-1}^{n}}^{t_{l}^{n}}(V_{\alpha}^{i}(X_{s})-V_{\alpha}^{i}(X_{s}^{n}))dB_{s}^{\alpha}+\int_{t_{k-1}^{n}}^{t}(V_{\alpha}^{i}(X_{s})-V_{\alpha}^{i}(X_{s}^{n}))dB_{s}^{\alpha},\\
J_{1}^{i} & = & \sum_{l=1}^{k-1}\int_{t_{l-1}^{n}}^{t_{l}^{n}}(b^{i}(X_{s})-b^{i}(X_{s}^{n}))ds+\int_{t_{k-1}^{n}}^{t}(b^{i}(X_{s})-b^{i}(X_{s}^{n}))ds,\\
K_{1}^{i} & = & \sum_{l=1}^{k-1}\int_{t_{l-1}^{n}}^{t_{l}^{n}}(h_{\alpha\beta}^{i}(X_{s})-h_{\alpha\beta}^{i}(X_{s}^{n}))d\langle B^{\alpha},B^{\beta}\rangle_{s}+\int_{t_{k-1}^{n}}^{t}(h_{\alpha\beta}^{i}(X_{s})-h_{\alpha\beta}^{i}(X_{s}^{n}))d\langle B^{\alpha},B^{\beta}\rangle_{s},\\
I_{2}^{i} & = & \sum_{l=1}^{k-1}\int_{t_{l-1}^{n}}^{t_{l}^{n}}(V_{\alpha}^{i}(X_{s}^{n})-V_{\alpha}^{i}(X_{l-1}^{n}))dB_{s}^{\alpha}+\int_{t_{k-1}^{n}}^{t}(V_{\alpha}^{i}(X_{s})-V_{\alpha}^{i}(X_{l-1}^{n}))dB_{s}^{\alpha},\\
J_{2}^{i} & = & \sum_{l=1}^{k-1}\int_{t_{l-1}^{n}}^{t_{l}^{n}}(b^{i}(X_{s}^{n})-b^{i}(X_{l-1}^{n}))ds+\int_{t_{k-1}^{n}}^{t}(b^{i}(X_{s})-b^{i}(X_{l-1}^{n}))ds,\\
K_{2}^{i} & = & \sum_{l=1}^{k-1}\int_{t_{l-1}^{n}}^{t_{l}^{n}}(h_{\alpha\beta}^{i}(X_{s})-h_{\alpha\beta}^{i}(X_{s}^{n}))d\langle B^{\alpha},B^{\beta}\rangle_{s}+\int_{t_{k-1}^{n}}^{t}(h_{\alpha\beta}^{i}(X_{s})-h_{\alpha\beta}^{i}(X_{s}^{n}))d\langle B^{\alpha},B^{\beta}\rangle_{s}.
\end{eqnarray*}
It follows that 
\begin{equation}
(X_{t}^{i}-(X_{t}^{n})^{i})^{2}\leqslant6((I_{1}^{i})^{2}+(J_{1}^{i})^{2}+(K_{1}^{i})^{2}+(I_{2}^{i})^{2}+(J_{2}^{i})^{2}+(K_{2}^{i})^{2}).\label{solution minus Euler}
\end{equation}

Throughout the rest of this section, we will always use the same notation
$C$ to denote constants only depending on $d,N,G$ and the coefficients
of (\ref{Ito SDE}), although they may be different from line to line.

Now the following estimates are important for further development.

(1) From the $G$-It$\hat{\mbox{o}}$ isometry, the distribution of
$\langle B^{\alpha}\rangle$ and the Lipschitz property, we have,
\[
\mathbb{E}^{G}[(\int_{t_{l-1}^{n}}^{u}(V_{\alpha}^{i}(X_{s})-V_{\alpha}^{i}(X_{s}^{n}))dB_{s}^{\alpha})^{2}]\leqslant C\int_{t_{l-1}^{n}}^{u}\mathbb{E}^{G}[|X_{s}-X_{s}^{n}|^{2}]ds,\ \forall u\in[t_{l-1}^{n},t_{l}^{n}].
\]

(2) Similarly, by Cauchy-Schwarz inequality, we have 
\[
\mathbb{E}^{G}[(\int_{t_{l-1}^{n}}^{u}(b^{i}(X_{s})-b^{i}(X_{s}^{n}))ds)^{2}]\leqslant C(u-t_{l-1}^{n})\int_{t_{l-1}^{n}}^{u}\mathbb{E}^{G}[|X_{s}-X_{s}^{n}|^{2}]ds,\ \forall u\in[t_{l-1}^{n},t_{l}^{n}].
\]
By the definition of $\langle B^{\alpha},B^{\beta}\rangle$ and Lemma
\ref{lem 1}, we also have 
\[
\mathbb{E}^{G}[(\int_{t_{l-1}^{n}}^{u}(h_{\alpha\beta}^{i}(X_{s})-h_{\alpha\beta}^{i}(X_{s}^{n}))d\langle B^{\alpha},B^{\beta}\rangle_{s})^{2}]\leqslant C(u-t_{l-1}^{n})\int_{t_{l-1}^{n}}^{u}\mathbb{E}^{G}[|X_{s}-X_{s}^{n}|^{2}]ds,\ \forall u[t_{l-1}^{n},t_{l}^{n}].
\]

(3) By construction and similar arguments to (1), (2), we have 
\begin{eqnarray*}
\mathbb{E}^{G}[(\int_{t_{l-1}^{n}}^{u}(V_{\alpha}^{i}(X_{s}^{n})-V_{\alpha}^{i}(X_{l-1}^{n}))dB_{s}^{\alpha})^{2}] & \leqslant & C(u-t_{l-1}^{n})^{2},\\
\mathbb{E}^{G}[(\int_{t_{l-1}^{n}}^{u}(b^{i}(X_{s}^{n})-b^{i}(X_{l-1}^{n}))ds)^{2}] & \leqslant & C(u-t_{l-1}^{n})^{3},\\
\mathbb{E}^{G}[(\int_{t_{l-1}^{n}}^{u}(h_{\alpha\beta}^{i}(X_{s})-h_{\alpha\beta}^{i}(X_{s}^{n}))d\langle B^{\alpha},B^{\beta}\rangle_{s})^{2}] & \leqslant & C(u-t_{l-1}^{n})^{3},
\end{eqnarray*}
for all $u\in[t_{l-1}^{n},t_{l}^{n}].$

(4) By conditioning and from the properties of It$\hat{\mbox{o}}$
integral with respect to $G$-Brownian motion, we know that the $G$-expectation
of each ``cross term'' in $(I_{1}^{i})^{2}$ and in $(I_{2}^{i})^{2}$
is zero.

Combining (1) to (4) and applying the following elementary inequality
to $(J_{1}^{i})^{2},(J_{2}^{i})^{2},(K_{1}^{i})^{2}$ and $(K_{2}^{i})^{2}$:
\[
(a_{1}+\cdots+a_{m})^{2}\leqslant m(a_{1}^{2}+\cdots+a_{m}^{2}),
\]
it is not hard to obtain that 
\[
\mathbb{E}^{G}[\|X_{t}-X_{t}^{n}\|^{2}]\leqslant C\int_{0}^{t}\mathbb{E}^{G}[\|X_{s}-X_{s}^{n}\|^{2}]ds+C(\Delta t^{n}),\ \forall t\in[0,1].
\]
By using Gronwall inequality, we arrive at 
\[
\mathbb{E}^{G}[\|X_{t}-X_{t}^{n}\|^{2}]\leqslant C(\Delta t^{n}),
\]
which completes the proof of the theorem.
\end{proof}

\section{$G$-Brownian Motion as Rough Paths and RDEs Driven by $G$-Brownian
Motion}

In this section, we are going to study the nature of sample paths
of $G$-Brownian motion under the framework of rough path theory.
More precisely, we are going to show that: on the canonical path space,
outside a Borel-measurable set of capacity zero, the sample paths
of $G$-Brownian motion can be enhanced to the second level in a canonical
way so that they become geometric rough paths with roughness $2<p<3.$
As pointed out before, such a result will enable us to establish RDEs
driven by $G$-Brownian motion in the space of geometric rough paths. 

Recall that $(\Omega,L_{ip}(\Omega),\mathbb{E}^{G})$ is the canonical
path space associated with the function $G,$ on which the coordinate
process 
\[
B_{t}(\omega):=\omega_{t},\ t\in[0,1],
\]
 is a $d$-dimensional $G$-Brownian motion with continuous sample
paths.

By the following moment inequality for $B_{t}$:
\begin{equation}
\mathbb{E}^{G}[|B_{t}-B_{s}|^{2q}]\leqslant C_{q}(t-s)^{q},\ \forall0\leqslant s<t\leqslant1,\ q>1,\label{moment inequality}
\end{equation}
and the generalized Kolmogorov criterion (see \cite{peng2010nonlinear}
for details), we know that for quasi-surely, the sample paths of $B_{t}$
are $\alpha$-H$\ddot{\mbox{o}}$lder continuous for any $\alpha\in(0,\frac{1}{2}).$
Therefore, if the sample paths of $B_{t}$ can be regarded as objects
in the space of geometric rough paths, the correct roughness should
be $2<p<3$ (so we should look for the enhancement of $B_{t}$ to
the second level); or in other words, the right topology we should
work with is the $p$-variation topology induced by the $p$-variation
distance $d_{p}$ on the space of geometric rough paths with roughness
$2<p<3$. The situation here is the same as the classical Brownian
motion, and the fundamental reason behind lies in the distribution
of $B_{t}$ (or more precisely, the moment inequality (\ref{moment inequality})),
which yields the same kind of H$\ddot{\mbox{o}}$lder continuity for
sample paths of $B_{t}$ as the classical one.

From now on, we will assume that $p\in(2,3)$ is some fixed constant.

As in the last section, for $n\geqslant1,$ $k=0,1,\cdots,2^{n},$
let $t_{k}^{n}=\frac{k}{2^{n}}$ be the dyadic partition of $[0,1],$
and let $B_{t}^{n}$ be the piecewise linear approximation of $B_{t}$
over the partition points $\{t_{0}^{n},t_{1}^{n},\cdots,t_{2^{n}}^{n}\}.$
Since the sample paths of $B_{t}^{n}$ are smooth, $B_{t}^{n}$ has
a unique enhancement 
\[
\boldsymbol{B}_{s,t}^{n}=(1,B_{s,t}^{n,1},B_{s,t}^{n,2}),\ 0\leqslant s<t\leqslant1,
\]
 to the space $G\Omega_{p}(\mathbb{R}^{d})$ of geometric rough paths
with roughness $p$ (in fact, for any $p\geqslant1$) determined by
iterated integrals. 

Our goal is to show that for quasi-surely, $\boldsymbol{B}^{n}$ is
a Cauchy sequence under the $p$-variation distance $d_{p}.$ It follows
that for quasi-surely, the sample paths of $B_{t}$ can be enhanced
to the second level as geometric rough paths with roughness $p$,
which are defined as limits of $\boldsymbol{B}^{n}$ under $d_{p}.$
Such an enhancement can be regarded as a canonical lifting by using
dyadic approximations.

Throughout the rest of this section, we will use $\|\cdot\|_{q}$
to denote the $L^{q}$-norm under the $G$-expectation $\mathbb{E}^{G}.$
Moreover, we will use the same notation $C$ to denote constants only
depending on $d,G,p,$ although they may be different from line to
line.

The following estimates are crucial for the proof of the main result
of this section. 
\begin{lem}
\label{estimates of 1 and 2 level}Let $m,n\geqslant1,$ and $k=1,2,\cdots,2^{n}.$
Then

(1) 
\[
\|B_{t_{k-1}^{n},t_{k}^{n}}^{m,j}\|_{\frac{p}{j}}\leqslant\begin{cases}
C(\frac{1}{2^{\frac{n}{2}}})^{j}, & n\leqslant m;\\
C(\frac{2^{\frac{m}{2}}}{2^{n}})^{j}, & n>m,
\end{cases}
\]
where $j=1,2.$

(2) 
\[
\|B_{t_{k-1}^{n},t_{k}^{n}}^{m+1,1}-B_{t_{k-1}^{n},t_{k}^{n}}^{m,1}\|_{p}\leqslant\begin{cases}
0, & n\leqslant m;\\
C\frac{2^{\frac{m}{2}}}{2^{n}}, & n>m,
\end{cases}
\]
\[
\|B_{t_{k-1}^{n},t_{k}^{n}}^{m+1,2}-B_{t_{k-1}^{n},t_{k}^{n}}^{m,2}\|_{\frac{p}{2}}\leqslant\begin{cases}
C\frac{1}{2^{\frac{m}{2}}2^{\frac{n}{2}}}, & n\leqslant m;\\
C\frac{2^{m}}{2^{2n}}, & n>m.
\end{cases}
\]

Here $\|\cdot\|_{q}$ denotes the $L^{q}$-norm under the $G$-expectation
$\mathbb{E}^{G},$ and $C$ is some positive constant not depending
on $m,n,k.$\end{lem}
\begin{proof}
(1) The first level. 

If $n\leqslant m$, then
\[
B_{t_{k-1}^{n},t_{k}^{n}}^{m,1}=B_{t_{k}^{n}}-B_{t_{k-1}^{n}}.
\]
It follows from the moment inequality (\ref{moment inequality}) that
\[
\mathbb{E}^{G}[|B_{t_{k-1}^{n},t_{k}^{n}}^{m,1}|^{p}]\leqslant C\frac{1}{2^{\frac{np}{2}}},
\]
and thus 
\[
\|B_{t_{k-1}^{n},t_{k}^{n}}^{m,1}\|_{p}\leqslant C\frac{1}{2^{\frac{n}{2}}}.
\]
Also it is trivial to see that 
\[
B_{t_{k-1}^{n},t_{k}^{n}}^{m+1,1}-B_{t_{k-1}^{n},t_{k}^{n}}^{m,1}=(B_{t_{k}^{n}}-B_{t_{k-1}^{n}})-(B_{t_{k}^{n}}-B_{t_{k-1}^{n}})=0.
\]

If $n>m,$ then by construction we know that 
\[
B_{t_{k-1}^{n},t_{k}^{n}}^{m,1}=\frac{2^{m}}{2^{n}}(B_{t_{l}^{m}}-B_{t_{l-1}^{m}}),
\]
where $l$ is the unique integer such that $[t_{k-1}^{n},t_{k}^{n}]\subset[t_{l-1}^{m},t_{l}^{m}].$
Therefore, 
\[
\|B_{t_{k-1}^{n},t_{k}^{n}}^{m,1}\|_{p}=\frac{2^{m}}{2^{n}}\|B_{t_{l}^{m}}-B_{t_{l-1}^{m}}\|_{p}\leqslant C\frac{2^{\frac{m}{2}}}{2^{n}}.
\]
On the other hand, if $[t_{k-1}^{n},t_{k}^{n}]\subset[t_{2l-2}^{m+1},t_{2l-1}^{m+1}],$
then 
\begin{eqnarray*}
B_{t_{k-1}^{n},t_{k}^{n}}^{m+1,1}-B_{t_{k-1}^{n},t_{k}^{n}}^{m,1} & = & \frac{2^{m+1}}{2^{n}}(B_{t_{2l-1}^{m+1}}-B_{t_{2l-2}^{m+1}})-\frac{2^{m}}{2^{n}}(B_{t_{l}^{m}}-B_{t_{l-1}^{m}})\\
 & = & \frac{2^{m}}{2^{n}}((B_{\frac{2l-1}{2^{m+1}}}-B_{\frac{2l-2}{2^{m+1}}})-(B_{\frac{2l}{2^{m+1}}}-B_{\frac{2l-1}{2^{m+1}}})).
\end{eqnarray*}
It follows that 
\[
\|B_{t_{k-1}^{n},t_{k}^{n}}^{m+1,1}-B_{t_{k-1}^{n},t_{k}^{n}}^{m,1}\|_{p}\leqslant C\frac{2^{\frac{m}{2}}}{2^{n}}.
\]
Similarly, if $[t_{k-1}^{n},t_{k}^{n}]\subset[t_{2l-1}^{m+1},t_{2l}^{m+1}],$
we will obtain the same estimate.

(2) The second level.

Since $\frac{p}{2}<2$, by monotonicity it suffices to establish the
desired estimates under the $L^{2}$-norm. 

First consider the term $B_{t_{k-1}^{n},t_{k}^{n}}^{m+1,2}-B_{t_{k-1}^{n},t_{k}^{n}}^{m,2}.$ 

If $n\leqslant m,$ by the construction of $B_{s,t}^{m,2}$, we have
\begin{eqnarray*}
B_{t_{k-1}^{n},t_{k}^{n}}^{m,2;\alpha,\beta} & = & \int_{t_{k-1}^{n}<u<v<t_{k}^{n}}dB_{u}^{\alpha}dB_{v}^{\beta}\\
 & = & \int_{t_{k-1}^{n}}^{t_{k}^{n}}B_{t_{k-1}^{n},v}^{m,1;\alpha}dB_{v}^{m,1;\beta}\\
 & = & \sum_{l=2^{(m-n)}(k-1)+1}^{2^{m-n}k}\frac{\Delta_{l}^{m}B^{\beta}}{\Delta t^{m}}\int_{t_{l-1}^{m}}^{t_{l}^{m}}(\frac{v-t_{l-1}^{m}}{\Delta t^{m}}B_{t_{l}^{m}}^{\alpha}+\frac{t_{l}^{m}-v}{\Delta t^{m}}B_{t_{l-1}^{m}}^{\alpha}-B_{t_{k-1}^{n}}^{\alpha})dv\\
 & = & \sum_{l=2^{(m-n)}(k-1)+1}^{2^{m-n}k}(\frac{B_{t_{l-1}^{m}}^{\alpha}+B_{t_{l}^{m}}^{\alpha}}{2}-B_{t_{k-1}^{n}}^{\alpha})\Delta_{l}^{m}B^{\beta}.
\end{eqnarray*}
Therefore, 
\begin{eqnarray*}
 &  & B_{t_{k-1}^{n},t_{k}^{n}}^{m+1,2;\alpha,\beta}-B_{t_{k-1}^{n},t_{k}^{n}}^{m,2;\alpha,\beta}\\
 & = & \sum_{l=2^{(m+1-n)}(k-1)+1}^{2^{m+1-n}k}(\frac{B_{t_{l-1}^{m+1}}^{\alpha}+B_{t_{l}^{m+1}}^{\alpha}}{2}-B_{t_{k-1}^{n}}^{\alpha})\Delta_{l}^{m}B^{\beta}\\
 &  & -\sum_{l=2^{(m-n)}(k-1)+1}^{2^{m-n}k}(\frac{B_{t_{l-1}^{m}}^{\alpha}+B_{t_{l}^{m}}^{\alpha}}{2}-B_{t_{k-1}^{n}}^{\alpha})\Delta_{l}^{m}B^{\beta}
\end{eqnarray*}
\begin{eqnarray*}
 & = & \sum_{l=2^{(m-n)}(k-1)+1}^{2^{(m-n)}k}((\frac{B_{t_{2l-2}^{m+1}}^{\alpha}+B_{t_{2l-1}^{m+1}}^{\alpha}}{2}-B_{t_{k-1}^{n}}^{\alpha})\Delta_{2l-1}^{m+1}B^{\beta}+(\frac{B_{t_{2l-1}^{m+1}}^{\alpha}+B_{t_{2l}^{m+1}}^{\alpha}}{2}-B_{t_{k-1}^{n}}^{\alpha})\Delta_{2l}^{m+1}B^{\beta}\\
 &  & -(\frac{B_{t_{2l-2}^{m+1}}^{\alpha}+B_{t_{2l}^{m+1}}^{\alpha}}{2}-B_{t_{k-1}^{n}}^{\alpha})(\Delta_{2l-1}^{m+1}B^{\beta}+\Delta_{2l}^{m+1}B^{\beta}))\\
 & = & \frac{1}{2}\sum_{l=2^{m-n}(k-1)+1}^{2^{m-n}k}(\Delta_{2l-1}^{m+1}B^{\alpha}\Delta_{2l}^{m+1}B^{\beta}-\Delta_{2l}^{m+1}B^{\alpha}\Delta_{2l-1}^{m+1}B^{\beta}).
\end{eqnarray*}
By using the notation of tensor products, we have 
\[
B_{t_{k-1}^{n},t_{k}^{n}}^{m+1,2}-B_{t_{k-1}^{n},t_{k}^{n}}^{m,2}=\frac{1}{2}\sum_{l=2^{m-n}(k-1)+1}^{2^{m-n}k}(\Delta_{2l-1}^{m+1}B\otimes\Delta_{2l}^{m+1}B-\Delta_{2l}^{m+1}B\otimes\Delta_{2l-1}^{m+1}B).
\]
It follows that 
\begin{eqnarray*}
 &  & \mathbb{E}^{G}[|B_{t_{k-1}^{n},t_{k}^{n}}^{m+1,2}-B_{t_{k-1}^{n},t_{k}^{n}}^{m,2}|^{2}]\\
 & = & \frac{1}{4}\mathbb{E}^{G}[|\sum_{l=2^{m-n}(k-1)+1}^{2^{m-n}k}(\Delta_{2l-1}^{m+1}B\otimes\Delta_{2l}^{m+1}B-\Delta_{2l}^{m+1}B\otimes\Delta_{2l-1}^{m+1}B)|^{2}]\\
 & \leqslant & C\sum_{\substack{\alpha\neq\beta\\
\alpha,\beta=1,\cdots,d
}
}\mathbb{E}^{G}[|\sum_{l}(\Delta_{2l-1}^{m+1}B^{\alpha}\Delta_{2l}^{m+1}B^{\beta}-\Delta_{2l}^{m+1}B^{\alpha}\Delta_{2l-1}^{m+1}B^{\beta})|^{2}]\\
 & \leqslant & C\sum_{\alpha\neq\beta}\sum_{l,r}\mathbb{E}^{G}[(\Delta_{2l-1}^{m+1}B^{\alpha}\Delta_{2l}^{m+1}B^{\beta}-\Delta_{2l}^{m+1}B^{\alpha}\Delta_{2l-1}^{m+1}B^{\beta})\\
 &  & \cdot(\Delta_{2r-1}^{m+1}B^{\alpha}\Delta_{2r}^{m+1}B^{\beta}-\Delta_{2r}^{m+1}B^{\alpha}\Delta_{2r-1}^{m+1}B^{\beta})]\\
 & \leqslant & C\sum_{\alpha\neq\beta}\sum_{l,r}(\mathbb{E}^{G}[\Delta_{2l-1}^{m+1}B^{\alpha}\Delta_{2r-1}^{m+1}B^{\alpha}\Delta_{2l}^{m+1}B^{\beta}\Delta_{2r}^{m+1}B^{\beta}]\\
 &  & +\mathbb{E}^{G}[\Delta_{2l}^{m+1}B^{\alpha}\Delta_{2r}^{m+1}B^{\alpha}\Delta_{2l-1}^{m+1}B^{\beta}\Delta_{2r-1}^{m+1}B^{\beta}]+\mathbb{E}^{G}[-\Delta_{2l-1}^{m+1}B^{\alpha}\Delta_{2r}^{m+1}B^{\alpha}\Delta_{2r-1}^{m+1}B^{\beta}\Delta_{2l}^{m+1}B^{\beta}]\\
 &  & +\mathbb{E}^{G}[-\Delta_{2r-1}^{m+1}B^{\alpha}\Delta_{2l}^{m+1}B^{\alpha}\Delta_{2l-1}^{m+1}B^{\beta}\Delta_{2r}^{m+1}B^{\beta}]),
\end{eqnarray*}
where the summation over $l$ and $r$ is taken from $2^{m-n}(k-1)+1$
to $2^{m-n}k.$ Here we have used the sublinearity of $\mathbb{E}$.
Now we study every term separately. If $l<r,$ by the properties of
conditional $G$-expectation and the distribution of $B_{t}$, we
have 
\begin{eqnarray*}
 &  & \mathbb{E}^{G}[\Delta_{2l-1}^{m+1}B^{\alpha}\Delta_{2l}^{m+1}B^{\beta}\Delta_{2r-1}^{m+1}B^{\alpha}\Delta_{2r}^{m+1}B^{\beta}]\\
 & = & \mathbb{E}^{G}[\mathbb{E}[\Delta_{2l-1}^{m+1}B^{\alpha}\Delta_{2l}^{m+1}B^{\beta}\Delta_{2r-1}^{m+1}B^{\alpha}\Delta_{2r}^{m+1}B^{\beta}|\Omega_{t_{2r-1}^{m+1}}]]\\
 & = & \mathbb{E}^{G}[\eta^{+}\mathbb{E}[\Delta_{2r}^{m+1}B^{\beta}|\Omega_{t_{2r-1}^{m+1}}]+\eta^{-}\mathbb{E}[-\Delta_{2r}^{m+1}B^{\beta}|\Omega_{t_{2r-1}^{m+1}}]]\\
 & = & 0,
\end{eqnarray*}
where $\eta=\Delta_{2l-1}^{m+1}B^{\alpha}\Delta_{2l}^{m+1}B^{\beta}\Delta_{2r-1}^{m+1}B^{\alpha}.$
Similarly, we can prove that for any $l\neq r,$ 
\begin{eqnarray*}
\mathbb{E}^{G}[\Delta_{2l-1}^{m+1}B^{\alpha}\Delta_{2r-1}^{m+1}B^{\beta}\Delta_{2l}^{m+1}B^{\beta}\Delta_{2r}^{m+1}B^{\beta}] & = & \mathbb{E}^{G}[(\Delta_{2l}^{m+1}B^{\alpha}\Delta_{2r}^{m+1}B^{\alpha}\Delta_{2l-1}^{m+1}B^{\beta}\Delta_{2r-1}^{m+1}B^{\beta})]\\
 & = & \mathbb{E}^{G}[(-\Delta_{2l-1}^{m+1}B^{\alpha}\Delta_{2r}^{m+1}B^{\alpha}\Delta_{2r-1}^{m+1}B^{\beta}\Delta_{2l}^{m+1}B^{\beta})]\\
 & = & \mathbb{E}^{G}[(-\Delta_{2r-1}^{m+1}B^{\alpha}\Delta_{2l}^{m+1}B^{\alpha}\Delta_{2l-1}^{m+1}B^{\beta}\Delta_{2r}^{m+1}B^{\beta})]\\
 & = & 0.
\end{eqnarray*}
On the other hand, if $l=r,$ it is straight forward that 
\begin{eqnarray*}
\mathbb{E}^{G}[(\Delta_{2l-1}^{m+1}B^{\alpha})^{2}(\Delta_{2l}^{m+1}B^{\beta})^{2}] & \leqslant & \frac{1}{2}(\mathbb{E}^{G}[(\Delta_{2l-1}^{m+1}B^{\alpha})^{4}]+\mathbb{E}^{G}[(\Delta_{2l}^{m+1}B^{\beta})^{4}])\leqslant C\frac{1}{2^{2m}},
\end{eqnarray*}
and similarly, 
\begin{eqnarray*}
\mathbb{E}^{G}(-\Delta_{2l-1}^{m+1}B^{\alpha}\Delta_{2l-1}^{m+1}B^{\beta}\Delta_{2l}^{m+1}B^{\alpha}\Delta_{2l}^{m+1}B^{\beta}) & \leqslant & \frac{1}{4}(\mathbb{E}^{G}[(\Delta_{2l-1}^{m+1}B^{\alpha})^{4}]+\mathbb{E}^{G}[(\Delta_{2l-1}^{m+1}B^{\beta})^{4}]\\
 &  & +\mathbb{E}^{G}[(\Delta_{2l}^{m+1}B^{\alpha})^{4}]+\mathbb{E}^{G}[(\Delta_{2l}^{m+1}B^{\beta})^{4}])\\
 & \leqslant & C\frac{1}{2^{2m}}.
\end{eqnarray*}
Combining all the estimates above, we arrive at
\begin{eqnarray*}
\mathbb{E}^{G}[|B_{t_{k-1}^{n},t_{k}^{n}}^{m+1,2}-B_{t_{k-1}^{n},t_{k}^{n}}^{m,2}|^{2}] & \leqslant & C\sum_{\alpha\neq\beta}\sum_{l=2^{m-n}(k-1)+1}^{2^{m-n}k}\frac{1}{2^{2m}}\leqslant C\frac{1}{2^{m}2^{n}},
\end{eqnarray*}
and hence 
\[
\|B_{t_{k-1}^{n},t_{k}^{n}}^{m+1,2}-B_{t_{k-1}^{n},t_{k}^{n}}^{m,2}\|_{2}\leqslant C\frac{1}{2^{\frac{m}{2}}2^{\frac{n}{2}}}.
\]

If $n>m,$ by construction we have 
\begin{eqnarray*}
B_{t_{k-1}^{n},t_{k}^{n}}^{m,2;\alpha,\beta} & = & \int_{t_{k-1}^{n}<u<v<t_{k}^{n}}d(B^{m})_{u}^{\alpha}d(B^{m})_{v}^{\beta}\\
 & = & \int_{t_{k-1}^{n}}^{t_{k}^{n}}B_{t_{k-1}^{n},v}^{m,1;\alpha}d(B^{m})_{v}^{\beta}\\
 & = & \frac{\Delta_{l}^{m}B^{\alpha}\Delta_{l}^{m}B^{\beta}}{(\Delta t^{m})^{2}}\int_{t_{k-1}^{n}}^{t_{k}^{n}}(v-t_{k-1}^{n})dv\\
 & = & \frac{1}{2}2^{2(m-n)}\Delta_{l}^{m}B^{\alpha}\Delta_{l}^{m}B^{\beta},
\end{eqnarray*}
where $l$ is the unique integer such that $[t_{k-1}^{n},t_{k}^{n}]\subset[t_{l-1}^{m},t_{l}^{m}]$.
In other words, we have

\[
B_{t_{k-1}^{n},t_{k}^{n}}^{m,2}=\frac{1}{2}2^{2(m-n)}(\Delta_{l}^{m}B)^{\otimes2},
\]
 It follows that
\[
B_{t_{k-1}^{n},t_{k}^{n}}^{m+1,2}-B_{t_{k-1}^{n},t_{k}^{n}}^{m,2}=\begin{cases}
2^{2(m-n)+1}(\Delta_{2l-1}^{m+1}B)^{\otimes2}-2^{2(m-n)-1}(\Delta_{l}^{m}B)^{\otimes2}, & [t_{k-1}^{n},t_{k}^{n}]\subset[t_{2l-2}^{m+1},t_{2l-1}^{m+1}];\\
2^{2(m-n)+1}(\Delta_{2l}^{m+1}B)^{\otimes2}-2^{2(m-n)-1}(\Delta_{l}^{m}B)^{\otimes2}, & [t_{k-1}^{n},t_{k}^{n}]\subset[t_{2l-1}^{m+1},t_{2l}^{m+1}].
\end{cases}
\]
By using the Minkowski inequality, the Cauchy-Schwarz inequality and
the sublinearity of $\mathbb{E},$ it is easy to obtain that 
\[
\|B_{t_{k-1}^{n},t_{k}^{n}}^{m+1,2}-B_{t_{k-1}^{n},t_{k}^{n}}^{m,2}\|_{2}\leqslant C\frac{2^{m}}{2^{2n}}.
\]

Now consider the term $B_{t_{k-1}^{n},t_{k}^{n}}^{m,2}.$ 

If $n\geqslant m,$ by using
\[
B_{t_{k-1}^{n},t_{k}^{n}}^{m,2}=2^{2(m-n)-1}(\Delta_{l}^{m}B)^{\otimes2},
\]
we can proceed in the same way as before to obtain that 
\[
\|B_{t_{k-1}^{n},t_{k}^{n}}^{m,2}\|_{2}\leqslant C\frac{2^{m}}{2^{2n}}.
\]

If $n<m,$ then
\[
B_{t_{k-1}^{n},t_{k}^{n}}^{m,2}=\sum_{l=n+1}^{m}(B_{t_{k-1}^{n},t_{k}^{n}}^{l,2}-B_{t_{k-1}^{n},t_{k}^{n}}^{l-1,2})+B_{t_{k-1}^{n},t_{k}^{n}}^{n,2}.
\]
It follows that 
\begin{eqnarray*}
\|B_{t_{k-1}^{n},t_{k}^{n}}^{m,2}\|_{2} & \leqslant & \sum_{l=n+1}^{m}\|B_{t_{k-1}^{n},t_{k}^{n}}^{l,2}-B_{t_{k-1}^{n},t_{k}^{n}}^{l-1,2}\|_{2}+\|B_{t_{k-1}^{n},t_{k}^{n}}^{n,2}\|_{2}\\
 & \leqslant & C(\frac{1}{2^{\frac{n}{2}}}\sum_{l=n+1}^{\infty}\frac{1}{2^{\frac{l}{2}}}+\frac{1}{2^{n}})\\
 & \leqslant & C\frac{1}{2^{n}}.
\end{eqnarray*}

Now the proof is complete.
\end{proof}

In order to study the behavior of $\boldsymbol{B}^{m}$ in the space
$G\Omega_{p}(\mathbb{R}^{d})$, we may need to control the $p$-variation
distance $d_{p}$ in a suitable way. For $\boldsymbol{w},\widetilde{\boldsymbol{w}}\in G\Omega(\mathbb{R}^{d}),$
define 
\begin{equation}
\rho_{j}(\boldsymbol{w},\widetilde{\boldsymbol{w}}):=(\sum_{n=1}^{\infty}n^{\gamma}\sum_{k=1}^{2^{n}}|w_{t_{k-1}^{n},t_{k}^{n}}^{j}-\widetilde{w}_{t_{k-1}^{n},t_{k}^{n}}^{j}|^{\frac{p}{j}})^{\frac{j}{p}},\ j=1,2,\label{def rho}
\end{equation}
where $\gamma>p-1$ is some fixed universal constant. The functional
$\rho_{j}$ was initially introduced by Hambly and Lyons \cite{hambly1998stochastic}
to construct the stochastic area process associated with the Brownian
motion on the Sierpinski gasket. We use $\rho_{j}(\boldsymbol{w})$
to denote $\rho_{j}(\boldsymbol{w},\widetilde{\boldsymbol{w}})$ with
$\widetilde{\boldsymbol{w}}=(1,0,0).$

The following result, which is important for us, is proved in \cite{lyons2002system}. 
\begin{prop}
\label{p-var control}There exists some positive constant $R=R(p,\gamma),$
such that for any $\boldsymbol{w},\widetilde{\boldsymbol{w}}\in G\Omega(\mathbb{R}^{d})$,
\[
d_{p}(\boldsymbol{w},\widetilde{\boldsymbol{w}})\leqslant R\mbox{max\ensuremath{\{\rho_{1}(\boldsymbol{w},\widetilde{\boldsymbol{w}}),\rho_{1}(\boldsymbol{w},\widetilde{\boldsymbol{w}})(\rho_{1}(\boldsymbol{w})+\rho_{1}(\widetilde{\boldsymbol{w}})),\rho_{2}(\boldsymbol{w},\widetilde{\boldsymbol{w}})\}.}}
\]

\end{prop}

Now let 
\begin{equation}
I(\boldsymbol{w},\widetilde{\boldsymbol{w}}):=\mbox{max}\ensuremath{\{\rho_{1}(\boldsymbol{w},\widetilde{\boldsymbol{w}}),\rho_{1}(\boldsymbol{w},\widetilde{\boldsymbol{w}})(\rho_{1}(\boldsymbol{w})+\rho_{1}(\widetilde{\boldsymbol{w}})),\rho_{2}(\boldsymbol{w},\widetilde{\boldsymbol{w}})\},}\label{rho}
\end{equation}
and observe that
\begin{eqnarray}
\{\omega:\ \boldsymbol{B}^{m}\mbox{ is not Cauchy under \ensuremath{d_{p}}}\} & \subset & \{\omega:\ \sum_{m=1}^{\infty}d_{p}(\boldsymbol{B}^{m},\boldsymbol{B}^{m+1})=\infty\}\nonumber \\
 & \subset & \limsup_{m\rightarrow\infty}\{\omega:\ d_{p}(\boldsymbol{B}^{m},\boldsymbol{B}^{m+1})>\frac{R}{2^{m\beta}}\}\nonumber \\
 & \subset & \limsup_{m\rightarrow\infty}\{\omega:\ I(\boldsymbol{B}^{m},\boldsymbol{B}^{m+1})>\frac{1}{2^{m\beta}}\}.\label{not Cauchy}
\end{eqnarray}
where $\beta$ is some positive constant to be chosen. Notice that
the R.H.S. of (\ref{not Cauchy}) is $\mathcal{B}(\Omega)$-measurable
so its capacity is well-defined. Therefore, in order to prove that
for quasi-surely, $\boldsymbol{B}^{m}$ is a Cauchy sequence under
$d_{p},$ it suffices to show that the R.H.S. of (\ref{not Cauchy})
has capacity zero. This can be shown by using the Borel-Cantelli lemma.

According to (\ref{rho}), we may first need to establish estimates
for 
\[
c(\rho_{j}(\boldsymbol{B}^{m},\boldsymbol{B}^{m+1})>\lambda),\ j=1,2,
\]
and 
\[
c(\rho_{1}(\boldsymbol{B}^{m})>\lambda),
\]
where $m\geqslant1$ and $\lambda>0.$ They are contained in the following
lemma.
\begin{lem}
\label{estimate of rho}For $m\geqslant1,$ $\lambda>0$, we have
the following estimates.

(1) 
\[
c(\rho_{1}(\boldsymbol{B}^{m})>\lambda)\leqslant C\lambda^{-p}.
\]

(2) Let \textup{$\theta\in(0,\frac{p}{2}-1)$ be some constant such
that 
\[
n^{\gamma+1}\leqslant C\frac{2^{n(p-1)}}{2^{n(p-\theta-1)}},\ \forall n\geqslant1.
\]
Then we have
\[
c(\rho_{j}(\boldsymbol{B}^{m},\boldsymbol{B}^{m+1})>\lambda)\leqslant C\lambda^{-\frac{p}{j}}\frac{1}{2^{m(\frac{p}{2}-\theta-1)}},\ j=1,2.
\]
}\end{lem}
\begin{proof}
First consider 
\[
c(\rho_{1}(\boldsymbol{B}^{m})>\lambda)=c(\sum_{n=1}^{\infty}n^{\gamma}\sum_{k=1}^{2^{n}}|B_{t_{k-1}^{n},t_{k}^{n}}^{m,1}|^{p}>\lambda^{p}).
\]
Define 
\[
A_{N}=\{\omega:\ \sum_{n=1}^{N}n^{\gamma}\sum_{k=1}^{2^{n}}|B_{t_{k-1}^{n},t_{k}^{n}}^{m,1}|^{p}>\lambda^{p}\}\in\mathcal{B}(\Omega),
\]
and 
\[
A=\{\omega:\ \sum_{n=1}^{\infty}n^{\gamma}\sum_{k=1}^{2^{n}}|B_{t_{k-1}^{n},t_{k}^{n}}^{m,1}|^{p}>\lambda^{p}\}\in\mathcal{B}(\Omega).
\]
It is obvious that $A_{N}\uparrow A.$ By the properties of the capacity
$c,$ we have 
\[
c(A)=\lim_{N\rightarrow\infty}c(A_{N}).
\]
On the other hand, by the sublinearity of $\mathbb{E}^{G}$, the Chebyshev
inequality for the capacity $c$ and Lemma \ref{moment inequality},
we have 
\begin{eqnarray*}
c(A_{N}) & \leqslant & \lambda^{-p}\sum_{n=1}^{N}n^{\gamma}\sum_{k=1}^{2^{n}}\mathbb{E}[|B_{t_{k-1}^{n},t_{k}^{n}}^{m,1}|^{p}]\\
 & \leqslant & C\lambda^{-p}[\sum_{n=1}^{m}n^{\gamma}2^{n}\frac{1}{2^{\frac{np}{2}}}+\sum_{n=m+1}^{\infty}n^{\gamma}2^{n}\frac{2^{\frac{mp}{2}}}{2^{np}}]\\
 & = & C\lambda^{-p}[\sum_{n=1}^{m}n^{\gamma}\frac{1}{2^{n(\frac{p}{2}-1)}}+2^{\frac{mp}{2}}\sum_{n=m+1}^{\infty}n^{\gamma}\frac{1}{2^{n(p-1)}}]\\
 & \leqslant & C\lambda^{-p}.
\end{eqnarray*}
It follows that
\[
c(\rho_{1}(\boldsymbol{B}^{m})>\lambda)=c(A)\leqslant C\lambda^{-p}.
\]

Now consider 
\[
c(\rho_{1}(\boldsymbol{B}^{m},\boldsymbol{B}^{m+1})>\lambda)=c(\sum_{n=1}^{\infty}n^{\gamma}\sum_{k=1}^{2^{n}}|B_{\frac{k-1}{2^{n}},\frac{k}{2^{n}}}^{(m+1),1}-B_{\frac{k-1}{2^{n}},\frac{k}{2^{n}}}^{(m),1}|^{p}>\lambda^{p}).
\]
By similar reasons we will have 
\begin{eqnarray*}
c(\rho_{1}(\boldsymbol{B}^{m},\boldsymbol{B}^{m+1})>\lambda) & \leqslant & \lambda^{-p}\sum_{n=1}^{\infty}n^{\gamma}\sum_{k=1}^{2^{n}}\mathbb{E}[|B_{t_{k-1}^{n},t_{k}^{n}}^{m+1,1}-B_{t_{k-1}^{n},t_{k}^{n}}^{m,1}|^{p}]\\
 & \leqslant & C\lambda^{-p}(\sum_{n=m+1}^{\infty}n^{\gamma}2^{n}\frac{2^{\frac{mp}{2}}}{2^{np}})\\
 & = & C\lambda^{-p}2^{\frac{mp}{2}}\sum_{n=m+1}^{\infty}n^{\gamma}\frac{1}{2^{n(p-1)}}.
\end{eqnarray*}
Since $\theta\in(0,\frac{p}{2}-1)$ is such that 
\[
n^{\gamma+1}\leqslant C\frac{2^{n(p-1)}}{2^{n(p-\theta-1)}},\ \forall n\geqslant1,
\]
we arrive at 
\[
c(\rho_{1}(\boldsymbol{B}^{m},\boldsymbol{B}^{m+1})>\lambda)\leqslant C\lambda^{-p}\frac{1}{2^{m(\frac{p}{2}-\theta-1)}}.
\]

Finally, consider the second level part. By similar reasons, we have
\begin{eqnarray*}
c(\rho_{2}(\boldsymbol{B}^{m},\boldsymbol{B}^{m+1})>\lambda) & \leqslant & C\lambda^{-\frac{p}{2}}[\sum_{n=1}^{m}n^{\gamma}2^{n}\frac{1}{2^{\frac{mp}{4}}2^{\frac{np}{4}}}+2^{\frac{mp}{2}}\sum_{n=m+1}^{\infty}n^{\gamma}2^{n}\frac{1}{2^{np}}]\\
 & = & C\lambda^{-\frac{p}{2}}[\frac{1}{2^{\frac{mp}{4}}}\sum_{n=1}^{m}n^{\gamma}2^{n(1-\frac{p}{4})}+2^{\frac{mp}{2}}\sum_{n=m+1}^{\infty}n^{\gamma}\frac{1}{2^{n(p-1)}}]\\
 & \leqslant & C\lambda^{-\frac{p}{2}}[\frac{1}{2^{\frac{mp}{4}}}m^{\gamma+1}2^{m(1-\frac{p}{4})}+2^{\frac{mp}{2}}\frac{1}{2^{m(p-\theta-1)}}]\\
 & \leqslant & C\lambda^{-\frac{p}{2}}\frac{1}{2^{m(\frac{p}{2}-\theta-1)}}.
\end{eqnarray*}

\end{proof}

Now we are in position to prove the main result of this section.
\begin{thm}
\label{enhancement}Outside a $\mathcal{B}(\Omega)$-measurable set
of capacity zero, $\boldsymbol{B}^{m}$ is a Cauchy sequence under
the $p$-variation distance $d_{p}.$ In particular, for quasi-surely,
the sample paths of $B_{t}$ can be enhanced to be geometric rough
paths 
\[
\boldsymbol{B}_{s,t}=(1,B_{s,t}^{1},B_{s,t}^{2}),\ 0\leqslant s<t\leqslant1,
\]
with roughness $p,$ which are defined as the limit of sample (geometric
rough) paths of $\boldsymbol{B}^{m}$ in $G\Omega_{p}(\mathbb{R}^{d})$
under the $p$-variation distance $d_{p}.$\end{thm}
\begin{proof}
By Lemma \ref{estimate of rho}, we have 
\begin{eqnarray*}
c(I(\boldsymbol{B}^{m},\boldsymbol{B}^{m+1})>\frac{1}{2^{m\beta}}) & \leqslant & \sum_{j=1}^{2}c(\rho_{j}(\boldsymbol{B}^{m},\boldsymbol{B}^{m+1})>\frac{1}{2^{m\beta}})\\
 &  & +c(\rho_{1}(\boldsymbol{B}^{m},\boldsymbol{B}^{m+1})(\rho_{1}(\boldsymbol{B}^{m})+\rho_{1}(\boldsymbol{B}^{m+1}))>\frac{1}{2^{m\beta}})\\
 & \leqslant & 2c(\rho_{1}(\boldsymbol{B}^{m},\boldsymbol{B}^{m+1})>\frac{1}{2^{2m\beta}})+c(\rho_{2}(\boldsymbol{B}^{m},\boldsymbol{B}^{m+1})>\frac{1}{2^{m\beta}})\\
 &  & +c(\rho_{1}(\boldsymbol{B}^{m})>\frac{2^{m\beta}}{2})+c(\rho_{1}(\boldsymbol{B}^{m+1})>\frac{2^{m\beta}}{2})\\
 & \leqslant & C[\frac{1}{2^{m\beta p}}+\frac{1}{2^{m(\frac{p}{2}-\theta-2\beta p-1)}}+\frac{1}{2^{m(\frac{p}{2}-\theta-\frac{\beta p}{2}-1)}}],
\end{eqnarray*}
where $\theta\in(0,\frac{p}{2}-1)$ is some fixed constant. 

If we choose $\beta$ such that 
\[
0<\beta<\frac{p-2\theta-2}{4p},
\]
then 
\[
\sum_{m=1}^{\infty}c(I(\boldsymbol{B}^{m},\boldsymbol{B}^{m+1})>\frac{1}{2^{m\beta}})<\infty.
\]
By the Borel-Cantelli lemma, we have 
\[
c(\limsup_{m\rightarrow\infty}\{\omega:\ I(\boldsymbol{B}^{m},\boldsymbol{B}^{m+1})>\frac{1}{2^{m\beta}}\})=0,
\]
and the result follows from the inclusion (\ref{not Cauchy}).
\end{proof}

With the help of Theorem \ref{enhancement} and the smoothness of
$\langle B^{\alpha},B^{\beta}\rangle_{t}$ (by definition the sample
paths of $\langle B^{\alpha},B^{\beta}\rangle_{t}$ are smooth), we
are able to apply the universal limit theorem in rough path theory
to define RDEs driven by $G$-Brownian motion in the pathwise sense.
More precisely, consider the following $N$-dimensional RDE in the
sense of rough paths:

\begin{equation}
dY_{t}=\widetilde{b}(Y_{t})dt+\widetilde{h}_{\alpha\beta}(Y_{t})d\langle B^{\alpha},B^{\beta}\rangle_{t}+V_{\alpha}(Y_{t})dB_{t}^{\alpha},\label{RDE}
\end{equation}
with initial condition $Y_{0}=x,$ where $\widetilde{b},\widetilde{h}_{\alpha\beta},V_{\alpha}$
are $C_{b}^{3}$-vector fields on $\mathbb{R}^{N}$. Then outside
a $\mathcal{B}(\Omega)$-measurable set of capacity zero, (\ref{RDE})
has a unique full solution $\boldsymbol{Y}$ in $G\Omega_{p}(\mathbb{R}^{N})$.
$\boldsymbol{Y}$ is constructed as the limit of the enhancement of
$Y_{t}^{n}$ in $G\Omega_{p}(\mathbb{R}^{N})$ under the $p$-variation
distance, where $Y_{t}^{n}$ is the unique classical solution of the
following ordinary differential equation:
\begin{eqnarray}
dY_{t}^{n} & = & \widetilde{b}(Y_{t}^{n})dt+\widetilde{h}_{\alpha\beta}(Y_{t}^{n})d\langle B^{\alpha},B^{\beta}\rangle_{t}+V_{\alpha}(Y_{t}^{n})d(B^{n})_{t}^{\alpha},\label{approximated ODE}
\end{eqnarray}
with $Y_{0}^{n}=x,$ in which $B_{t}^{n}$ is the dyadic piecewise
linear approximation of $B_{t}.$ 

If we only consider solutions instead of full solutions (i.e., only
consider the first level), then for quasi-surely, (\ref{RDE}) has
a unique solution $Y_{t}\in C([0,1];\mathbb{R}^{N})$, which is constructed
as the uniform limit of the solution of (\ref{approximated ODE})
with initial condition $Y_{0}^{n}=x.$

Before the end of this section, we are going to give an explicit description
of the second level $B_{s,t}^{2}$ of $B_{t}$ defined in Theorem
\ref{enhancement}, which reveals the nature of $B_{s,t}^{2}$ itself.
Such result is fundamental to understand the relation between  SDEs
and RDEs driven by $G$-Brownian motion.
\begin{lem}
\label{same limit}Assume that $X_{n}$ converges to $X$ in $L_{G}^{2}(\Omega)$
and converges to $Y$ quasi-surely. Then for quasi-surely, $X=Y.$\end{lem}
\begin{proof}
By the Chebyshev inequality for the capacity, we have 
\[
c(|X_{n}-X|>\epsilon)\leqslant\frac{1}{\epsilon^{2}}\mathbb{E}^{G}[|X_{n}-X|^{2}],\ \forall\epsilon>0.
\]
Since 
\[
X_{n}\rightarrow X\ \ \ \mbox{in \ensuremath{L_{G}^{2}(\Omega),}}
\]
we can extract a subsequence $X_{n_{k}},$ such that for any $k\geqslant1,$
\[
\mathbb{E}^{G}[|X_{n_{k}}-X|^{2}]\leqslant\frac{1}{k^{4}}.
\]
 It follows that 
\[
c(|X_{n_{k}}-X|>\frac{1}{k})\leqslant\frac{1}{k^{2}},\ \ \ \forall k\geqslant1,
\]
and 
\[
\sum_{k=1}^{\infty}c(|X_{n_{k}}-X|>\frac{1}{k})<\infty.
\]
By the Borel-Cantelli lemma for the capacity, we arrive at for quasi-surely,
$X_{n_{k}}$ converges to $X.$ By assumption it follows that for
quasi-surely, $X=Y.$
\end{proof}

The following result shows the nature of the second level of $B_{t}.$
In the case when $B_{t}$ reduces to the classical Brownian motion,
it is essentially the relation between Stratonovich and It$\hat{\mbox{o}}$
integrals. 
\begin{prop}
\label{representation of second level}Let $\boldsymbol{B}_{s,t}=(1,B_{s,t}^{1},B_{s,t}^{2})$
be the quasi-surely defined enhancement of $B_{t}$ in Theorem \ref{enhancement}.
Then for any $0\leqslant s<t\leqslant1,$ for quasi-surely, we have
\begin{equation}
B_{s,t}^{2;\alpha,\beta}=\int_{s}^{t}B_{s,u}^{\alpha}dB_{u}^{\beta}+\frac{1}{2}\langle B^{\alpha},B^{\beta}\rangle_{s,t},\label{second level}
\end{equation}
where the integral on the R.H.S. of (\ref{second level}) is the It$\hat{\mbox{o}}$
integral.\end{prop}
\begin{proof}
We know from Theorem \ref{enhancement} that for quasi-surely, 
\[
\lim_{n\rightarrow\infty}d_{p}(\boldsymbol{B}^{n},\boldsymbol{B})=0.
\]
From the definition of $d_{p},$ it is straight forward that for quasi-surely,
$B_{s,t}^{n,2}$ converges uniformly to $B_{s,t}^{2}$. 

Without lost of generality, we assume that $s,t$ are both dyadic
points in $[0,1]$. It follows that when $n$ is large enough, 
\begin{eqnarray*}
B_{s,t}^{n,2;\alpha,\beta} & = & \int_{s<u<v<t}d(B^{n})_{u}^{\alpha}d(B^{n})_{v}^{\beta}\\
 & = & \int_{s}^{t}(B^{n})_{s,v}^{\alpha}d(B^{n})_{v}^{\beta}\\
 & = & \sum_{k:[t_{k-1}^{n},t_{k}^{n}]\subset[s,t]}\frac{\Delta_{k}^{n}B^{\beta}}{\Delta t^{n}}\int_{t_{k-1}^{n}}^{t_{k}^{n}}(\frac{v-t_{k-1}^{n}}{\Delta t^{n}}B_{k}^{\alpha}+\frac{t_{k}^{n}-v}{\Delta t^{n}}B_{k-1}^{\alpha}-B_{s}^{\alpha})dv\\
 & = & \sum_{k:[t_{k-1}^{n},t_{k}^{n}]\subset[s,t]}(\frac{B_{k-1}^{\alpha}+B_{k}^{\alpha}}{2}-B_{s}^{\alpha})\Delta_{k}^{n}B^{\beta}\\
 & = & \sum_{k:[t_{k-1}^{n},t_{k}^{n}]\subset[s,t]}(B_{k-1}^{\alpha}-B_{s}^{\alpha})\Delta_{k}^{n}B^{\beta}+\frac{1}{2}\sum_{k}\Delta_{k}^{n}B^{\alpha}\Delta_{k}^{n}B^{\beta}.
\end{eqnarray*}
From properties of It$\hat{\mbox{o}}$ integral and the cross-variation
$\langle B^{\alpha},B^{\beta}\rangle_{t}$, we know that the R.H.S.
of the above equality converges to $\int_{s}^{t}B_{s,u}^{\alpha}dB_{u}^{\beta}+\frac{1}{2}\langle B^{\alpha},B^{\beta}\rangle_{s,t}$
in $L_{G}^{2}(\Omega).$ 

Consequently, by Lemma \ref{same limit} $B_{s,t}^{2}$ must coincide
with $\int_{s}^{t}B_{s,u}^{\alpha}dB_{u}^{\beta}+\frac{1}{2}\langle B^{\alpha},B^{\beta}\rangle_{s,t}$
quasi-surely.
\end{proof}

\section{The Relation between SDEs and RDEs Driven by $G$-Brownian Motion}

So far we already know that there are two types of well-defined differential
equations driven by $G$-Brownian motion: SDEs which are defined in
the $L_{G}^{2}$-sense with respect to the $G$-expectation $\mathbb{E}^{G},$
and RDEs which are quasi-surely defined in the pathwise sense. This
section is devoted to the study of the fundamental relation between
these two types of differential equations.

Consider the following $N$-dimensional  SDE driven by $G$-Brownian
motion on $(\Omega,L_{G}^{2}(\Omega),\mathbb{E}):$ 
\begin{equation}
dX_{t}=b(X_{t})dt+h_{\alpha\beta}(X_{t})d\langle B^{\alpha},B^{\beta}\rangle_{t}+V_{\alpha}(X_{t})dB_{t}^{\alpha},\label{Ito in section 4}
\end{equation}
with initial condition $X_{0}=x\in\mathbb{R}^{N}.$ Here we assume
that $b,h_{\alpha\beta},V_{\alpha}$ are $C_{b}^{3}$-vector fields
on $\mathbb{R}^{N}.$ 

Our aim is to find the correct RDE of the form (\ref{RDE}) whose
strong solution coincides with $X_{t}$ quasi-surely in the pathwise
sense.

Let's first illustrate the idea in an informal way. We are going to
use the rough Taylor expansion in the theory of RDEs (see Corollary
12.8 in \cite{friz2010multidimensional}) and Proposition \ref{representation of second level}
to find the correct form of the RDE we are looking for. 

Consider the following general RDE:
\begin{equation}
dY_{t}=\widetilde{b}(Y_{t})dt+\widetilde{h}_{\alpha\beta}(Y_{t})d\langle B^{\alpha},B^{\beta}\rangle_{t}+\widetilde{V}_{\alpha}(Y_{t})dB_{t}^{\alpha},\label{RDE tilde}
\end{equation}
with initial condition $Y_{0}=x,$ where $\widetilde{b},\widetilde{h}_{\alpha\beta},\widetilde{V}_{\alpha}$
are $C_{b}^{3}$-vector fields on $\mathbb{R}^{N}$. By the smoothness
of the cross variation process $\langle B^{\alpha},B^{\beta}\rangle$,
and the roughness of $B_{t}$ studied in the last section, we know
from the rough Taylor expansion theorem that for quasi-surely, for
some control function $\omega(s,t),$ the solution $Y_{t}$ of (\ref{RDE tilde})
satisfies, when $\omega(s,t)\leqslant1,$ 
\begin{equation}
|Y_{s,t}-\widetilde{b}(Y_{s})(t-s)-\widetilde{h}_{\alpha\beta}(Y_{s})\langle B^{\alpha},B^{\beta}\rangle_{s,t}-\widetilde{V}_{\alpha}(Y_{s})B_{s,t}^{1;\alpha}-D\widetilde{V}_{\beta}(Y_{s})\cdot\widetilde{V}_{\alpha}(Y_{s})B_{s,t}^{2;\alpha,\beta}|\leqslant C\omega(s,t)^{\theta},\label{rough Taylor expansion}
\end{equation}
where $\omega(s,t)$ $C$ and $\theta>1$ are two constants not depending
on $s,t.$ Note that inequality (\ref{rough Taylor expansion}) reveals
the local behavior of the solution $Y_{t}$. It follows from Proposition
\ref{representation of second level} that for quasi-surely,
\[
|Y_{s,t}-\widetilde{I}_{s,t}|\leqslant C\omega(s,t)^{\theta},
\]
where
\begin{eqnarray}
\widetilde{I}_{s,t}: & = & \widetilde{b}(Y_{s})(t-s)+(\widetilde{h}_{\alpha\beta}(Y_{s})+\frac{1}{2}D\widetilde{V}_{\beta}(Y_{s})\cdot\widetilde{V}_{\alpha}(Y_{s}))d\langle B^{\alpha},B^{\beta}\rangle_{t}+\widetilde{V}_{\alpha}(Y_{s})B_{s,t}^{1;\alpha}\label{rough Taylor with Ito}\\
 &  & +D\widetilde{V}_{\beta}(Y_{s})\cdot\widetilde{V}_{\alpha}(Y_{s})\int_{s}^{t}B_{s,u}^{\alpha}dB_{u}^{\beta}\nonumber 
\end{eqnarray}
Now if we consider the global behavior of $Y_{t},$ we may sum up
inequality (\ref{rough Taylor with Ito}) over dyadic intervals $[t_{k-1}^{n},t_{k}^{n}]$
and then take limit (in $L_{G}^{2}(\Omega;\mathbb{R}^{N})$) to obtain
that for quasi-surely,
\begin{eqnarray}
Y_{s,t} & = & \int_{s}^{t}\widetilde{b}(Y_{u})du+\int_{s}^{t}(\widetilde{h}_{\alpha\beta}(Y_{u})+\frac{1}{2}D\widetilde{V}_{\beta}(Y_{u})\cdot\widetilde{V}_{\alpha}(Y_{u}))d\langle B^{\alpha},B^{\beta}\rangle_{u}+\int_{s}^{t}\widetilde{V}_{\alpha}(Y_{u})dB_{u}^{\alpha}\nonumber \\
 &  & +(L_{G}^{2}-)\lim_{n\rightarrow\infty}\sum_{k:[t_{k-1}^{n},t_{k}^{n}]\subset[s,t]}D\widetilde{V}_{\alpha}(Y_{t_{k-1}^{n}})\cdot\widetilde{V}_{\beta}(Y_{t_{k-1}^{n}})\int_{t_{k-1}^{n}}^{t_{k}^{n}}B_{t_{k-1}^{n},u}^{\alpha}dB_{u}^{\beta},\label{representation of Y_s,t}
\end{eqnarray}
where the integrals with respect to $B_{t}$ are interpreted as It$\hat{\mbox{o}}$
integrals. On the other hand, by the distribution of $B_{t}$ and
properties of $G$-It$\hat{\mbox{o}}$ integral, it is not hard to
prove that the $L_{G}^{2}$-limit in the last term of the above identity
is zero. Therefore, we know that $Y_{t}$ solves the  SDE
\[
dX_{t}=\widetilde{b}(X_{t})dt+(\widetilde{h}_{\alpha\beta}(X_{t})+\frac{1}{2}D\widetilde{V}_{\beta}(X_{t})\cdot\widetilde{V}_{\alpha}(X_{t}))d\langle B^{\alpha},B^{\beta}\rangle_{t}+\widetilde{V}_{\alpha}(X_{t})dB_{t}^{\alpha}.
\]
In other words, if $X_{t}$ is the solution of the  SDE (\ref{Ito in section 4}),
it is natural to expect that for quasi-surely, $X_{t}$ is the solution
of the following RDE:
\begin{equation}
dY_{t}=b(Y_{t})dt+(h_{\alpha\beta}(Y_{t})-\frac{1}{2}DV_{\beta}(Y_{t})\cdot V_{\alpha}(Y_{t}))d\langle B^{\alpha},B^{\beta}\rangle+V_{\alpha}(Y_{t})dB_{t}^{\alpha},\label{RDE section 4}
\end{equation}
with the same initial condition.

In the remaining of this section, we are going to prove this claim
in a rigorous way.

From now on, assume that $X_{t}$ is the solution of the  SDE (\ref{Ito in section 4})
and $Y_{t}$ is the solution of the RDE (\ref{RDE section 4}) with
the same initial condition $x\in\mathbb{R}^{N}$, where the coefficients
$b,h_{\alpha\beta},V_{\alpha}$ are $C_{b}^{3}$-vector fields on
$\mathbb{R}^{N}.$ For simplicity we will also use the same notation
to denote constants only depending on $d,N,G,p$ and the coefficients
of (\ref{Ito in section 4}), although they may be different from
line to line.

The following lemma enables us to show that the $L_{G}^{2}$-limit
in the last term of the identity (\ref{representation of Y_s,t})
is zero.
\begin{lem}
\label{zero limit}Let $f\in C_{b}(\mathbb{R}^{N})$, and $s<t$ be
two dyadic points in $[0,1]$ (i.e., $s=t_{k}^{m}$ and $t=t_{l}^{m}$
for some $m$ and $k<l$). Then for any $\alpha,\beta=1,2,\cdots,d,$
\[
\lim_{n\rightarrow\infty}\mathbb{E}^{G}[(\sum_{k:[t_{k-1}^{n},t_{k}^{n}]\subset[s,t]}f(Y_{t_{k-1}^{n}})\int_{t_{k-1}^{n}}^{t_{k}^{n}}B_{t_{k-1}^{n},u}^{\alpha}dB_{u}^{\beta})^{2}]=0.
\]
\end{lem}
\begin{proof}
From direct calculation, we have 
\begin{eqnarray*}
 &  & \mathbb{E}^{G}[(\sum_{k:[t_{k-1}^{n},t_{k}^{n}]\subset[s,t]}f(Y_{t_{k-1}^{n}})\int_{t_{k-1}^{n}}^{t_{k}^{n}}B_{t_{k-1}^{n},u}^{\alpha}dB_{u}^{\beta})^{2}]\\
 & \leqslant & \|f\|_{\infty}^{2}\sum_{k:[t_{k-1}^{n},t_{k}^{n}]\subset[s,t]}\mathbb{E}^{G}[(\int_{t_{k-1}^{n}}^{t_{k}^{n}}B_{t_{k-1}^{n},u}^{\alpha}dB_{u}^{\beta})^{2}]\\
 &  & +2\sum_{\substack{k<l\\
{}[t_{k-1}^{n},t_{k}^{n}],[t_{l-1}^{n},t_{l}^{n}]\subset[s,t]
}
}\mathbb{E}^{G}[f(Y_{t_{k-1}^{n}})(\int_{t_{k-1}^{n}}^{t_{k}^{n}}B_{t_{k-1}^{n},u}^{\alpha}dB_{u}^{\beta})f(Y_{t_{l-1}^{n}})(\int_{t_{l-1}^{n}}^{t_{l}^{n}}B_{t_{l-1}^{n},u}^{\alpha}dB_{u}^{\beta})]
\end{eqnarray*}
\begin{eqnarray*}
 & \leqslant & C\|f\|_{\infty}^{2}\sum_{k:[t_{k-1}^{n},t_{k}^{n}]\subset[s,t]}(\Delta t^{n})^{2}\\
 &  & +2\sum_{\substack{k<l\\
{}[t_{k-1}^{n},t_{k}^{n}],[t_{l-1}^{n},t_{l}^{n}]\subset[s,t]
}
}(\mathbb{E}^{G}[(f(Y_{t_{k-1}^{n}})(\int_{t_{k-1}^{n}}^{t_{k}^{n}}B_{t_{k-1}^{n},u}^{\alpha}dB_{u}^{\beta})f(Y_{t_{l-1}^{n}}))^{+}\\
 &  & \cdot\mathbb{E}^{G}[\int_{t_{l-1}^{n}}^{t_{l}^{n}}B_{t_{l-1}^{n},u}^{\alpha}dB_{u}^{\beta}|\Omega_{t_{l-1}^{n}}]]+\mathbb{E}^{G}[(f(Y_{t_{k-1}^{n}})(\int_{t_{k-1}^{n}}^{t_{k}^{n}}B_{t_{k-1}^{n},u}^{\alpha}dB_{u}^{\beta})f(Y_{t_{l-1}^{n}}))^{-}\\
 &  & \cdot\mathbb{E}^{G}[-\int_{t_{l-1}^{n}}^{t_{l}^{n}}B_{t_{l-1}^{n},u}^{\alpha}dB_{u}^{\beta}|\Omega_{t_{l-1}^{n}}]])\\
 & \leqslant & C\|f\|_{\infty}^{2}\Delta t^{n},
\end{eqnarray*}
and the result follows easily.
\end{proof}

Now we are in position to prove our main result of this section.
\begin{thm}
\label{relation}For quasi-surely, 
\[
X_{t}=Y_{t},\ \forall t\in[0,1].
\]
\end{thm}
\begin{proof}
Since the coefficients of the RDE (\ref{RDE section 4}) are in $C_{b}^{3}(\mathbb{R}^{N}),$
for quasi-surely define the following pathwise control: for $0\leqslant s<t\leqslant1,$
\[
\omega(s,t):=(\|V\|_{2,\infty}\|\boldsymbol{B}\|_{p-var\mbox{\ensuremath{;[s,t]}}})^{p}+\|b\|_{1,\infty}(t-s)+\|h-\frac{1}{2}DV\cdot V\|_{1,\infty}\|\langle B,B\rangle\|_{1-var;\ensuremath{[s,t]}},
\]
where $\|\cdot\|_{m,\infty}$ denotes the maximum of uniform norms
of derivatives up to order $m.$ It follows from the rough Taylor
expansion (Corollary 12.8 \cite{friz2010multidimensional}) that for
quasi-surely, there exists some positive constant $\theta>1,$ such
that for $0\leqslant s<t\leqslant1,$ when $\omega(s,t)\leqslant1,$
we have
\[
|Y_{s,t}-I_{s,t}|\leqslant C\omega(s,t)^{\theta},
\]
where
\begin{eqnarray*}
I_{s,t} & = & b(Y_{s})(t-s)+(h_{\alpha\beta}(Y_{s})-\frac{1}{2}DV_{\beta}(Y_{s})\cdot V_{\alpha}(Y_{s}))\langle B^{\alpha},B^{\beta}\rangle_{s,t}+V_{\alpha}(Y_{s})B_{s,t}^{1;\alpha}\\
 &  & +DV_{\beta}(Y_{s})\cdot V_{\alpha}(Y_{s})B_{s,t}^{2;\alpha,\beta}
\end{eqnarray*}
By Proposition \ref{representation of second level}, we have for
quasi-surely,
\begin{equation}
|Y_{s,t}-b(Y_{s})(t-s)-h_{\alpha\beta}(Y_{s})\langle B^{\alpha},B^{\beta}\rangle_{s,t}-V_{\alpha}(Y_{s})B_{s,t}^{1;\alpha}-DV_{\beta}(Y_{s})\cdot V_{\alpha}(Y_{s})\int_{s}^{t}B_{s,u}^{\alpha}dB_{u}^{\beta}|\leqslant C\omega(s,t)^{\theta}.\label{useful rough Taylor expansion}
\end{equation}

Now consider fixed $s<t$ being two dyadic points in $[0,1].$ When
$n$ is large enough, by applying inequality (\ref{useful rough Taylor expansion})
on each small dyadic interval $[t_{k-1}^{n},t_{k}^{n}]\subset[s,t]$
and summing up through the triangle inequality, we obtain that for
quasi-surely,
\begin{eqnarray*}
|Y_{s,t}-I_{s,t}^{n}| & \leqslant & C\sum\omega(t_{k-1}^{n},t_{k}^{n})^{\theta}\\
 & \leqslant & C\omega(s,t)\mbox{ max}\{\omega(t_{k-1}^{n},t_{k}^{n})^{\theta-1}:\ [t_{k-1}^{n},t_{k}^{n}]\subset[s,t]\},
\end{eqnarray*}
where 
\begin{eqnarray*}
I_{s,t}^{n} & = & \sum b(Y_{t_{k-1}^{n}})\Delta t^{n}+\sum h_{\alpha\beta}(Y_{t_{k-1}^{n}})\Delta_{k}^{n}\langle B^{\alpha},B^{\beta}\rangle+\sum V_{\alpha}(Y_{t_{k-1}^{n}})\Delta_{k}^{n}B^{\alpha}\\
 &  & +\sum DV_{\beta}(Y_{t_{k-1}^{n}})\cdot V_{\alpha}(Y_{t_{k-1}^{n}})\int_{t_{k-1}^{n}}^{t_{k}^{n}}B_{t_{k-1}^{n},u}^{\alpha}dB_{u}^{\beta},
\end{eqnarray*}
and each sum is over all $k$ such that $[t_{k-1}^{n},t_{k}^{n}]\subset[s,t].$
It follows that for quasi-surely, 
\[
I_{s,t}^{n}\rightarrow Y_{s,t},\ \mbox{\ensuremath{n\rightarrow\infty.}}
\]

On the other hand, the following convergence in $L_{G}^{2}(\Omega;\mathbb{R}^{N})$
holds: 
\begin{eqnarray*}
\sum b(Y_{t_{k-1}^{n}})\Delta t^{n} & \rightarrow & \int_{s}^{t}b(Y_{u})du,\\
\sum h_{\alpha\beta}(Y_{t_{k-1}^{n}})\Delta_{k}^{n}\langle B^{\alpha},B^{\beta}\rangle & \rightarrow & \int_{s}^{t}h_{\alpha\beta}(Y_{u})d\langle B^{\alpha},B^{\beta}\rangle_{u},\\
\sum V_{\alpha}(Y_{t_{k-1}^{n}})\Delta_{k}^{n}B^{\alpha} & \rightarrow & \int_{s}^{t}V_{\alpha}(Y_{u})dB_{u}^{\alpha},
\end{eqnarray*}
as $n\rightarrow\infty.$ 

The reason is the following. For simplicity we only consider the third
one, as the first two are similar (and in fact easier). It is straight
forward that 
\begin{align*}
 & \int_{0}^{1}|V_{\alpha}(Y_{t})-\sum_{k=1}^{2^{n}}V_{\alpha}(Y_{t_{k-1}^{n}})\mathbf{1}_{[t_{k-1}^{n},t_{k}^{n})}(t)|^{2}dt\\
= & \sum_{k=1}^{2^{n}}\int_{t_{k-1}^{n}}^{t_{k}^{n}}|V_{\alpha}(Y_{t})-V_{\alpha}(Y_{t_{k-1}^{n}})|^{2}dt\\
\leqslant & C\sum_{k=1}^{2^{n}}\int_{t_{k-1}^{n}}^{t_{k}^{n}}|Y_{t}-Y_{t_{k-1}^{n}}|^{2}dt
\end{align*}
\begin{eqnarray*}
 & \leqslant & C\sum_{k=1}^{2^{n}}\|Y\|_{p-var\mbox{;}[t_{k-1}^{n},t_{k}^{n}]}^{2}\Delta t^{n}\\
 & \leqslant & C(\sum_{k=1}^{2^{n}}\|Y\|_{p-var\ensuremath{;[t_{k-1}^{n},t_{k}^{n}]}}^{p}\Delta t^{n})^{\frac{2}{p}}\\
 & \leqslant & C(\Delta t^{n})^{\frac{2}{p}}\|Y\|_{p-var;[0,1]}^{2},
\end{eqnarray*}
where $C$ depends only on $V_{\alpha}.$ Therefore, it suffices to
show that $\|Y\|_{p-var;[0,1]}\in L_{G}^{2}(\Omega),$ as it will
imply the $G$-It$\hat{\mbox{o}}$ integrability of $V_{\alpha}(Y_{t})$
and the desired convergence in $L_{G}^{2}(\Omega;\mathbb{R}^{N})$
will hold. For simplicity we assume that $Y_{t}$ is the solution
of the following RDE
\[
dY_{t}=V_{\alpha}(Y_{t})dB_{t}^{\alpha}
\]
with $Y_{0}=\xi$ (there is no substantial difference because $dt$
and $d\langle B^{\alpha},B^{\beta}\rangle_{t}$ are more regular than
$dB_{t}$), then by Theorem 10.14 in \cite{friz2010multidimensional},
we know that 
\[
\|Y\|_{p-var;[0,1]}\leqslant C\|\boldsymbol{B}\|_{p-var;[0,1]}\vee\|\boldsymbol{B}\|_{p-var;[0,1]}^{p}.
\]
Therefore, we only need to show that $\|\boldsymbol{B}\|_{p-var;[0,1]}^{p}\in L_{G}^{2}(\Omega).$
For this purpose, we use Proposition \ref{p-var control} to control
the $p$-variation norm by the functions $\rho_{1},\rho_{2}$ defined
in (\ref{def rho}). It follows that
\[
\|\boldsymbol{B}\|_{p-var}\leqslant C(1+\rho_{1}(\boldsymbol{B})^{2}+\rho_{2}(\boldsymbol{B})).
\]
Therefore, it remains to show that $\rho_{1}(\boldsymbol{B})^{2p},\rho_{2}(\boldsymbol{B})^{p}\in L_{G}^{1}(\Omega).$
First consider level one. By the distribution of $B_{t},$ we have
\begin{eqnarray*}
\|\sum_{n=1}^{\infty}n^{\gamma}\sum_{k=1}^{2^{n}}|B_{t_{k-1}^{n},t_{k}^{n}}^{1}|^{p}\|_{2} & \leqslant & \sum_{n=1}^{\infty}n^{\gamma}\sum_{k=1}^{2^{n}}\||B_{t_{k-1}^{n},t_{k}^{n}}^{1}|^{p}\|_{2}\\
 & \leqslant & \sum_{n=1}^{\infty}n^{\gamma}(\Delta t^{n})^{\frac{p}{2}-1}\\
 & < & \infty,
\end{eqnarray*}
and we know that $\rho_{1}(\boldsymbol{B})^{2p}\in L_{G}^{1}(\Omega).$
Now consider level two. By Proposition \ref{representation of second level}
and the distribution of $B_{t}$ and $\langle B,B\rangle_{t}$, we
have 
\begin{eqnarray*}
\|\sum_{n=1}^{\infty}n^{\gamma}\sum_{k=1}^{2^{n}}|B_{t_{k-1}^{n},t_{k}^{n}}^{2}|^{\frac{p}{2}}\|_{2} & = & \|\sum_{n=1}^{\infty}n^{\gamma}\sum_{k=1}^{2^{n}}|\int_{t_{k-1}^{n}}^{t_{k}^{n}}B_{t_{k-1}^{n},u}\otimes dB_{u}+\frac{1}{2}\langle B,B\rangle_{t_{k-1}^{n},t_{k}^{n}}|^{\frac{p}{2}}\|_{2}\\
 & \leqslant & \sum_{n=1}^{\infty}n^{\gamma}\sum_{k=1}^{2^{n}}\||\int_{t_{k-1}^{n}}^{t_{k}^{n}}B_{t_{k-1}^{n},u}\otimes dB_{u}+\frac{1}{2}\langle B,B\rangle_{t_{k-1}^{n},t_{k}^{n}}|^{\frac{p}{2}}\|_{2}\\
 & \leqslant & C\sum_{n=1}^{\infty}n^{\gamma}(\Delta t^{n})^{\frac{p}{2}-1}\\
 & < & \infty.
\end{eqnarray*}
It follows that $\rho_{2}(\boldsymbol{B})^{p}\in L_{G}^{1}(\Omega).$
Therefore, the desired $L_{G}^{2}$-convergence holds. 

In addition, by Lemma \ref{zero limit} we also have the following
$L_{G}^{2}$-convergence:
\[
\sum DV_{\beta}(Y_{t_{k-1}^{n}})\cdot V_{\alpha}(Y_{t_{k-1}^{n}})\int_{t_{k-1}^{n}}^{t_{k}^{n}}B_{t_{k-1}^{n},u}^{\alpha}dB_{u}^{\beta}\rightarrow0,\ n\rightarrow\infty.
\]

Consequently, in $L_{G}^{2}(\Omega;\mathbb{R}^{N}),$
\[
I_{s,t}^{n}\rightarrow\int_{s}^{t}b(Y_{u})du+\int_{s}^{t}h_{\alpha\beta}(Y_{u})d\langle B^{\alpha},B^{\beta}\rangle_{u}+\int_{s}^{t}V_{\alpha}(Y_{u})dB_{u}^{\alpha},
\]
as $n\rightarrow\infty.$

From Lemma \ref{same limit}, we conclude that for quasi-surely,
\[
Y_{s,t}=\int_{s}^{t}b(Y_{u})du+\int_{s}^{t}h_{\alpha\beta}(Y_{u})d\langle B^{\alpha},B^{\beta}\rangle_{u}+\int_{s}^{t}V_{\alpha}(Y_{u})dB_{u}^{\alpha}.
\]
Since $X_{t}$ and $Y_{t}$ are both quasi-surely continuous, it follows
that $X$ coincides with $Y$ quasi-surely.
\end{proof}

\begin{rem}
As we mentioned at the beginning of Section 2, it is possible to prove
Theorem \ref{relation} by establishing the Wong-Zakai type approximation.
More precisely, if we let $X_{t}^{n}$ to be the Euler-Maruyama approximation
of the SDE (\ref{Ito in section 4}) and let $Y_{t}^{n}$ to be the
unique classical solution of the following ODE:
\[
dY_{t}^{n}=b(Y_{t}^{n})dt+(h_{\alpha\beta}(Y_{t}^{n})-\frac{1}{2}DV_{\beta}(Y_{t}^{n})\cdot V_{\alpha}(Y_{t}^{n}))d\langle B^{\alpha},B^{\beta}\rangle_{t}+V_{\alpha}(Y_{t}^{n})d(B^{n})_{t}^{\alpha}
\]
with $X_{0}^{n}=Y_{0}^{n}=\xi,$ where $B_{t}^{n}$ is the dyadic
piecewise linear approximation of $B_{t},$ then by using our main
result in Section 2 and establishing related $L_{G}^{2}$-estimates,
we can prove that 
\[
\sup_{t\in[0,1]}\mathbb{E}^{G}[|X_{t}^{n}-Y_{t}^{n}|^{2}]\leqslant C\sqrt{1+\xi^{2}}(\Delta t^{n})^{\frac{1}{2}}.
\]
In other words, $Y_{t}^{n}$ converges to the solution $X_{t}$ of
the SDE (\ref{Ito in section 4}) in the $L_{G}^{2}$-sense. However,
we know that for quasi-surely, $Y_{t}^{n}$ converges uniformly to
the solution $Y_{t}$ of the RDE(\ref{RDE section 4}). Again by Lemma
\ref{same limit} and continuity, we conclude that for quasi-sure,
$X$ coincides with $Y.$

From the above discussion, if we forget about the RDE (\ref{RDE section 4})
and only consider the $L_{G}^{2}$-limit of $Y_{t}^{n},$ it seems
that there is nothing to do with rough paths at all as everything
is well-defined in the classical sense. However, the fundamental point
of understanding the convergence of $Y_{t}^{n}$ in the pathwise sense
lies in the crucial fact that $B_{t}$ can be regarded as geometric
rough paths (i.e., the enhancement defined in Section 3) with approximating
sequence in $G\Omega_{p}(\mathbb{R}^{d})$ being the enhancement of
the natural dyadic piecewise linear approximation $B_{t}^{n}.$ This
is exactly what the universal limit theorem tells us.
\end{rem}

\begin{rem}
From the RDE point of view, it is possible to reduce the regularity
assumptions on the coefficients. In particular, since the regularity
of $t$ and $\langle B^{\alpha},B^{\beta}\rangle_{t}$ are both ``better''
than $B_{t},$ the regularity assumptions on the coefficients of $dt$
and $d\langle B^{\alpha},B^{\beta}\rangle_{t}$ can be weaker than
the one imposed on the coefficient of $dB_{t}.$ However, we are not
going to present the results under such generality. Please refer to
\cite{friz2010multidimensional} for general existence and uniqueness
results of RDEs.
\end{rem}

\section{SDEs on a Differentiable Manifold Driven by $G$-Brownian Motion}

Our main result in Section 5 can be used to establish SDEs on a differentiable
manifold driven by $G$-Brownian motion, which will be the main focus
of this section. The development is based on the idea in the classical
case, for which one may refer to \cite{elworthy1998stochastic}, \cite{hsu2002stochastic},
\cite{ikeda1989stochastic}. This part is the foundation of developing
$G$-Brownian motion on a Riemannina manifold in the next section.

In classical stochastic analysis, SDEs on a manifold is established
under the Stratonovich type formulation, which can be regarded as
a pathwise approach. The reason of using Stratonovich type formulation
instead of the It$\hat{\mbox{o}}$ type one is the following. First
of all, the notion of SDE can be introduced by using test functions
on the manifold from an intrinsic point of view, which is consistent
with ordinary differential calculus and invariant under diffeomorphisms.
Moreover, when we construct solutions extrinsically, we can prove
that for almost surely, the solution of the extended SDE which starts
from the manifold will always live on it. This reveals the intrinsic
nature of ordinary differential equations. 

In the setting of $G$-expectation, we will adopt the same idea for
the development. However, there is a major difficulty here. The method
of constructing solutions in the classical case from the extrinsic
point of view depends heavily on the localization technique, which
is not available in the setting of $G$-expectation, mainly due to
the reason that concepts of information flows and stopping times are
not well understood. To get around with this difficulty, we will use
our main result in Section 5 to obtain a pathwise construction. The
advantage of such approach is that we can still use localization arguments
but don't need to care about measurability and integrability under
$G$-expectation.

Now assume that $M$ is a differentiable manifold. For technical reasons
we further assume that $M$ is compact (it is not necessary if we
impose more restrictive regularity assumptions on the generating vector
fields). Let $\{b,h_{\alpha\beta},V_{\alpha}:\alpha,\beta=1,2,\cdots,d\}$
be a family of $C^{3}$-vector fields on $M,$ and let $B_{t}$ be
the canonical $d$-dimensional $G$-Brownian motion on the path space
$(\Omega,L_{G}^{2}(\Omega),\mathbb{E}^{G})$, where $G$ is a function
given by (\ref{G representation}).

Consider the following symbolic Stratonovich type SDE over $[0,1]$:
\begin{equation}
\begin{cases}
dX_{t} & =b(X_{t})dt+h_{\alpha\beta}(X_{t})d\langle B^{\alpha},B^{\beta}\rangle_{t}+V_{\alpha}(X_{t})\circ dB_{t}^{\alpha},\\
X_{0} & =\xi\in M,
\end{cases}\label{Stratonovich SDE on manifold}
\end{equation}
on $M$. 
\begin{defn}
\label{definition of solution on manifold} A solution $X_{t}$ of
the SDE (\ref{Stratonovich SDE on manifold}) is an $M$-valued continuous
stochastic process such that for any $f\in C^{\infty}(M),$ 
\[
\{h_{\alpha\beta}f(X_{t}):t\in[0,1]\}\in M_{G}^{1}(0,1),\ \{V_{\alpha}f(X_{t}):t\in[0,1]\}\in M_{G}^{2}(0,1),\ \forall\alpha,\beta=1,2,\cdots,d,
\]
and the following equality holds on $[0,1]:$ 
\begin{equation}
f(X_{t})=f(\xi)+\int_{0}^{t}bf(X_{s})ds+\int_{0}^{t}h_{\alpha\beta}f(X_{s})d\langle B^{\alpha},B^{\beta}\rangle_{s}+\int_{0}^{t}V_{\alpha}f(X_{s})\circ dB_{s}^{\alpha},\label{integral form of Stratonovich}
\end{equation}
where the last term is defined as 
\[
\int_{0}^{t}V_{\alpha}f(X_{s})\circ dB_{s}^{\alpha}:=\int_{0}^{t}V_{\alpha}f(X_{s})dB_{s}^{\alpha}+\frac{1}{2}\int_{0}^{t}V_{\beta}V_{\alpha}f(X_{s})d\langle B^{\alpha},B^{\beta}\rangle_{s}.
\]

\end{defn}

\begin{rem}
Definition \ref{definition of solution on manifold} is intrinsic.
It is easy to see that Definition \ref{definition of solution on manifold}
is consistent with the Euclidean case. 
\end{rem}

Now we are going to construct the solution of (\ref{Stratonovich SDE on manifold})
from the extrinsic point of view.

According to the Whitney  embedding theorem (see \cite{derham1984riemannian}),
$M$ can be embedded into some ambient Euclidean space $\mathbb{R}^{N}$
as a submanifold such that the image $i(M)$ of $M$ is closed in
$\mathbb{R}^{N}.$ We simply regard $M$ as a subset of $\mathbb{R}^{N}.$ 

Let $F^{1},\cdots,F^{N}\in C^{\infty}(M)$ be the coordinate functions
on $M.$ The following result is easy to prove. It is similar to the
classical case.
\begin{prop}
\label{equivalent condition}$X_{t}$ is a solution of (\ref{Stratonovich SDE on manifold})
if and only of for any $i=1,2,\cdots,N,$ $\alpha,\beta=1,2,\cdots,d,$
\[
\{h_{\alpha\beta}F^{i}(X_{t}):t\in[0,1]\}\in M_{G}^{1}(0,1),\ \{V_{\alpha}F^{i}(X_{t}):t\in[0,1]\}\in M_{G}^{2}(0,1),
\]
and 
\begin{equation}
F^{i}(X_{t})=F^{i}(\xi)+\int_{0}^{t}bF^{i}(X_{s})ds+\int_{0}^{t}h_{\alpha\beta}F^{i}(X_{s})d\langle B^{\alpha},B^{\beta}\rangle_{s}+\int_{0}^{t}V_{\alpha}F^{i}(X_{s})\circ dB_{s}^{\alpha},\ \forall t\in[0,1].\label{coordinates}
\end{equation}
\end{prop}
\begin{proof}
Necessity is obvious since $F^{i}\in C^{\infty}(M)$ for any $i=1,2,\cdots,N.$ 

Now consider sufficiency. Let $f\in C^{\infty}(M),$ and choose a
$C^{\infty}$-extension $\widetilde{f}$ of $f$ with compact support
in $\mathbb{R}^{N}$ (it is possible since $M$ is compact). Then
for any $x\in M,$
\[
f(x)=\widetilde{f}(F^{1}(x),\cdots,F^{N}(x)),
\]
and thus 
\[
f(X_{t})=\widetilde{f}(F^{1}(X_{t}),\cdots,F^{N}(X_{t})),\ \forall t\in[0,1].
\]
Since $M$ is compact and $\widetilde{f}$ is smooth with compact
support, it follows from the $G$-It$\hat{\mbox{o}}$ formula that
for $t\in[0,1],$ 
\begin{eqnarray*}
\widetilde{f}(F^{1}(X_{t}),\cdots,F^{N}(X_{t})) & = & f(\xi)+\int_{0}^{t}\frac{\partial\widetilde{f}}{\partial y^{i}}(bF^{i}(X_{s})ds+h_{\alpha\beta}F^{i}(X_{s})d\langle B^{\alpha},B^{\beta}\rangle_{s}\\
 &  & +V_{\alpha}F^{i}(X_{s})\circ dB_{s}^{\alpha})\\
 & = & f(\xi)+\int_{0}^{t}(bf(X_{s})ds+h_{\alpha\beta}f(X_{s})d\langle B^{\alpha},B^{\beta}\rangle_{s}+V_{\alpha}f(X_{s})\circ dB_{s}^{\alpha}),
\end{eqnarray*}
where we have used the simple fact that for any $C^{1}$-vector field
$V$ on $M,$ 
\[
Vf=\sum_{i=1}^{N}\frac{\partial\widetilde{f}}{\partial y^{i}}VF^{i}.
\]
By Definition \ref{definition of solution on manifold}, we know that
$X_{t}$ is a solution of the SDE (\ref{Stratonovich SDE on manifold}).
\end{proof}

Now we are going to prove the existence and uniqueness of (\ref{Stratonovich SDE on manifold})
by using the main result of Section 5, namely, a pathwise approach
based on the associated RDE.

Let $\widetilde{b},\widetilde{h}_{\alpha\beta},\widetilde{V}_{\alpha}$
be $C_{b}^{3}$-extensions (not unique) of the vector fields $b,h_{\alpha\beta},V_{\alpha}$.
Consider the following Stratonovich type SDE in the ambient space
$\mathbb{R}^{N}:$
\begin{equation}
dX_{t}=\widetilde{b}(X_{t})dt+\widetilde{h}_{\alpha\beta}(X_{t})d\langle B^{\alpha},B^{\beta}\rangle_{t}+\widetilde{V}_{\alpha}(X_{t})\circ dB_{t}^{\alpha}\label{extended Stratonovich SDE}
\end{equation}
with $X_{0}=x\in\mathbb{R}^{N},$ which is interpreted as the following
It$\hat{\mbox{o}}$ type SDE:
\[
dX_{t}=\widetilde{b}(X_{t})dt+(\widetilde{h}_{\alpha\beta}(X_{t})+\frac{1}{2}D\widetilde{V}_{\alpha}(X_{t})\cdot\widetilde{V}_{\beta}(X_{t}))d\langle B^{\alpha},B^{\beta}\rangle_{t}+\widetilde{V}_{\alpha}(X_{t})dB_{t}^{\alpha}.
\]
According to Section 5, we can alternatively interpret (\ref{extended Stratonovich SDE})
as an RDE which is pathwisely defined. Both the SDE and the RDE has
a unique solution, and according to Theorem \ref{relation} they coincide
quasi-surely. Our aim is to show that for quasi-surely, the solution
$X_{t}$ of (\ref{extended Stratonovich SDE}) never leaves $M$ and
it is the unique solution of (\ref{Stratonovich SDE on manifold}). 

The following result is important to prove the existence and uniqueness
of the SDE (\ref{Stratonovich SDE on manifold}) on the manifold $M$. 
\begin{prop}
\label{ODE on manifolds}Let $x_{t}$ be a path of bounded variation
in $\mathbb{R}^{d}$. Let $W_{1},\cdots,W_{d}$ be a family of $C^{1}$-vector
fields on $M$ and $\widetilde{W}_{1},\cdots,\widetilde{W}_{d}$ be
their $C_{b}^{1}$-extensions to $\mathbb{R}^{N}.$ Consider the following
ODE in the ambient space $\mathbb{R}^{N}$ over $[0,1]:$
\begin{equation}
dy_{t}=\widetilde{W}_{\alpha}(y_{t})dx_{t}^{\alpha}\label{ODE}
\end{equation}
with $y_{0}=x\in M.$ Then the solution $y_{t}\in M$ for all $t\in[0,1].$
Moreover, $y_{t}$ does not depend on extensions of the vector fields.\end{prop}
\begin{proof}
Let $F(x):=d(x,M)^{2}$ be the squared distance function to the submanifold
$M.$ It follows that $ $$F$ is smooth in an open neighborhood of
$M$. By using the cut-off function we may assume that $F\in C_{b}^{\infty}(M)$.
Now we are able to choose an open neighborhood $U$ of $M$, such
that for any $x\in U,$ $F(x)=0$ if and only if $x\in M$. Moreover,
since $\widetilde{W}_{\alpha}$ ($\alpha=1,2,\cdots,d$) are tangent
vector fields of $M$ when restricted on $M,$ $U$ can be chosen
such that for any $x\in U$ and $\alpha=1,2,\cdots,d,$ 
\begin{equation}
|\widetilde{W}_{\alpha}F(x)|\leqslant CF(x),\label{WF<CF}
\end{equation}
for some positive constant $C$ depending on $U.$ The function $F(x)$
was used in \cite{hsu2002stochastic} to construct SDEs on $M$ driven
by classical Brownian motion.

Since $x_{t}$ is a path of bounded variation and $y_{0}=\xi\in M$,
by the change of variables formula in ordinary calculus, we have
\[
F(y_{t})=\int_{0}^{t}\widetilde{W}_{\alpha}F(y_{s})dx_{s}^{\alpha},\ \forall t\in[0,1].
\]
Define $\tau:=\inf\{t\in[0,1]:\ y_{t}\notin U\}.$ It follows from
(\ref{WF<CF}) that 
\[
F(y_{t})\leqslant C\int_{0}^{t}F(y_{s})d|x|_{s},\ \forall t\in[0,\tau],
\]
where $|x|_{t}$ is the total variation of the path $x_{t}.$ 

By iteration and Fubini theorem, on $[0,\tau]$ we have 
\begin{eqnarray*}
F(y_{t}) & \leqslant & C^{2}\int_{0}^{t}(\int_{0}^{s}F(y_{u})d|x|_{u})d|x|_{s}\\
 & = & C^{2}\int_{0}^{t}(|x|_{t}-|x|_{s})F(y_{s})d|x|_{s}.
\end{eqnarray*}
By induction, it is easy to see that for any $k\geqslant1,$
\[
F(y_{t})\leqslant C^{k}\int_{0}^{t}\frac{(|x|_{t}-|x|_{s})^{k-1}}{(k-1)!}F(y_{s})d|x|_{s},\ \forall t\in[0,\tau].
\]
Since $F$ is bounded, we obtain further that for any $k\geqslant1,$
\[
F(y_{t})\leqslant\|F\|_{\infty}\frac{C^{k}(|x|_{t}-|x|_{0})^{k}}{k!},\ \forall t\in[0,\tau].
\]
By letting $k\rightarrow\infty,$ it follows that $F(y_{t})\equiv0$
on $[0,\tau]$, which implies that $y_{t}\in M$ for any $t\in[0,\tau].$
Since $y_{t}$ is continuous, the only possibility is that $y_{t}$
never leaves $M$ on $[0,1].$

If we rewrite the ODE (\ref{ODE}) in its integral form:
\begin{equation}
y_{t}=\xi+\int_{0}^{t}\widetilde{W}_{\alpha}(y_{s})dx_{s}^{\alpha},\ t\in[0,1],\label{integral form}
\end{equation}
we know from previous discussion that equation (\ref{integral form})
depends only on the values of $\widetilde{W}_{\alpha}$ on $M,$ that
is, of $W_{\alpha}$ ($\alpha=1,2,\cdots,d$). In other words, if
$\hat{W}_{\alpha}$ is another extension of $W_{\alpha}$ and $\widehat{y}_{t}$
is the solution of the corresponding ODE with the same initial condition,
$\widehat{y}_{t}$ is also a solution of (\ref{ODE}). By uniqueness,
we have $y=\widehat{y}.$ Therefore, $y_{t}$ does not depend on extensions
of the vector fields.
\end{proof}

With the help of Proposition \ref{ODE on manifolds}, we can prove
the following existence and uniqueness result.
\begin{thm}
Let $b,h_{\alpha\beta},V_{\alpha}$ be $C^{3}$-vector fields on $M.$
Then the Stratonovich type SDE (\ref{Stratonovich SDE on manifold})
has a solution $X_{t}$ which is unique quasi-surely.\end{thm}
\begin{proof}
Fix $C_{b}^{3}$-extensions $\widetilde{b},\widetilde{h}_{\alpha\beta},\widetilde{V}_{\alpha}$
of $b,h_{\alpha\beta},V_{\alpha}$, and let $X_{t}$ be the solution
of the Stratonovich type SDE (\ref{extended Stratonovich SDE}) in
$\mathbb{R}^{N}$ over $[0,1].$ By Theorem \ref{relation}, for quasi-surely
$X_{t}$ coincides with the solution of (\ref{extended Stratonovich SDE})
when it is interpreted as an RDE. Since $M$ is closed in $\mathbb{R}^{N},$
it follows from Proposition \ref{ODE on manifolds} and Theorem \ref{universal limit theorem}
(the universal limit theorem) that for quasi-surely, $X_{t}$ never
leaves $M$ over $[0,1].$ In this case, (\ref{extended Stratonovich SDE})
is equivalent to (\ref{coordinates}), which implies from Proposition
\ref{equivalent condition} that $X_{t}$ is a solution of (\ref{Stratonovich SDE on manifold}).
On the other hand, if $Y_{t}$ is another solution of (\ref{Stratonovich SDE on manifold}),
then it is a solution of (\ref{extended Stratonovich SDE}) (interpreted
as an SDE or an RDE). By the uniqueness of RDEs, we know that $X=Y$
quasi-surely.
\end{proof}

\begin{rem}
It is possible to formulate uniqueness in the $L_{G}^{2}$-sense when
$M$ is regarded as a closed submanifold of $\mathbb{R}^{N}.$ However,
we use the quasi-sure formulation because the notion itself is intrinsic
although the proof is developed from the extrinsic point of view.
\end{rem}

\section{$G$-Brownian Motion on a Compact Riemannian Manfold and the Generating
PDE}

In this section, we are going to introduce the notion of $G$-Brownian
motion on a Riemannian manifold for a wide and interesting class of
$G$-functions, based on Eells-Elworthy-Malliavin's horizontal lifting
construction (see \cite{elworthy1998stochastic}, \cite{hsu2002stochastic},
\cite{ikeda1989stochastic} for the construction of Brownian motion
on a Riemannian manifold and related topics). Roughly speaking, we
will ``roll'' an Euclidean $G$-Brownian motion up to a Riemannian
manifold ``without slipping'' via a proper frame bundle (for the
class of $G$-functions we are interested in, such bundle is the orthonormal
frame bundle).

In the classical case, we know that the law of a $d$-dimensional
Brownian motion $B_{t}$ is invariant under orthogonal transformations
on $\mathbb{R}^{d}.$ This is a crucial point to obtain a linear parabolic
PDE (in fact, the standard heat equation associated with the Bochner
horizontal Laplacian $\Delta_{\mathcal{O}(M)}$) on the orthonormal
frame bundle $\mathcal{O}(M)$ over a Riemannian manifold $M$ governing
the law of the horizontal lifting $\xi_{t}$ of $B_{t}$ to $\mathcal{O}(M),$
which is invariant under orthogonal transformations along fibers.
It is such an invariance that enables us to ``project'' the PDE
onto the base manifold $M$ and obtain the standard heat equation
associated with the Laplace-Beltrami operator $\Delta_{M}$ on $M.$
This heat equation governs the law of the development $X_{t}=\pi(\xi_{t})$
of $B_{t}$ to the Riemannian manifold $M$ via the horizontal lifting
$\xi_{t}$. As a stochastic process on $M,$ although $X_{t}$ depends
on the initial orthonormal frame $\xi$ at $x$ as well as the initial
position $x\in M$, the law of $X_{t}$ depends only on the initial
position $x,$ and it is characterized by the Laplace-Beltrami operator
$\Delta_{M}$ via the heat equation. Equivalently, it can be shown
that the law of $X_{t}$ is the unique solution of the martingale
problem on $M$ associated with $\Delta_{M}$ starting at $x$. $X_{t}$
is called the Brownian motion on $M$ starting at $x$ in the sense
of Eells-Elworthy-Malliavin.

It is quite natural to expect that the Brownian sample paths $X_{t}$
on $M$ will depend on the initial orthonormal frame $\xi$ at $x$
if we look back into the Euclidean case, in which we actually fix
the standard orthonormal basis in advance and define Brownian motion
in the corresponding coordinate system. If we use another orthonormal
basis, we obtain a process (still a Brownian motion) which is an orthogonal
transformation of the original Brownian motion. Therefore, it is the
law, which is characterized by the Laplace operator on $\mathbb{R}^{d}$,
rather than the sample paths that captures the intrinsic nature of
the Brownian motion, and such nature can be developed in a Riemannian
geometric setting. 

It should be remarked that in a pathwise manner, we can lift $B_{t}$
horizontally to the total frame bundle $\mathcal{F}(M)$ instead of
$\mathcal{O}(M)$ by solving the same SDE generating by the horizontal
vector fields but using a general frame instead of an orthonormal
one as initial condition. Moreover, we can write down the generating
heat equation on $\mathcal{F}(M)$ which takes the same form of the
one on $\mathcal{O}(M).$ The key difference here is that although
the horizontal lifting of $B_{t}$ can be projected onto $M,$ the
heat equation on $\mathcal{F}(M)$ cannot. In other words, the heat
equation is not invariant under nondegenerate linear transformations
along fibers. This becomes uninteresting to us, as we are not able
to obtain an intrinsic law of the development of $B_{t}$ on $M$
which is independent of initial frames. The fundamental reason of
using the orthonormal frame bundle is that the Laplace operator on
$\mathbb{R}^{d}$ is invariant exactly under orthogonal transformations.

The case of $G$-Brownian motion can be understood in a similar manner.
From the last section we are able to solve SDEs on a differentiable
manifold (in particular, on $\mathcal{F}(M)$) driven by an Euclidean
$G$-Brownian motion $B_{t}$. By projection we obtain the development
$X_{t}$ of $B_{t}$ to $M$. As we've pointed out before, such development
is of no interest unless we are able to prove that the law of $X_{t}$
depends only on the initial position $x$ rather than the initial
frame. In fact, if the law of $X_{t}$ depends on the initial frame,
we might not be able to write down the generating PDE of $X_{t}$
intrinsically on $M$ although it is possible on $\mathcal{F}(M)$.
Therefore, for a given $G$-function, it is crucial to identify a
proper frame bundle over $M$ with a specific structure group such
that parallel transport preserves fibers and the generating PDE (associated
with $G$) of the horizontal lifting $\xi_{t}$ of $B_{t}$ to such
frame bundle is invariant under actions by the structure group along
fibers. From this, the law of $X_{t}$ will be independent of initial
frames in the fibre over $x$ ($x$ is the starting point of $X_{t}$)
and we might be able to obtain the generating PDE of $X_{t}$, which
is associated with $G$ and intrinsically defined on $M.$

As we shall see, such idea depends on a crucial algebraic quantity
associated with the $G$-function called the invariant group $I(G)$
of $G$, which will be defined later on. In this paper, we are interested
in the case when $I(G)$ is the orthogonal group. We will see that
it contains a wide class of $G$-functions. In particular, one example
is the generalization of the one-dimensional Barenblatt equation to
higher dimensions.

The concept of the invariant group of $G$ is motivated from the study
of infinitesimal diffusive nature of SDEs driven by $G$-Brownian
motion and their generating PDEs, which will be discussed below.

We first consider the Euclidean case. 

From now on, we always assume that $G:\ S(d)\rightarrow\mathbb{R}$
is a given continuous, sublinear and monotonic function. Equivalently,
from Section 2 we know that $G$ is represented by
\begin{equation}
G(A)=\frac{1}{2}\sup_{B\in\Sigma}\mbox{tr}(AB),\ \forall A\in S(d),\label{rep of G in section 7}
\end{equation}
where $\Sigma$ is some bounded, closed and convex subset of $S_{+}(d).$
Let $B_{t}$ be the standard $d$-dimensional $G$-Brownian motion
on the path space.

Assume that $V_{1},\cdots,V_{d}$ are $C_{b}^{3}$ -vector fields
on $\mathbb{R}^{N}.$ Consider the following $N$-dimensional Stratonovich
type SDE over $[0,1]$:
\begin{equation}
\begin{cases}
dX_{t,x}=V_{\alpha}(X_{t,x})\circ dB_{t}^{\alpha},\\
X_{0,x}=x,
\end{cases}\label{SDE in section 7}
\end{equation}
which is either interpreted as an RDE or the associated It$\hat{\mbox{o}}$
type SDE
\[
\begin{cases}
dX_{t,x}=V_{\alpha}(X_{t,x})dB_{t}^{\alpha}+\frac{1}{2}DV_{\alpha}(X_{t,x})V_{\beta}(X_{t,x})\langle B^{\alpha},B^{\beta}\rangle_{t},\\
X_{t,x}=x,
\end{cases}
\]
according to the main result of Section 5. 

The following result characterizes the generator of the SDE (\ref{SDE in section 7})
in terms of $G$. It describes the infinitesimal diffusive nature
of (\ref{SDE in section 7}). One might compare it with the case of
linear diffusion processes.
\begin{prop}
\label{Proposition infinitesimal generator}For any $p\in\mathbb{R}^{N},\ A\in S(N),$
\begin{align}
 & \lim_{\delta\rightarrow0^{+}}\frac{1}{\delta}\mathbb{E}^{G}[\langle p,X_{\delta,x}-x\rangle+\frac{1}{2}\langle A(X_{\delta,x}-x),X_{\delta,x}-x\rangle]\nonumber \\
= & G((\frac{1}{2}\langle p,DV_{\alpha}(x)V_{\beta}(x)+DV_{\beta}(x)V_{\alpha}(x)\rangle+\langle AV_{\alpha}(x),V_{\beta}(x)\rangle)_{1\leqslant\alpha,\beta\leqslant d}).\label{infinitesimal generator}
\end{align}
\end{prop}
\begin{proof}
From the distribution of $B_{t}$ we know that 
\begin{eqnarray*}
G(A) & = & \frac{1}{2}\mathbb{E}^{G}[\langle AB_{1},B_{1}\rangle]\\
 & = & \frac{1}{2t}\mathbb{E}^{G}[\langle AB_{t},B_{t}\rangle],\ \forall t>0.
\end{eqnarray*}
Therefore, the R.H.S. of (\ref{infinitesimal generator}) is equal
to 
\[
I_{\delta}=\frac{1}{2\delta}\mathbb{E}^{G}[(\langle p,DV_{\alpha}(x)V_{\beta}(x)\rangle+\langle AV_{\alpha}(x),V_{\beta}(x)\rangle)B_{\delta}^{\alpha}B_{\delta}^{\beta}],
\]
for any $\delta>0.$

Since
\[
X_{\delta,x}-x=\int_{0}^{\delta}V_{\alpha}(X_{s,x})dB_{s}^{\alpha}+\frac{1}{2}\int_{0}^{\delta}DV_{\alpha}(X_{s,x})V_{\beta}(X_{s,x})d\langle B^{\alpha},B^{\beta}\rangle_{s},
\]
by the properties of $\mathbb{E}^{G}$ and the distribution of $B_{t},$
we have 
\begin{align*}
 & |\frac{1}{\delta}\mathbb{E}^{G}[\langle p,X_{\delta,x}-x\rangle+\frac{1}{2}\langle A(X_{\delta,x}-x),X_{\delta,x}-x\rangle]-I_{\delta}|\\
\leqslant & |\frac{1}{2\delta}\mathbb{E}^{G}[\int_{0}^{\delta}\langle p,DV_{\alpha}(X_{s,x})\cdot V_{\beta}(X_{s,x})\rangle d\langle B^{\alpha},B^{\beta}\rangle_{s}\\
 & +\langle A\int_{0}^{\delta}V_{\alpha}(X_{s,x})dB_{s}^{\alpha},\int_{0}^{\delta}V_{\beta}(X_{s,x})dB_{s}^{\beta}\rangle]-\frac{1}{2\delta}\mathbb{E}^{G}[\langle p,DV_{\alpha}(x)\cdot V_{\beta}(x)\rangle\langle B^{\alpha},B^{\beta}\rangle_{\delta}\\
 & +\langle AV_{\alpha}(x)B_{\delta}^{\alpha},V_{\beta}(x)B_{\delta}^{\beta}\rangle]|+C\delta^{\frac{1}{2}}+C\delta\\
\leqslant & \frac{1}{2\delta}(C\int_{0}^{\delta}\sqrt{\mathbb{E}^{G}[|X_{s,x}-x|^{2}]}ds+C\int_{0}^{\delta}\mathbb{E}^{G}[|X_{s,x}-x|^{2}]ds+\\
 & C\delta^{\frac{1}{2}}\sqrt{\int_{0}^{\delta}\mathbb{E}^{G}[|X_{s,x}-x|^{2}]ds})+C\delta^{\frac{1}{2}}+C\delta,
\end{align*}
where we've also used the fact that $G$-It$\hat{\mbox{o}}$ integrals
and $B_{\delta}^{\alpha}B_{\delta}^{\beta}-\langle B^{\alpha},B^{\beta}\rangle_{\delta}$
have zero mean uncertainty. Here $C$ always denotes positive constants
independent of $\delta$.

Now the result follows easily from the fact that 
\[
\mathbb{E}^{G}[|X_{t,x}-x|^{2}]\leqslant Ct,\ \forall t\in[0,1].
\]

\end{proof}

The infinitesimal diffusive nature of (\ref{SDE in section 7}) characterized
by Proposition \ref{Proposition infinitesimal generator} enables
us to establish the generating PDE of (\ref{SDE in section 7}) in
terms of viscosity solutions. The understanding of this PDE, especially
its intrinsic nature, is essential for the development in a geometric
setting.
\begin{thm}
\label{Theorem PDE of SDE}Let $\varphi\in C_{b}^{\infty}(\mathbb{R}^{N}),$
and define
\[
u(t,x)=\mathbb{E}^{G}[\varphi(X_{t,x})],\ (t,x)\in[0,1]\times\mathbb{R}^{N}.
\]
Then $u(t,x)$ is the unique viscosity solution of the following nonlinear
parabolic PDE:
\begin{equation}
\begin{cases}
\frac{\partial u}{\partial t}-G((\widehat{V_{\alpha}V_{\beta}}u)_{1\leqslant\alpha,\beta\leqslant d})=0,\\
u(0,x)=\varphi(x),
\end{cases}\label{PDE of SDE}
\end{equation}
where $\widehat{V_{\alpha}V_{\beta}}$ denotes the symmetrization
of the second order differential operator $V_{\alpha}V_{\beta},$
that is, 
\[
\widehat{V_{\alpha}V_{\beta}}=\frac{1}{2}(V_{\alpha}V_{\beta}+V_{\beta}V_{\alpha}).
\]
\end{thm}
\begin{proof}
The continuity of $u$ in $t$ and $x$ can be shown in a standard
way by using the Lipschitz continuity of $\varphi$ (in fact, $u$
is Lipchitz in $x$ and $\frac{1}{2}$-H$\ddot{\mbox{o}}$lder continuous
in $t$). Here the proof is omitted.

Fix $(t_{0},x_{0})\in(0,1)\times\mathbb{R}^{N}.$ Let $v(t,x)\in C_{b}^{2,3}([0,1]\times\mathbb{R}^{N})$
be a test function such that 
\[
u(t_{0},x_{0})=v(t_{0},x_{0})
\]
and 
\[
u(t,x)\leqslant v(t,x),\ \forall(t,x)\in[0,1]\times\mathbb{R}^{N}.
\]
For $0<\delta<t_{0},$ by the uniqueness of the SDE (\ref{SDE in section 7})
and the fact that $B_{t}$ and $\langle B^{\alpha},B^{\beta}\rangle_{t}$
have independent and identically distributed increments, we know that
\begin{eqnarray*}
\mathbb{E}^{G}[\varphi(X_{t_{0},x_{0}})|\Omega_{\delta}] & = & \mathbb{E}^{G}[\varphi(X_{\delta,x_{0}}+\int_{\delta}^{t_{0}}V_{\alpha}(X_{s,x_{0}})dB_{s}^{\alpha}\\
 &  & +\frac{1}{2}\int_{\delta}^{t_{0}}DV_{\alpha}(X_{s,x_{0}})\cdot V_{\beta}(X_{s,x_{0}})d\langle B^{\alpha},B^{\beta}\rangle_{s})|\Omega_{\delta}]\\
 & = & \mathbb{E}^{G}[\varphi(X_{t_{0}-\delta,y})]|_{y=X_{\delta,x_{0}}}.
\end{eqnarray*}
Therefore, 
\begin{eqnarray*}
v(t_{0},x_{0}) & = & \mathbb{E}^{G}[\varphi(X_{t_{0},x_{0}})]\\
 & = & \mathbb{E}^{G}[\mathbb{E}^{G}[\varphi(X_{t_{0},x_{0}})|\Omega_{\delta}]]\\
 & = & \mathbb{E}^{G}[u(t_{0}-\delta,X_{\delta,x_{0}})]\\
 & \leqslant & \mathbb{E}^{G}[v(t_{0}-\delta,X_{\delta,x_{0}})].
\end{eqnarray*}
It follows that
\begin{eqnarray*}
0 & \leqslant & \mathbb{E}^{G}[v(t_{0}-\delta,X_{\delta,x_{0}})-v(t_{0},x_{0})]\\
 & = & \mathbb{E}^{G}[v(t_{0}-\delta,X_{\delta,x_{0}})-v(t_{0},X_{\delta,x_{0}})+v(t_{0},X_{\delta,x_{0}})-v(t_{0},x_{0})]\\
 & = & \mathbb{E}^{G}[-\delta\int_{0}^{1}\frac{\partial v}{\partial t}(t_{0}-(1-\alpha)\delta,X_{\delta,x_{0}})d\alpha+\langle\nabla v(t_{0},x_{0}),X_{\delta,x_{0}}-x_{0}\rangle\\
 &  & +\int_{0}^{1}\int_{0}^{1}\langle\nabla^{2}v(t_{0},x_{0}+\alpha\beta(X_{\delta,x_{0}}-x_{0}))(X_{\delta,x_{0}}-x_{0}),X_{\delta,x_{0}}-x_{0}\rangle\alpha d\alpha d\beta]\\
 & \leqslant & -\delta\frac{\partial v}{\partial t}(t_{0},x_{0})+\mathbb{E}^{G}[\langle\nabla v(t_{0},x_{0}),X_{\delta,x_{0}}-x_{0}\rangle\\
 &  & +\frac{1}{2}\langle\nabla^{2}v(t_{0},x_{0})(X_{\delta,x_{0}}-x_{0}),X_{\delta,x_{0}}-x_{0}\rangle]+\mathbb{E}^{G}[|I_{\delta}|]+\mathbb{E}^{G}[|J_{\delta}|],
\end{eqnarray*}
where 
\begin{eqnarray*}
I_{\delta} & = & -\delta\int_{0}^{1}(\frac{\partial v}{\partial t}(t_{0}-(1-\alpha)\delta,X_{\delta,x_{0}})-\frac{\partial v}{\partial t}(t_{0},x_{0}))d\alpha,\\
J_{\delta} & = & \int_{0}^{1}\int_{0}^{1}\langle(\nabla^{2}v(t_{0},x_{0}+\alpha\beta(X_{\delta,x_{0}}-x_{0}))-\nabla^{2}v(t_{0},x_{0}))(X_{\delta,x_{0}}-x_{0}),X_{\delta,x_{0}}-x_{0}\rangle\alpha d\alpha d\beta.
\end{eqnarray*}

By a standard argument one can easily show that 
\[
\mathbb{E}^{G}[|I_{\delta}|]+\mathbb{E}^{G}[|J_{\delta}|]\leqslant C\delta^{\frac{3}{2}},
\]
where $C$ is a positive constant independent of $\delta.$ On the
other hand, the R.H.S. of (\ref{infinitesimal generator}) applying
to 
\[
p=\nabla v(t_{0},x_{0}),\ A=\nabla^{2}v(t_{0},x_{0}),
\]
is exactly the same as $G((\widehat{V_{\alpha}V_{\beta}}v(t_{0},x_{0}))_{1\leqslant\alpha,\beta\leqslant d}).$
Therefore, by Proposition \ref{Proposition infinitesimal generator},
we arrive at
\[
\frac{\partial v}{\partial t}(t_{0},x_{0})-G((\widehat{V_{\alpha}V_{\beta}}v(t_{0},x_{0}))_{1\leqslant\alpha,\beta\leqslant d})\leqslant0.
\]
Consequently, $ $$u(t,x)$ is a viscosity subsolution of (\ref{PDE of SDE}).

Similarly, one can show that $u(t,x)$ is a viscosity supersolution
of (\ref{PDE of SDE}). Therefore, $u(t,x)$ is a viscosity solution
of (\ref{PDE of SDE}).

The reason of uniqueness is the following. Define a function $F:\ \mathbb{R}^{N}\times\mathbb{R}^{N}\times S(N)\rightarrow\mathbb{R}$
by the R.H.S. of (\ref{infinitesimal generator}), that is, 
\[
F(x,p,A)=G((\frac{1}{2}\langle p,DV_{\alpha}(x)\cdot V_{\beta}(x)+DV_{\beta}(x)\cdot V_{\alpha}(x)\rangle+\langle AV_{\alpha}(x),V_{\beta}(x)\rangle)_{1\leqslant\alpha,\beta\leqslant d}),
\]
for $(x,p,A)\in\mathbb{R}^{N}\times\mathbb{R}^{N}\times S(N).$ It
is easy to prove that $F$ is sublinear in $(p,A)$ and monotonically
increasing in $S(N),$ due to the same properties held by $G$. Moreover,
$F$ satisfies the continuity condition (Assumption (G) in Appendix
C of \cite{peng2010nonlinear}) for the uniqueness of the associated
nonlinear PDE, due to the regularity of the given vector fields $V_{\alpha}.$
In other words, all properties of $G$ to ensure uniqueness are preserved
in $F,$ and the space dependence of $F$ coming out are uniformly
controlled. Therefore, according to the uniqueness results (see \cite{crandall1992user},
\cite{peng2010nonlinear}), the parabolic PDE has a unique viscosity
solution, which is given by $u(t,x).$
\end{proof}

\begin{example}
\label{Euclidean G-Brownian motino under transformation}An example
which motivates the study of $G$-Brownian motion on a Riemannian
manifold is the following.

Let $Q\in GL(d,\mathbb{R})$, where $GL(d,\mathbb{R})$ is the group
of $d\times d$ real invertible matrices. Define $B_{t}^{Q}=QB_{t},$
and for $\varphi\in C_{b}^{\infty}(\mathbb{R}^{d}),$ define 
\[
u(t,x)=\mathbb{E}^{G}[\varphi(x+B_{t}^{Q})],\ (t,x)\in[0,1]\times\mathbb{R}^{d}.
\]
Then $u(t,x)$ is the unique viscosity solution of the PDE:
\[
\begin{cases}
\frac{\partial u}{\partial t}-G(Q^{T}\cdot\nabla^{2}u\cdot Q)=0,\\
u(0,x)=\varphi(x).
\end{cases}
\]
In fact, it follows directly from Theorem \ref{Theorem PDE of SDE}
if we regard $x+B_{t}^{Q}$ as the solution of the SDE over $[0,1]$:
\begin{equation}
\begin{cases}
dX_{t,x}=Q_{\alpha}\circ dB_{t}^{\alpha},\\
X_{0,x}=x,
\end{cases}\label{Brownian motion transformed}
\end{equation}
where $Q=(Q_{1},\cdots,Q_{d}),$ and each $Q_{\alpha}$ is a constant
vector field on $\mathbb{R}^{d}$ (so the SDE (\ref{Brownian motion transformed})
coincides exactly with the It$\hat{\mbox{o}}$ type one).
\end{example}

The result of Theorem \ref{PDE of SDE} is similar to the discussion
of nonlinear Feynman-Kac formula in \cite{peng2010nonlinear}, in
which the solution of a forward-backward SDE is used to represent
the viscosity solution of an associated nonlinear backward parabolic
PDE. In our case, the intrinsic nature of (\ref{PDE of SDE}) is fundamental
and should be emphasized below in order to develop $G$-Brownian motion
on a Riemannian manifold.

It is not hard to see that the nonlinear second order differential
operator $G((\widehat{V_{\alpha}V_{\beta}}\cdot)_{1\leqslant\alpha,\beta\leqslant d})$
is intrinsically defined on $\mathbb{R}^{N},$ since $V_{1},\cdots,V_{d}$
are vector fields independent of coordinates. Moreover, in local coordinates
it preserves the same properties of the $G$-function which is defined
under the standard coordinate system of $\mathbb{R}^{d}.$ In particular,
it shares the same ellipticity as $G$. Therefore, when the vector
fields $V_{\alpha}$ are regular enough, from our results in Section
6, we are able to establish the generating PDE of a nonlinear diffusion
process on a differentiable manifold. As in the last section, for
technical simplicity we restrict ourselves to compact manifolds.

Assume that $M$ is a compact manifold, and $V_{1},\cdots,V_{d}$
are $C^{3}$-vector fields on $M$. According to Section 6, the Stratonovich
type SDE over $[0,1]$
\begin{equation}
\begin{cases}
dX_{t,x}=V_{\alpha}(X_{t,x})\circ dB_{t}^{\alpha},\\
X_{0,x}=x\in M,
\end{cases}\label{SDE on manifold in Section 7}
\end{equation}
has a unique solution. The following result is immediate from Theorem
\ref{Theorem PDE of SDE}.
\begin{thm}
\label{Theorem PDE of SDE on manifold}Let $\varphi\in C^{\infty}(M),$
and define 
\[
u(t,x)=\mathbb{E}^{G}[\varphi(X_{t,x})],\ (t,x)\in[0,1]\times M,
\]
then $u(t,x)$ is the unique viscosity solution of the following nonlinear
parabolic PDE on $M$:
\begin{equation}
\begin{cases}
\frac{\partial u}{\partial t}-G((\widehat{V_{\alpha}V_{\beta}}u)_{1\leqslant\alpha,\beta\leqslant d})=0,\\
u(0,x)=\varphi(x),
\end{cases}\label{PDE of SDE on manifold}
\end{equation}
where $\widehat{V_{\alpha}V_{\beta}}$ is the symmetrization of $V_{\alpha}V_{\beta},$
defined in the same way as in Theorem \ref{Theorem PDE of SDE}. Here
the notion of viscosity solutions for the PDE (\ref{PDE of SDE on manifold})
can be defined in the same way as in the Euclidean case by using test
functions (see \cite{azagra2008viscosity}).\end{thm}
\begin{proof}
The result follows easily from an extrinsic point of view. 

In fact, assume that $M$ is embedded into an ambient Euclidean space
$\mathbb{R}^{N}$ as a closed submanifold, and take a $C^{3}$-extension
$\widetilde{V}_{\alpha}$ of $V_{\alpha}$ with compact support. Consider
the following Stratonovich type SDE over $[0,1]$:
\[
\begin{cases}
dX_{t,x}=\widetilde{V}_{\alpha}(X_{t,x})\circ dB_{t}^{\alpha},\\
X_{0,x}=x\in\mathbb{R}^{N}.
\end{cases}
\]
Let $\widetilde{\varphi}$ be a $C^{\infty}$-extension of $\varphi$
with compact support, and define 
\[
\widetilde{u}(t,x)=\mathbb{E}^{G}[\widetilde{\varphi}(X_{t,x})],\ (t,x)\in[0,1]\times\mathbb{R}^{N}.
\]
It follows from Theorem \ref{Theorem PDE of SDE} that $\widetilde{u}(t,x)$
is the unique viscosity solution of the nonlinear parabolic PDE generated
by the vector fields $\widetilde{V}_{\alpha}.$ 

According to Section 6, if $x\in M,$ $X_{t,x}$ will never leave
$M$ quasi-surely. Therefore, when restricted on $M,$ $\widetilde{u}=u.$
In particular, we know that $u$ is continuous. To see that $u$ is
a viscosity subsolution of (\ref{PDE of SDE on manifold}), let $(t_{0},x_{0})\in(0,1)\times M,$
and $v(t,x)\in C^{2,3}([0,1]\times M)$ be a test function such that
\[
v(t_{0},x_{0})=u(t_{0},x_{0})
\]
and 
\[
u(t,x)\leqslant v(t,x),\ \forall(t,x)\in[0,1]\times M.
\]
Take an $C_{b}^{2,3}$-extension $\widetilde{v}$ of $v$ such that
\[
\widetilde{u}(t,x)\leqslant\widetilde{v}(t,x),\ \forall(t,x)\in[0,1]\times\mathbb{R}^{N}.
\]
It follows from previous discussion that 
\[
\frac{\partial\widetilde{v}}{\partial t}(t_{0},x_{0})-G((\widehat{\widetilde{V}_{\alpha}\widetilde{V}_{\beta}}\widetilde{v}(t_{0},x_{0}))_{1\leqslant\alpha,\beta\leqslant d})\leqslant0.
\]
Since 
\[
\widetilde{V}_{\alpha}|_{M}=V_{\alpha},\ \widetilde{v}|_{M}=v,
\]
from the intrinsic nature of the generating PDE, we know that 
\[
\frac{\partial\widetilde{v}}{\partial t}(t_{0},x_{0})=\frac{\partial v}{\partial t}(t_{0},x_{0})
\]
and
\[
G((\widehat{\widetilde{V}_{\alpha}\widetilde{V}_{\beta}}\widetilde{v}(t_{0},x_{0}))_{1\leqslant\alpha,\beta\leqslant d})=G((\widehat{V_{\alpha}V_{\beta}}v(t_{0},x_{0}))_{1\leqslant\alpha,\beta\leqslant d}).
\]
It follows that 
\[
\frac{\partial v}{\partial t}(t_{0},x_{0})-G((\widehat{V_{\alpha}V_{\beta}}v(t_{0},x_{0}))_{1\leqslant\alpha,\beta\leqslant d})\leqslant0.
\]
Therefore, $u(t,x)$ is a viscosity subsolution of (\ref{PDE of SDE on manifold}).
Similarly we can show that it is a viscosity supersolution as well,
and thus a viscosity solution.

The uniqueness of (\ref{PDE of SDE on manifold}) follows from the
same reason as in the proof of Theorem \ref{Theorem PDE of SDE} once
we notice that the second order differential operator $G((\widehat{V_{\alpha}V_{\beta}}\cdot)_{1\leqslant\alpha,\beta\leqslant d})$
on $M$ shares exactly the same properties as $G$ (in particular,
the same ellipticity), which can be seen either from an extrinsic
way or via local computation. Another way to see the uniqueness is
to use the results in \cite{azagra2008viscosity} as long as we assign
a complete Riemannian metric on $M,$ which is always possible according
to \cite{nomizu1961existence}. In this case 
\[
G((\widehat{V_{\alpha}V_{\beta}}u)_{1\leqslant\alpha,\beta\leqslant d})=G((\frac{1}{2}\langle\nabla u,\nabla_{V_{\alpha}}V_{\beta}+\nabla_{V_{\beta}}V_{\alpha}\rangle+\mbox{Hess}u(V_{\alpha},V_{\beta}))_{1\leqslant\alpha,\beta\leqslant d}),
\]
where $\nabla$ is the Levi-Civita connection corresponding to the
Riemannian metric. The uniqueness of (\ref{PDE of SDE on manifold})
follows from Theorem 5.1 in \cite{azagra2008viscosity} directly,
as the assumptions in the theorem are verified by the properties of
$G$. Note that we don't need the Ricci curvature condition in \cite{azagra2008viscosity}
due to the compactness of $M$ and uniform continuity of $G((\widehat{V_{\alpha}V_{\beta}}\cdot)_{1\leqslant\alpha,\beta\leqslant d}).$
\end{proof}

\begin{rem}
The study of the SDE (\ref{SDE on manifold in Section 7}) as a nonlinear
diffusion process on $M$ does not require a Riemannian metric or
a connection on $M.$ The fundamental reason is that (\ref{SDE on manifold in Section 7})
is defined in the pathwise sense as an RDE generated by the vector
fields $V_{\alpha}$ on $M$. Such an RDE only depends on the differential
structure of $M.$ The infinitesimal diffusive nature of (\ref{SDE on manifold in Section 7})
can be studied by local computation.
\end{rem}
Now we turn to the study of $G$-Brownian motion on a Riemannian manifold.
The Riemannian structure (the Levi-Civita connection) is used to ``roll''
the Euclidean $G$-Brownian motion up to the manifold ``without slipping''
by solving an SDE generated by the fundamental horizontal vector fields
on a proper frame bundle (known as horizontal lifting). This is the
fundamental idea of Eells-Elworthy-Malliavin on the construction of
Brownian motion on a Riemannian manifold.

As is pointed out at the beginning of this section, the essential
point of such development is the invariance of the generating PDE
on the frame bundle under actions by the structure group along fibers.
The key of capturing such invariance is Theorem \ref{PDE of SDE}
and Example \ref{Euclidean G-Brownian motino under transformation},
which leads to the following important concept.
\begin{defn}
The invariant group $I(G)$ of $G$ is defined by 
\[
I(G)=\{Q\in GL(d,\mathbb{R}):\ \forall A\in S(d),\ G(Q^{T}AQ)=G(A)\}.
\]

\end{defn}

It is easy to check the $I(G)$ is a group, and hence a subgroup of
$GL(d,\mathbb{R}).$

By using the representation (\ref{rep of G in section 7}) of $G$,
we have the following equivalent characterization of the invariant
group $I(G).$
\begin{prop}
Let $G$ be represented by
\[
G(A)=\frac{1}{2}\sup_{B\in\Sigma}\mbox{tr}(AB),\ \forall A\in S(d),
\]
where $\Sigma$ is some bounded, closed and convex subset of $S_{+}(d).$
Then $\Sigma$ is uniquely determined by $G$ and the invariant group
$I(G)$ of $G$ is given by 
\begin{equation}
I(G)=\{Q\in GL(d,\mathbb{R}):\ Q\Sigma Q^{T}=\Sigma\}.\label{rep of I(G)}
\end{equation}
\end{prop}
\begin{proof}
It suffices to show the uniqueness of $\Sigma$, and (\ref{rep of I(G)})
will follow immediately from the commutativity of the trace operator
and the uniqueness of $\Sigma.$ Note that for any $Q\in GL(d,\mathbb{R}),$
$Q\Sigma Q^{T}$ is also a bounded, closed and convex subset of $S_{+}(d).$

Introduce a symmetric bilinear form $\langle\cdot,\cdot\rangle_{\mbox{tr}}$
on the finite dimensional vector space $S(d)$ by 
\[
\langle A_{1},A_{2}\rangle_{\mbox{tr}}=\mbox{tr}(A_{1}A_{2}),\ A_{1},A_{2}\in S(d).
\]
It is easy to check that $\langle\cdot,\cdot\rangle_{\mbox{tr}}$
is indeed an inner product, thus $(S(d),\langle\cdot,\cdot\rangle_{\mbox{tr}})$
is a finite dimensional Hilbert space. The form $\|\cdot\|_{tr}$
induced by $\langle\cdot,\cdot\rangle_{\mbox{tr}}$ is equivalent
to any other matrix norm on $S(d)$ since $S(d)$ is finite dimensional.

Let $\Sigma_{1},\Sigma_{2}$ be two bounded, closed and convex subsets
of $S_{+}(d),$ such that 
\[
\sup_{B\in\Sigma_{1}}\mbox{tr}(AB)=\sup_{B\in\Sigma_{2}}\mbox{tr}(AB),\ \forall A\in S(d).
\]
If $\Sigma_{1}\neq\Sigma_{2},$ without loss of generality assume
that $B_{0}\in\Sigma_{2}\backslash\Sigma_{1}.$ According to the Mazur
separation theorem in functional analysis (see \cite{yosida1980functional}),
there exists a bounded linear functional $f\in S(d)^{*}$ and some
$\alpha\in\mathbb{R},$ such that 
\[
f(B)<\alpha<f(B_{0}),\ \forall B\in\Sigma_{1}.
\]
By the Riesz representation theorem, there exists a unique $A^{*}\in S(d),$
such that 
\[
f(B)=\langle A^{*},B\rangle_{\mbox{tr}}=\mbox{tr}(A^{*}B),\ \forall B\in S(d).
\]
It follows that 
\[
\sup_{B\in\Sigma_{1}}\mbox{tr}(A^{*}B)\leqslant\alpha<\mbox{tr}(A^{*}B_{0})\leqslant\sup_{B\in\Sigma_{2}}\mbox{tr}(A^{*}B),
\]
which is a contradiction. Therefore, $\Sigma_{1}=\Sigma_{2}.$
\end{proof}

We list some examples for the invariant groups $I(G)$ of different
$G$-functions.
\begin{example}
If $\Sigma=\{0\},$ then it is obvious that $I(G)=GL(d,\mathbb{R}),$
which is a noncompact group.
\end{example}

\begin{example}
\label{Example finite group}It is possible that $I(G)$ is a finite
group.

Consider $\Sigma$ is the set of diagonal matrices 
\[
\Lambda=\mbox{diag}(\lambda_{1},\cdots,\lambda_{d})
\]
such that each $\lambda_{\alpha}\in[0,1],$ then $\Sigma$ is a bounded,
closed and convex subset of $S_{+}(d).$ We claim that 
\begin{equation}
I(G)=\{(\pm e_{\sigma(1)},\cdots,\pm e_{\sigma(d)}):\ \sigma\mbox{ is a permutation of order \ensuremath{d}}\},\label{finite group}
\end{equation}
where $\{e_{1},\cdots,e_{d}\}$ is the standard orthonormal basis
of $\mathbb{R}^{d}$, each $e_{i}$ being regarded as a column vector. 

In fact, if $Q\in GL(d,\mathbb{R})$ has the form (\ref{finite group}),
by direct computation one can show easily that 
\begin{equation}
Q\Sigma Q^{T}=\Sigma.\label{invariant Q Sigma}
\end{equation}
Conversely, if $Q$ satisfies (\ref{invariant Q Sigma}), by choosing
\[
\Lambda=\mbox{diag}(1,0,\cdots,0),
\]
we know that 
\[
(Q\Lambda Q^{T})_{\beta}^{\alpha}=Q_{1}^{\alpha}Q_{1}^{\beta}.
\]
Therefore, if $Q\Lambda Q^{T}\in\Sigma,$ the first column of $Q$
must contain exactly one nonzero element $q_{1}$ such that $q_{1}^{2}\leqslant1$.
Similarly for other columns of $Q$. Moreover, the corresponding nonzero
elements in any two different columns of $Q$ must be in different
rows, otherwise $Q$ will be degenerate. Consequently, $Q$ has the
form 
\[
Q=(q_{1}e_{\sigma(1)},\cdots,q_{d}e_{\sigma(d)})
\]
with $q_{i}^{2}\leqslant1$ ($i=1,2,\cdots,d$). On the other hand,
for the identity matrix $I_{d},$ there exists $\Lambda\in\Sigma,$
such that 
\[
Q\Lambda Q^{T}=I_{d}.
\]
By taking determinants on both sides, we have 
\[
q_{1}^{2}\cdots q_{d}^{2}\mbox{det}(\Lambda)=1,
\]
which implies that $q_{\alpha}=\pm1$ ($\alpha=1,2,\cdots,d$). Therefore,
$Q$ has the form of (\ref{finite group}).

Note that in this case $I(G)$ is a finite subgroup of the orthogonal
group $O(d)$ with order $2^{d}d!.$ Moreover, $G$ is given by 
\[
G(A)=\frac{1}{2}\sum_{\alpha=1}^{d}(A_{\alpha}^{\alpha})^{+},\ \forall A\in S(d).
\]

\end{example}

\begin{example}
\label{example O(d)}Now we give some examples of $G$ such that $I(G)=O(d).$
Such case will be our main interest in this paper.

(1) $\Sigma=\{I_{d}\}.$ 

Obviously (\ref{invariant Q Sigma}) is equivalent to $Q\in O(d).$ 

This corresponds to the case of classical Brownian motion, in which
\[
G(A)=\frac{1}{2}\mbox{tr}(A)
\]
and the generator is $\frac{1}{2}\Delta.$ 

(2) $\Sigma$ is given by the segment joining $\lambda I_{d}$ and
$\mu I_{d}$, where $0\leqslant\lambda<\mu.$

If $Q\in GL(d,\mathbb{R})$ such that (\ref{invariant Q Sigma}) holds,
then 
\[
\mu QQ^{T}=tI_{d},
\]
for some $t\in[\lambda,\mu].$ On the other hand, there exists some
$t'\in[\lambda,\mu]$ such that
\[
t'QQ^{T}=\mu I_{d}.
\]
The only possibility is that $QQ^{T}=I_{d},$ which means $Q\in O(d).$
The converse is trivial.

In this case, $G$ is given by 
\[
G(A)=\frac{1}{2}(\mu(\mbox{tr}A)^{+}-\lambda(\mbox{tr}A)^{-}).
\]
The corresponding $G$-heat equation can be regarded as the generalization
of the one-dimensional Barenblatt equation to higher dimensions.

(3) $\Sigma$ is given by the subset of matrices $B\in S_{+}(d)$
such that the eigenvalues of $B$ lie in the bounded interval $[\lambda,\mu]$,
where $0\leqslant\lambda<\mu.$ Equivalently,
\[
\Sigma=\{B\in S_{+}(d):\ \lambda\leqslant x^{T}Bx\leqslant\mu,\ \forall x\in\mathbb{R}^{d}\mbox{ with \ensuremath{|x|=1}}\}.
\]
It follows that $\Sigma$ is a bounded, closed and convex subset of
$S_{+}(d).$

Since $\Sigma$ is characterized by eigenvalues, and the eigenvalues
of a symmetric matrix is preserved under change of orthonormal basis,
it follows that for any $Q\in O(d),$ (\ref{invariant Q Sigma}) holds.
Conversely, let $Q\in GL(d,\mathbb{R})$ with (\ref{invariant Q Sigma}).
Then there exists $B_{1},B_{2}\in\Sigma$, such that 
\[
\mu QQ^{T}=B_{1},\ QB_{2}Q^{T}=\mu I_{d}.
\]
It follows that all eigenvalues of $QQ^{T}$ lie in $[\frac{\lambda}{\mu},1],$
and 
\[
\mbox{det}(QQ^{T})\mbox{det}(B_{2})=\mu^{d}.
\]
Therefore, the only possibility is that all eigenvalues of $QQ^{T}$
are equal to $1,$ which implies that $Q$ is an orthogonal matrix.

In this case $G$ can be expressed by 
\begin{eqnarray*}
G(A) & = & \frac{1}{2}\sup_{B\in\Sigma}\mbox{tr}(AB)\\
 & = & \frac{1}{2}\sup_{P\in O(d)}\sup_{\lambda\leqslant c_{1},\cdots,c_{d}\leqslant\mu}\mbox{tr}(AP^{T}\mbox{diag}(c_{1},\cdots,c_{d})P)\\
 & = & \frac{1}{2}\sup_{P\in O(d)}\sup_{\lambda\leqslant c_{1},\cdots,c_{d}\leqslant\mu}\mbox{tr}(PAP^{T}\mbox{diag}(c_{1},\cdots,c_{d}))\\
 & = & \frac{1}{2}\sup_{P\in O(d)}\sup_{\lambda\leqslant c_{1},\cdots,c_{d}\leqslant\mu}\sum_{\alpha=1}^{d}c_{\alpha}(PAP^{T})_{\alpha}^{\alpha}\\
 & = & \frac{1}{2}\sup_{P\in O(d)}\sum_{\alpha=1}^{d}(\mu((PAP^{T})_{\alpha}^{\alpha})^{+}-\lambda((PAP^{T})_{\alpha}^{\alpha})^{-}).
\end{eqnarray*}

\end{example}

Similar to Example \ref{example O(d)}, for those $\Sigma$'s characterized
by eigenvalues, we can construct a large class of $G$ such that $I(G)=O(d).$
\begin{rem}
If $\Sigma$ has at least one nondegenerate element, that is, there
exists some positive definite matrix $B_{0}\in\Sigma,$ then $I(G)$
is a compact group. In fact, if we introduce a matrix norm $\|\cdot\|_{B_{0}}$
on the space $\mbox{Mat}(d,\mathbb{R})$ of real $d\times d$ matrices
by 
\[
\|A\|_{B_{0}}=\sqrt{\mbox{tr}(AB_{0}A^{T})},\ A\in\mbox{Mat}(d,\mathbb{R}),
\]
it follows that 
\[
\sup_{Q\in I(G)}\|Q\|_{B_{0}}=\sup_{Q\in I(G)}\sqrt{\mbox{tr}(QB_{0}Q^{T})}\leqslant\sup_{B\in\Sigma}\sqrt{\mbox{tr}(B)}<\infty,
\]
since $\Sigma$ is bounded. It is obvious that $I(G)$ is closed.
Therefore, it is compact.
\end{rem}

Now assume that $(M,g)$ is a $d$-dimensional compact Riemannian
manifold. If we allow explosion of a nonlinear diffusion process at
some finite time, then the arguments below will carry through on a
noncompact Riemannian manifold as long as the time scope is restricted
from $0$ up to the explosion. Here we only consider the compact case,
in which explosion is not possible.

We first recall some basics about frame bundles, which is the central
concept in the horizontal lifting construction. For a systematic introduction
please refer to \cite{chern1999lectures}, \cite{kobayashi1996foundations}.

Let $\mathcal{F}(M)$ be the total frame bundle over $M$ defined
by 
\[
\mathcal{F}(M)=\cup_{x\in M}\mathcal{F}_{x}(M),
\]
where the fibre $\mathcal{F}_{x}(M)$ is the set of all frames (bases
of the tangent space $T_{x}(M)$) at $x$. A frame $\xi=(\xi_{1},\cdots,\xi_{d})\in\mathcal{F}_{x}(M)$
can be equivalently regarded as a linear isomorphism from $\mathbb{R}^{d}$
to $T_{x}M$ (also denoted by $\xi$) if we let 
\[
\xi(e_{\alpha})=\xi_{\alpha},\ \alpha=1,2,\cdots,d,
\]
and extend linearly to $\mathbb{R}^{d},$ where we always fix $\{e_{1},\cdots,e_{d}\}$
to be the standard orthonormal basis of $\mathbb{R}^{d}.$ $\mathcal{F}(M)$
is a principal bundle with structure group $GL(d,\mathbb{R})$ acting
along fibers from the right.

Fix a frame $\xi\in\mathcal{F}_{x}(M)$. A vector $X\in T_{\xi}\mathcal{F}(M)$
is called vertical if it is tangent to the fibre $\mathcal{F}_{x}(M)$.
The space of vertical vectors at $\xi$ is called the vertical subspace,
and it is denoted by $V_{\xi}\mathcal{F}(M).$ $V_{\xi}\mathcal{F}(M)$
is a $d^{2}$-dimensional vector space, which is independent of the
Riemannian structure. 

A smooth curve $\xi_{t}=(\xi_{1,t}\cdots,\xi_{d,t})\in\mathcal{F}(M)$
is called horizontal if $\xi_{\alpha,t}$ is a parallel vector field
along the projection curve $x_{t}=\pi(\xi_{t})$ for each $\alpha=1,2,\cdots,d.$
Given a smooth curve $x_{t}\in M$ and a frame $\xi_{0}=(\xi_{1},\cdots,\xi_{d})\in\mathcal{F}_{x_{0}}(M),$
by solving a first order linear ODE, we can determine a unique parallel
vector field $\xi_{\alpha,t}$ along $x_{t}$ with $\xi_{\alpha,0}=\xi_{\alpha}$
for each $\alpha=1,2,\cdots,d.$ The smooth curve 
\[
\xi_{t}=(\xi_{1,t},\cdots,\xi_{d,t})\in\mathcal{F}(M)
\]
is then the unique horizontal curve with $x_{t}=\pi(\xi_{t})$ and
initial position $\xi_{0}$. $\xi_{t}$ is called the horizontal lifting
of $x_{t}$ from $\xi_{0}.$ A vector $X\in T_{\xi}\mathcal{F}(M)$
is called horizontal if it is tangent to a horizontal curve through
$\xi$. The space of horizontal vectors at $\xi$ is called the horizontal
subspace, and it is denoted by $H_{\xi}\mathcal{F}(M).$ It is a $d$-dimensional
vector space characterized by the Levi-Civita connection $\nabla$. 

As $\xi$ varies, $V_{\xi}\mathcal{F}(M)$ (respectively, $H_{\xi}\mathcal{F}(M)$)
determines a vertical (respectively, horizontal) subspace field on
$M.$ The following result reveals the fundamental structure of $\mathcal{F}(M)$.
\begin{thm}
\label{bundle decomposition}The horizontal subspace field $H\mathcal{F}(M)$,
which is determined by $\nabla,$ has the following properties.

(1) For each $\xi\in\mathcal{F}_{x}(M),$ the tangent space $T_{\xi}\mathcal{F}(M)$
has the decomposition
\[
T_{\xi}\mathcal{F}(M)=H_{\xi}\mathcal{F}(M)\oplus V_{\xi}\mathcal{F}(M).
\]
Moreover, $H_{\xi}\mathcal{F}(M)$ is isomorphic to $T_{x}M$ under
the canonical projection $\pi:\ \mathcal{F}(M)\rightarrow M.$ 

(2) $H\mathcal{F}(M)$ is invariant under actions by the structure
group $GL(d,\mathbb{R}).$ More precisely, for any $\xi\in\mathcal{F}(M),\ Q\in GL(d,\mathbb{R}),$
\[
Q_{*}(H_{\xi}\mathcal{F}(M))=H_{\xi Q}\mathcal{F}(M).
\]

\end{thm}

It should be pointed out that given any horizontal subspace field
$H\mathcal{F}(M)$ satisfying the two properties in Theorem \ref{bundle decomposition},
there exists an affine connection $\nabla^{H}$ such that $H\mathcal{F}(M)$
is the horizontal subspace field determined by $\nabla^{H}.$ 

On $\mathcal{F}(M)$ there is a canonical way to define a frame field
globally, which is not always possible on a general Riemannian manifold.
This makes $\mathcal{F}(M)$ simpler than the base space $M$ in some
sense. Fix $w\in\mathbb{R}^{d}.$ For any $\xi\in\mathcal{F}_{x}(M)$
regarded as a linear isomorphism $\xi:\ \mathbb{R}^{d}\rightarrow T_{x}M,$
$\xi(w)$ is a tangent vector in $T_{x}M$. By Theorem \ref{bundle decomposition}
(1), $\xi(w)$ corresponds to a unique vector $H_{w}(\xi)\in H_{\xi}\mathcal{F}(M)$.
It follows that $H_{w}$ is a globally defined horizontal vector field
on $\mathcal{F}(M).$ If we take $w=e_{\alpha}$ ($\alpha=1,2,\cdots,d$),
then we obtain a family of horizontal vector fields $\{H_{e_{1}},\cdots,H_{e_{d}}\}$
as a basis of the horizontal subspace $H_{\xi}\mathcal{F}(M)$ at
each frame $\xi\in\mathcal{F}(M).$ $\{H_{e_{1}},\cdots,H_{e_{d}}\}$
are called the fundamental horizontal fields of $\mathcal{F}(M)$,
simply denoted by $\{H_{1},\cdots,H_{d}\}.$

Now we introduce the concept of development and anti-development (see
\cite{hsu2002stochastic}), which is crucial in the construction of
$G$-Brownian motion on $M.$ Assume that $x_{t}\in M$ is a smooth
curve and $\xi_{t}$ is the horizontal lifting of $x_{t}$ from $\xi_{0}.$
Then we can determine a smooth curve 
\[
w_{t}=\int_{0}^{t}\xi_{s}^{-1}\dot{x}_{s}ds\in\mathbb{R}^{d}
\]
starting from $0$ ($w_{t}$ is regarded as a column vector in $\mathbb{R}^{d}$).
$w_{t}$ is called the anti-development of $x_{t}$ in $\mathbb{R}^{d}$
with respect to $\xi_{0}.$ If $\xi_{t}$ and $\eta_{t}$ are two
horizontal liftings of $x_{t}$ with $\xi_{0}=\eta_{0}Q$ for some
$Q\in GL(d,\mathbb{R}),$ then the two corresponding anti-developments
are related by 
\[
w_{t}^{\eta}=Qw_{t}^{\xi}.
\]
The fundamental relation between the anti-development $w_{t}$ of
$x_{t}$ and the horizontal lifting $\xi_{t}$ is the following ODE
on $\mathcal{F}(M):$
\begin{equation}
d\xi_{t}=H_{\alpha}(\xi_{t})dw_{t}^{\alpha}.\label{Horizontal ODE}
\end{equation}
Conversely, given a smooth curve $w_{t}\in\mathbb{R}^{d}$ starting
from $0,$ by solving the ODE (\ref{Horizontal ODE}) on $\mathcal{F}(M)$
with initial frame $\xi_{0}$, we obtain a horizontal curve $\xi_{t}\in\mathcal{F}(M).$
The projection $x_{t}=\pi(\xi_{t})$ is called the development of
$w_{t}$ in $M$ with respect to $\xi_{0}.$ If we use another initial
frame $\eta_{0}=\xi_{0}Q^{-1}$ and the driven process $v_{t}=Qw_{t}\in\mathbb{R}^{d},$
by solving (\ref{Horizontal ODE}) from $\eta_{0}$ and projection
onto $M$ we obtain the same curve $x_{t}.$ In this way, we obtain
a one-to-one correspondence of the Euclidean curve $w_{t}$ and the
manifold curve $x_{t}$ via the horizontal curve $\xi_{t}$ in $\mathcal{F}(M)$,
which depends on the initial frame $\xi_{0}.$ The procedure of getting
$x_{t}$ from $w_{t}$ is usually known as ``rolling without slipping''.

A crucial point should be emphasized here is that such procedure is
carried out by solving the ODE (\ref{Horizontal ODE}) in the pathwise
sense, which fits well in the context of rough paths if the Euclidean
curve $w_{t}$ is interpreted as a rough path. In this case, (\ref{Horizontal ODE})
should be interpreted as an RDE. This is an important reason why we
need to develop the notion of Stratonovich type SDEs on a differentiable
manifold.

For a general Euclidean $G$-Brownian motion $B_{t}$, from Section
6 we are able to solve (\ref{Horizontal ODE}) pathwisely if the driven
curve $dw_{t}$ is replaced by $ $$dB_{t}$ in the Stratonovich sense
(or in the RDE sense). By projecting the solution $\xi_{t}\in\mathcal{F}(M)$
to the manifold $M,$ we obtain a process $X_{t}\in M$ pathwisely
which depends on the initial position $x_{0}$ and the initial frame
$\xi_{0}\in\mathcal{F}_{x_{0}}(M).$ A disadvantage of using the total
frame bundle $\mathcal{F}(M)$ is that in this way it is not possible
to write down the generating PDE governing the law of $X_{t}$ intrinsically
on $M$, which does not depend on the initial frame $\xi_{0}.$ Note
that the generating PDE of $\xi_{t}$ is well-defined on $\mathcal{F}(M)$
according to Theorem \ref{Theorem PDE of SDE on manifold}, which
takes the  form 
\begin{equation}
\frac{\partial u}{\partial t}-G((\widehat{H_{\alpha}H_{\beta}}u)_{1\leqslant\alpha,\beta\leqslant d})=0.\label{PDE on total frame bundle}
\end{equation}
The main reason for such disadvantage is that the PDE (\ref{PDE on total frame bundle})
is not invariant under actions by $GL(d,\mathbb{R})$ along fibers,
since the $G$-function does not have such kind of invariance.

To fix this issue, a possible way is to use the invariant group $I(G)$
of $G$ as the structure group, so that the generating PDE will be
invariant under actions by $I(G)$ along fibers due to the form (\ref{PDE on total frame bundle})
it takes. Therefore, we need to use a proper frame bundle (a submanifold
of $\mathcal{F}(M)$ which is a principal bundle over $M$ with structure
group $I(G)$ and fibers being a suitable class of frames) instead
of $\mathcal{F}(M)$. The fibers of such frame bundle should be preserved
by parallel transport so the fundamental horizontal fields can be
restricted on it and we are able to solve the RDE 
\[
d\xi_{t}=H_{\alpha}(\xi_{t})\circ dB_{t}^{\alpha}
\]
on the frame bundle. It will turn out that we are able to establish
the generating PDE of the projection process $X_{t}=\pi(\xi_{t})$
intrinsically on $M$, which does not depend on the initial frame.
Therefore, although as a process the sample paths of $X_{t}$ depends
on the initial frame (this is not surprising since in the Euclidean
case we also don't have a canonical Brownian motion if we do not fix
the frame $\{e_{1},\cdots,e_{d}\}$ in advance), the law of $X_{t}$
will not. In this way we obtain a canonical PDE on $M$ associated
with the original $G$-function, which can be regarded as the generating
PDE governing the law of $X_{t}.$ The process $X_{t}$ can be defined
as a $G$-Brownian motion on $M$ and the generating PDE will play
the role of the canonical Wiener measure (the solution of the martingale
problem for the operator $\frac{1}{2}\Delta_{M}$) on $M$ in a nonlinear
setting.

The construction of such frame bundle for a $G$-function with an
arbitrary invariant group $I(G)$ is not clear to us at the moment.
However, in the case when $I(G)$ is the orthogonal group $O(d)$,
which contains a wide and interesting class of $G$-functions, there
is a very natural frame bundle serving us well for the purpose: the
orthonormal frame bundle $\mathcal{O}(M).$

From now on, let $G$ be given by (\ref{rep of G in section 7}) with
$I(G)=O(d).$ 

The orthonormal frame bundle $\mathcal{O}(M)$ over $M$ is defined
by 
\[
\mathcal{O}(M)=\cup_{x\in M}\mathcal{O}_{x}(M),
\]
where the fibre $\mathcal{O}_{x}(M)$ is the set of orthonormal bases
of $T_{x}M$. Since $M$ is compact, $\mathcal{O}(M)$ is a compact
submanifold of $\mathcal{F}(M)$. Moreover, since the Levi-Civita
connection is compatible with the Riemannian metric $g,$ parallel
transport preserves the fibers of $\mathcal{O}(M).$ Therefore, statements
about $\mathcal{F}(M)$ before on the horizontal aspect can be carried
through in the case of $\mathcal{O}(M)$ directly. In particular,
the fundamental horizontal fields $H_{\alpha}$ can be restricted
to $\mathcal{O}(M).$ The only difference is in the vertical direction:
the fibre becomes orthonormal frames, and the structure group which
acts on fibers becomes the orthogonal group; the dimension in the
vertical direction is reduced to $\frac{d(d-1)}{2}$.

For $\xi\in\mathcal{O}_{x}(M),$ according to Section 6, let $U_{t,\xi}\in\mathcal{O}(M)$
be the unique solution of the following RDE over $[0,1]$: 
\begin{equation}
\begin{cases}
dU_{t,\xi}=H_{\alpha}(U_{t,\xi})\circ dB_{t}^{\alpha},\\
U_{0,\xi}=\xi.
\end{cases}\label{RDE on O(M)}
\end{equation}
Let $X_{t,\xi}=\pi(U_{t,\xi})$ be the projection of $U_{t,\xi}$
onto $M.$ 
\begin{defn}
$X_{t,\xi}$ is called a $G$-Brownian motion on the Riemannian manifold
$M$ with respect to the the initial orthonormal frame $\xi\in\mathcal{O}_{x}(M),$
and $U_{t,\xi}$ is called a horizontal $G$-Brownian motion in $\mathcal{O}(M)$
starting from $\xi.$ 
\end{defn}

For any $\varphi\in C_{Lip}(M)$ (under the Riemannian distance),
define 
\[
u(t,\xi)=\mathbb{E}^{G}[\varphi(X_{t,\xi})],\ (t,\xi)\in[0,1]\times\mathcal{O}(M).
\]
Let $\hat{\varphi}=\varphi\circ\pi$ be the lifting of $\varphi$
to $\mathcal{O}(M).$ It is obvious that 
\[
u(t,\xi)=\mathbb{E}^{G}[\hat{\varphi}(U_{t,\xi})].
\]
By Theorem \ref{Theorem PDE of SDE on manifold}, we know that $u(t,\xi)$
is the unique viscosity solution of the following nonlinear parabolic
PDE:

\begin{equation}
\begin{cases}
\frac{\partial u}{\partial t}-G((\widehat{H_{\alpha}H_{\beta}}u)_{1\leqslant\alpha,\beta\leqslant d})=0,\\
u(0,\xi)=\hat{\varphi}(\xi),
\end{cases}\label{PDE on O(M)}
\end{equation}
on $\mathcal{O}(M).$ 

The following result tells us that the law of $X_{t,\xi}$ depends
only on the initial position $x.$
\begin{prop}
\label{invariance along fibres}If $\xi,\eta\in\mathcal{O}_{x}(M),$
then 
\[
u(t,\xi)=u(t,\eta).
\]
\end{prop}
\begin{proof}
For any fixed orthogonal matrix $Q\in O(d),$ let $\widetilde{B}_{t}=QB_{t},$
which is an orthogonal transformation of the original $G$-Brownian
motion $B_{t},$ and let $W_{t,\zeta}$ be the pathwise solution of
the following RDE over $[0,1]$:

\begin{equation}
\begin{cases}
dW_{t,\zeta}=H_{\alpha}(W_{t,\zeta})\circ d\widetilde{B}_{t}^{\alpha},\\
W_{0,\zeta}=\zeta\in\mathcal{O}(M),
\end{cases}\label{transformed RDE on O(M)}
\end{equation}
on $\mathcal{O}(M).$ If we regard $\widetilde{B}_{t}$ as the solution
of the SDE 
\[
d\widetilde{B}_{t}=Q_{\alpha}dB_{t}^{\alpha}
\]
starting from $0$ with constant coefficients, then the RDE (\ref{transformed RDE on O(M)})
is equivalent to 
\[
\begin{cases}
dW_{t,\zeta}=H_{\beta}(W_{t,\zeta})Q_{\alpha}^{\beta}\circ dB_{t}^{\alpha},\\
W_{0,\zeta}=\zeta,
\end{cases}
\]
in which the generating vector fields are $H_{\beta}Q_{\alpha}^{\beta}.$
Since the invariant group $I(G)$ of $G$ is the orthogonal group,
by Theorem \ref{Theorem PDE of SDE on manifold} we know that the
function 
\[
v(t,\zeta)=\mathbb{E}^{G}[\hat{\varphi}(W_{t,\zeta})],\ (t,\zeta)\in[0,1]\times\mathcal{O}(M)
\]
is the unique viscosity solution of the same PDE (\ref{PDE on O(M)})
on $\mathcal{O}(M).$ Therefore,
\[
u(t,\zeta)=v(t,\zeta),\ \forall(t,\zeta)\in[0,1]\times\mathcal{O}(M).
\]

Now since $\xi,\eta\in\mathcal{O}_{x}(M),$ there exists some $Q\in O(d)$
such that $\xi=\eta Q.$ Define $W_{t,\zeta}$ as before. By the previous
discussion on the relation between different anti-developments, we
know that 
\[
X_{t,\xi}=\pi(U_{t,\xi})=\pi(W_{t,\eta}),\ \forall t\in[0,1].
\]
Therefore, 
\begin{eqnarray*}
u(t,\xi) & = & \mathbb{E}^{G}[\varphi\circ\pi(U_{t,\xi})]\\
 & = & \mathbb{E}^{G}[\varphi\circ\pi(W_{t,\eta})]\\
 & = & v(t,\eta)\\
 & = & u(t,\eta).
\end{eqnarray*}

\end{proof}

From Proposition \ref{invariance along fibres}, we know that $u(t,\xi)$
is invariant along each fibre. Therefore, the law of $X_{t,\xi}$
depends only on the initial position $x\in M$ but not on the initial
frame $\xi.$ We use $u(t,x)$ to denote $u(t,\xi),$ where $x$ is
the base point of $\xi.$ In this situation it is possible to establish
the PDE for $u(t,x)$ intrinsically on $M$ by ``projecting down''
(\ref{PDE on O(M)}), which should become the generating PDE governing
the law of $X_{t,\xi}$.

For any $u\in C^{\infty}(M),$ take an orthonormal frame $\xi=(\xi_{1},\cdots,\xi_{d})\in\mathcal{O}_{x}(M)$,
and consider the quantity 
\[
G((\mbox{Hess}u(\xi_{\alpha},\xi_{\beta}))_{1\leqslant\alpha,\beta\leqslant d}).
\]
Since $I(G)=O(d),$ it is easy to see that the above quantity is independent
of the orthonormal frame $\xi\in\mathcal{O}_{x}(M).$ In other words,
$G$ can be regarded as a functional of the Hessian, and the nonlinear
second order differential operator $G(\mbox{Hess}(\cdot))$ is globally
well-defined on $M.$ 

Now we have the following result.
\begin{thm}
\label{Theorem PDE on M}$u(t,x)$ is the unique viscosity solution
of the following nonlinear parabolic PDE on $M:$

\begin{equation}
\begin{cases}
\frac{\partial u}{\partial t}-G(\mbox{Hess}u)=0,\\
u(0,x)=\varphi(x).
\end{cases}\label{generating PDE on M}
\end{equation}
\end{thm}
\begin{proof}
It suffices to show that: if $f\in C^{\infty}(M),$ and $\hat{f}=f\circ\pi$
is the lifting of $f$ to $\mathcal{O}(M)$, then for any $\xi=(\xi_{1},\cdots,\xi_{d})\in\mathcal{O}_{x}(M),$
\[
\mbox{Hess}f(\xi_{\alpha},\xi_{\beta})(x)=H_{\alpha}H_{\beta}\hat{f}(\xi).
\]
Note that uniqueness follows from the same reason as pointed out in
the proof of Theorem \ref{Theorem PDE of SDE on manifold} by using
results in \cite{azagra2008viscosity}.

In fact, for any $\xi=(\xi_{1},\cdots,\xi_{d})\in\mathcal{O}_{x}(M),$
let $\xi_{t}$ be a horizontal curve through $\xi$ such that $H_{\beta}(\xi)$
is tangent to $\xi_{t}$ at $t=0,$ and let $x_{t}$ be its projection
onto $M.$ It follows that the tangent vector of $x_{t}$ at $t=0$
is $\xi_{\beta},$ and 
\begin{eqnarray*}
H_{\beta}\hat{f}(\xi) & = & \frac{d\hat{f}(\xi_{t})}{dt}|_{t=0}\\
 & = & \frac{df(x_{t})}{dt}|_{t=0}\\
 & = & \langle\xi_{\beta},\nabla f(x)\rangle_{g}.
\end{eqnarray*}
Therefore, if now assume that $\xi_{t}$ is a horizontal curve through
$\xi$ with tangent vector $H_{\alpha}(\xi)$ at $\xi$ and still
$x_{t}=\pi(\xi_{t}),$ then 
\begin{eqnarray*}
H_{\alpha}H_{\beta}\hat{f}(\xi) & = & H_{\alpha}\langle\xi_{\beta},\nabla f(\pi(\xi))\rangle_{g}\\
 & = & \frac{d}{dt}|_{t=0}\langle\xi_{\beta,t},\nabla f(x_{t})\rangle_{g}\\
 & = & \langle\frac{D\xi_{\beta,t}}{dt}|_{t=0},\nabla f(x)\rangle_{g}+\langle\xi_{\beta},\nabla_{\xi_{\alpha}}\nabla f(x)\rangle_{g}\\
 & = & \mbox{Hess}f(\xi_{\alpha},\xi_{\beta})(x),
\end{eqnarray*}
where we've used the fact that $\xi_{\beta,t}$ is parallel along
$x_{t}$.
\end{proof}

Since $X_{t,\xi}$ is the projection of $U_{t,\xi}$ and $U_{t,\xi}$
is the solution of the RDE (\ref{RDE on O(M)}) which is equivalent
to an It$\hat{\mbox{o}}$ type SDE from an extrinsic point of view,
by Theorem \ref{Theorem PDE on M} we can see that as a process on
$M$ the law of the $G$-Brownian motion $X_{t,\xi}$ is characterized
by the nonlinear parabolic PDE (\ref{generating PDE on M}).
\begin{example}
When $G$ is given by a functional of trace, as in Example \ref{example O(d)}
(1), (2), the generating PDE (\ref{generating PDE on M}) takes a
more explicit form in terms of the Laplace-Beltrami operator $\Delta_{M}$
on $M.$ This is due to the fact that 
\[
\Delta_{M}=\mbox{tr}(\mbox{Hess}).
\]
For instance, if $G(A)=\frac{1}{2}\mbox{tr}(A),$ then (\ref{generating PDE on M})
becomes the classical heat equation on $M$:
\[
\frac{\partial u}{\partial t}-\frac{1}{2}\Delta_{M}u=0,
\]
which governs the law of classical Brownian motion on $M$ (see \cite{hsu2002stochastic},
\cite{ikeda1989stochastic}). If $G$ is given by 
\[
G(A)=\frac{1}{2}(\mu(\mbox{tr}A)^{+}-\lambda(\mbox{tr}A)^{-}),
\]
where $0\leqslant\lambda<\mu,$ then (\ref{generating PDE on M})
becomes 
\[
\frac{\partial u}{\partial t}-\frac{1}{2}(\mu(\Delta_{M}u)^{+}-\lambda(\Delta_{M}u)^{-})=0.
\]
It is a generalization of the one-dimensional Barenblatt equation
to higher dimensions in a Riemannian geometric setting.
\end{example}

As pointed out before, as a process the $G$-Brownian motion $X_{t,\xi}$
on $M$ depends on the initial orthonormal frame $\xi$ and hence
there is not a canonical choice of a particular one. However, if we
consider the path space $W(M)=C([0,1];M)$, then for each $x\in M$,
it is possible to define a canonical sublinear expectation $\mathbb{E}_{x}$
on the space $\mathcal{H}(M)$ of functionals on $W(M)$ of the form
\[
f(x_{t_{1}},\cdots,x_{t_{n}}),
\]
where $0\leqslant t_{1}<\cdots<t_{n}\leqslant1$ and $f\in C_{Lip}(M),$
such that under $\mathbb{E}_{x}$ the law of the coordinate process
is characterized by the PDE (\ref{generating PDE on M}) with $\mathbb{E}_{x}[\varphi(x_{0})]=\varphi(x)$
for any $\varphi\in C_{Lip}(M).$

To see this, we will define $\mathbb{E}_{x}$ explicitly. We use $u_{\varphi}(t,x)$
to denote the solution of (\ref{generating PDE on M}), emphasizing
the dependence on $\varphi.$ For a functional of the form $f(x_{t}),$
we simply define 
\[
\mathbb{E}_{x}[f(x_{t})]:=u_{f}(t,x).
\]
For a functional of the form $f(x_{s},x_{t})$, $\mathbb{E}_{x}f(x_{s},x_{t})$
should be defined by $\mathbb{E}^{G}[f(X_{s,\xi},X_{t,\xi})],$ where
$X_{t,\xi}$ is a $G$-Brownian motion on $M$ with respect to an
initial orthonormal frame $\xi\in\mathcal{O}_{x}(M).$ Similar to
the proof of Theorem \ref{Theorem PDE of SDE} we know that 
\begin{eqnarray*}
\mathbb{E}^{G}[f(X_{s,\xi},X_{t,\xi})] & = & \mathbb{E}^{G}[\mathbb{E}^{G}[f(X_{s,\xi},X_{t,\xi})|\Omega_{s}]]\\
 & = & \mathbb{E}^{G}[\mathbb{E}^{G}[f(\pi(U_{s,\xi}),\pi(U_{t,\xi}))|\Omega_{s}]]\\
 & = & \mathbb{E}^{G}[\mathbb{E}^{G}[f(\pi(\eta),X_{t-s,\eta})]|_{\eta=U_{s,\xi}}].
\end{eqnarray*}
But since the law of $X_{t-s,\eta}$ does not depend on the initial
orthonormal frame $\eta,$ we obtain that 
\[
\mathbb{E}^{G}[f(\pi(\eta),X_{t-s,\eta})]|_{\eta=U_{s,\xi}}=u_{f(X_{s,\xi},\cdot)}(t-s,X_{s,\xi}).
\]
Therefore, we define 
\[
\mathbb{E}_{x}[f(x_{s},x_{t})]:=\mathbb{E}^{G}[f(X_{s,\xi},X_{t,\xi})]=u_{g}(s,x),
\]
where 
\[
g(y):=u_{f(y,\cdot)}(t-s,y),\ y\in M.
\]
Inductively, assume that 
\[
u_{f}^{(n)}(t_{1},\cdots,t_{n},x)=\mathbb{E}_{x}[f(x_{t_{1}},\cdots,x_{t_{n}})]
\]
is already defined. For a functional of the form $f(x_{t_{1}},\cdots,x_{t_{n+1}}),$
define 
\[
\mathbb{E}_{x}[f(x_{t_{1}},\cdots,x_{t_{n+1}})]:=u_{g}(t_{1},x),
\]
where 
\[
g(y):=u_{f(y,\cdot,\cdots,\cdot)}^{(n)}(t_{2}-t_{1},\cdots,t_{n+1}-t_{1},y),\ y\in M.
\]
Then $\mathbb{E}_{x}$ is the desired sublinear expectation on $\mathcal{H}(M).$
\begin{rem}
As we've pointed out before, for noncompact Riemannian manifolds,
the RDE (\ref{RDE on O(M)}) may possibly explode at some finite time
and so may the corresponding $G$-Brownian motion as well. An interesting
question is the study of explosion criterion. It might depend on the
curvature and topology of the Riemannian manifold.

On the other hand, for those $G$-functions with the same invariant
group, they may have some special features in common; while for those
with different invariant groups, their structure should be very different.
The study of classification of $G$-functions in terms of the invariant
group is interesting, and it might give us some hints on generalizing
our results to the case when $I(G)\neq O(d).$ We believe that in
some cases it is still possible to construct a proper frame bundle
with structure group $I(G)$ on which we can apply similar techniques
in this section. But in some extreme cases, for instance when $I(G)$
is a finite group as in Example \ref{Example finite group}, it seems
difficult to proceed along this direction unless we have a globally
defined frame field over the Riemannian manifold $M$, which is usually
not true. We probably need some very different methods for those extreme
cases.
\end{rem}

\section*{Acknowledgement}
The authors wish to thank Professor Shige Peng for so many valuable suggestions on the present paper.

\bibliographystyle{plain}
\bibliography{Bibliography}

\end{document}